\documentclass[11pt]{article}
\usepackage{amsmath, amssymb, amsthm}
\usepackage{mathrsfs}
\usepackage[margin=1in]{geometry}
\usepackage{enumitem}
\usepackage{xcolor}
\usepackage{soul}
\usepackage{authblk}
\usepackage[hypertexnames=false]{hyperref}

\numberwithin{equation}{section}
\newtheorem{theorem}{Theorem}[section]
\newtheorem{proposition}[theorem]{Proposition}
\newtheorem{lemma}[theorem]{Lemma}
\newtheorem{cor}[theorem]{Corollary}
\newtheorem{claim}[theorem]{Claim}

\theoremstyle{definition}
\newtheorem{defn}{Definition}[section]
\newtheorem{example}{Example}[section]
\newtheorem*{remark}{Remark}

\newcommand{\RR}{\mathbb{R}}

\DeclareMathOperator{\dist}{dist}

\DeclareMathOperator{\divg}{div}

\colorlet{darkgreen}{green!45!black} 

\title{Singularity formation and self-shrinkers for the parabolic Bernoulli problem}
\author{Hunter Liu\thanks{{University of California, Los Angeles, Department of Mathematics, Los Angeles CA 90095. Email: \texttt{hunterliu@math.ucla.edu}}} \,
Sebastian Munoz\thanks{{University of California, Los Angeles, Department of Mathematics, Los Angeles CA 90095. Email: \texttt{sebastian@math.ucla.edu}}} \, 
Zihui Zhao\thanks{{Johns Hopkins University, Department of Mathematics, Baltimore MD 21218. Email: \texttt{zhaozh@jhu.edu}}}
}

\date{}

\begin{document}
\maketitle

\begin{abstract} 
We study singularity formation for classical solutions of the parabolic Bernoulli free boundary problem and the associated self-similar profiles. Inspired by Ricci and mean curvature flows, we introduce a Type I assumption on the blow-up rate at a singular free boundary point, under which the parabolic rescalings are compact and every tangent flow is a non-trivial self-shrinking solution of the elliptic profile equation. The proof establishes convergence of the positivity sets and recovers the Bernoulli condition in the limit, despite the absence of a minimizing structure. In the subcritical regime, every tangent flow is a double plane. We next classify all radial self-shrinkers. Besides the known ball and exterior profiles, there is a unique annular profile. We obtain dimension-uniform bounds and sharp asymptotics for its inner and outer radii, including a quantitative outward bias of its midpoint. These delicate estimates yield a complete classification of the signs  of the linear spectrum for the annular solution in every dimension. As a consequence, the ball is dynamically stable modulo ambient symmetries, whereas the annulus is dynamically unstable, with genuine unstable modes of angular degrees $0,1,2,3$.  We finally prove, from the spectral nondegeneracy of the compact radial profiles, that whenever one tangent flow is a ball or an annular solution, the tangent flow is unique. One borderline spectral sign in dimensions $2\leq n \leq 4000$ is verified by a rigorous computer-assisted argument using interval arithmetic.

\end{abstract}



\setcounter{tocdepth}{1} % sections only in the table of contents
\tableofcontents

\section{Introduction}

We study singularity formation for classical solutions of the parabolic Bernoulli problem
\begin{equation}\tag{B} \label{eq:pBern}
	\begin{cases}
u_t - \Delta u = 0, & \text{ in } \{u>0\}, \\
|\nabla u(t, \cdot)| = 1, & \text{on } \partial\{u(t, \cdot) > 0\}.
\end{cases}
\end{equation}
The problem arises in combustion theory as a model for the propagation of flames with high activation energy, the free boundary $\partial\{u>0\}$ representing the flame front, see \cite{BL} for the physical background. 
While the stationary version of the problem, known as the one-phase or Bernoulli free boundary problem, has been studied extensively since the seminal work of Alt and Caffarelli \cite{AC}, much less is known about the time-dependent problem \eqref{eq:pBern}, partly because it has no underlying minimization principle. The main feature and difficulty of free boundary problems is that the regularity of the solution and that of the free boundary are mutually dependent. 

For the initial value problem, uniqueness under geometric assumptions on the initial data was proved in \cite{Kim03,LVW01,Pet02}, whereas nonuniqueness can occur for singular-perturbation limits \cite{PY08}. 
Weak formulations of \eqref{eq:pBern} were developed in \cite{CV,  Weiss}, where solutions are constructed as singular limits of regularized semilinear equations.
The weak solutions of \cite{CV} have receding supports, which excludes basic examples such as traveling wave solutions (see Example \ref{eg:exp}). 
Weiss introduced a domain-variation formulation in \cite{Weiss99} for solutions that exist globally in time and proved, under a scale-invariant nondegeneracy assumption and mild regularity and growth assumptions, that blow-ups within this class are backward self-similar variational solutions. 
Later, Weiss \cite{Weiss} introduced a more flexible class of solution pairs $(u,\chi)$, which is closed under the blow-up process and hence amenable to geometric measure theory tools. Using these tools, Weiss classified the behavior of almost every free boundary point at almost every time, therefore showing that singularities of the free boundary are rare.  We discuss these formulations in more detail in Subsection \ref{sec:rmk-hp}.

However, singular points are precisely where interesting changes of the free boundary -- extinction, collision, or a change of topology -- take place. In this paper, instead of working with the most general class of solutions, we take a more direct approach and analyze how such singularities form, for solutions that are classical up to the singular time.

Throughout the paper, $u$ denotes a classical solution of \eqref{eq:pBern} on $\RR^n\times[-T,0)$ whose spatial gradient satisfies $|\nabla u|\leq L$\footnote{In Subsection \ref{sec:rmk-hp}, we comment on how the assumptions can be weakened.}, and whose first singular time is $t=0$: we say that a space-time point $(t_0=0,x_0)$ is \emph{singular} if $u$ stops being a classical solution near $x_0$ as $t\to 0-$. By the interior regularity theory for the heat equation, singularities can only occur at the free boundary. 

Formally speaking, by the parabolic rescaling of the problem, self-similar solutions or \emph{self-shrinkers} of the form 
\begin{equation}\label{eq:selfshrinker}
	u(t,x)=\sqrt{-t}\,\phi\!\left(\frac{x}{\sqrt{-t}}\right),\qquad t<0,
\end{equation}
are natural and archetypal examples of such singular behavior, where the profile $\phi$ solves the elliptic equation
\begin{equation}\label{eq:profile}
\left\{\begin{array}{ll}
	\Delta\phi - \tfrac{1}{2}\,y \cdot \nabla\phi + \tfrac{1}{2}\,\phi = 0 & \text{in } \Omega:= \{\phi>0\}, \\
	|\nabla\phi| = 1 & \text{on } \partial\Omega.
\end{array}   \right.\tag{P}
\end{equation}
The positivity set of such a solution is $\{u(t,\cdot)>0\}=\sqrt{-t}\cdot\Omega$, and it becomes extinct as $t\to 0-$ when $\Omega$ is bounded, hence the name \emph{self-shrinkers}. Note that such model of singularity satisfies $|D^2 u| \sim 1/\sqrt{-t}$, namely the blowup rate of the Hessian of the solution is given by $1/\sqrt{-t}$. This motivates the following definition:
\begin{defn}
	We say that $(t_0,x_0)$ is a \emph{Type I singular point} if it is a singular point and there exists $r>0$ such that
\begin{equation}\label{eq:typeI}
    C_s(r):= \sup_{Q^{-}_{r}(t_0,x_0)\cap \{u>0\}} \sqrt{-t}|D^2 u(t,x)| < +\infty,
\end{equation}
where $Q_r^-(t_0,x_0):=[t_0 -r^2,t_0)\times B_r(x_0)$ is an $r$-parabolic cylinder centered at $(t_0,x_0)$ in space-time.
\end{defn}

In Theorems \ref{thm:blowup} and \ref{thm:blowup-twophase}, we will show that the Type I condition identifies the minimal blow-up rate at which nontrivial self-similar behavior can arise: at the scaling $|D^2u|\sim 1/\sqrt{-t}$, the parabolic rescalings of $u$ remain compact without degenerating, while below this rate every tangent flow is necessarily a double plane.
We also remark that in the study of Ricci and mean curvature flows, singularities are classified according to the rate at which the curvature blows up; for the mean curvature flow (Ricci flow resp.), blow-ups at the minimal such rate -- the Type I singularities -- are modeled by self-shrinkers (Ricci solitons resp.), by Huisken's monotonicity formula \cite{Huisken}.
It remains open whether genuinely non-Type I singularities can develop from classical solutions. A possible mechanism for such behavior is suggested in Example \ref{eg:non-type-I}. \footnote{Another possible mechanism is by pinching a radially-symmetric solution supported on a ball along one axis, resembling the degenerate neck pinch or peanut solution of the mean curvature flow studied in \cite{AAG, AV}; but such a singularity is expected to be highly unstable, see \cite{ADS}.} 

To state our first result, assume after a translation $(0,0)$ is a singular point, and for $\lambda>0$ consider the parabolic dilations
\[
	u_\lambda(t,x) := \frac{u(\lambda^2 t, \lambda x)}{\lambda},
\]
which are again classical solutions of \eqref{eq:pBern}. Our first main result shows that at a Type I singular point, the dilations converge, along subsequences, to self-shrinkers.

\begin{theorem}[Existence of tangent flows]\label{thm:blowup}
    Suppose $(0,0)$ is a Type I singular point for $u$.
    For any sequence $\lambda_j \to 0$, modulo passing to subsequences, there exists a non-trivial solution $\phi$ of \eqref{eq:profile} such that
    \[
        u_{\lambda_j}(t,x) \longrightarrow \sqrt{-t}\,\phi\!\left(\tfrac{x}{\sqrt{-t}}\right)
        \qquad \text{locally uniformly in } (-\infty,0)\times\RR^n.
    \]
    Moreover, the limit function $\phi$ satisfies $\|D^2\phi\|_{L^\infty(\{\phi>0\})} \leq C_s$, $\partial \{\phi>0\}$ is locally the union of finitely many $C^{1,1}$ hypersurfaces, and
    \begin{equation}\label{eq:density-tangent}
        \int_{\RR^n} \chi_{\{\phi>0\}} \cdot \frac{1}{(4\pi)^{\frac{n}{2}}} e^{-\frac{|y|^2}{4}} \, dy = \lim_{r\to 0} \theta(r) =: \theta(0+) 
    \end{equation}
    takes values in $(1/2, 1]$, where $\theta(r)$ is the time-averaged free boundary energy of $u$ at $(0,0)$ defined in \eqref{def:theta}.
\end{theorem}

We call the self-similar solutions $\sqrt{-t}\,\phi\left(\tfrac{x}{\sqrt{-t}}\right)$ arising in Theorem \ref{thm:blowup} \emph{tangent flows of $u$ at $(0,0)$}; a priori, different subsequences may produce different tangent flows. The main difficulty in the theorem is that the class of classical solutions is not closed under parabolic blow-up, and \eqref{eq:pBern} has no minimality property that could be passed to the limit: the dilations $u_{\lambda_j}$ could converge to zero, or their positivity sets could fail to converge to that of the limit, in which case the limit would not attain the free boundary condition. The Type I bound, combined with the Bernoulli condition, rules this out; we sketch the ideas of the proof in Section \ref{ssec:main-ideas} below.

Our second result shows that below the critical rate \eqref{eq:typeI}, the behavior of the solution is completely rigid.

\begin{theorem}[Rigidity below the critical rate]\label{thm:blowup-twophase}
    Under the same assumption as in Theorem \ref{thm:blowup}, assume in addition the blow-up rate in \eqref{eq:typeI} satisfies $C_s(r) \to 0$ as $r\to 0$. Then every tangent flow at $(0,0)$ is of the form $|\langle x, e \rangle|$ for some $e\in \mathbb{S}^{n-1}$. Moreover, if the sequence $\lambda_j\to0$ produces the tangent flow $|\langle x, e \rangle|$, then for every $j$ sufficiently large and every $t\in (-2\lambda_j^2, -\lambda_j^2/2)$, the positive set $\{u(t,\cdot)>0\} \cap B_{\lambda_j}(0)$ consists of two connected components, lying above and below the graphs of two $C^{1,1}$ functions $\eta_t^\pm$ with sufficiently small $C^{1}$ norm over the hyperplane $H= \{x\in \mathbb{R}^n: \langle x, e \rangle = 0\}$. If, moreover, the blow-up rate $C_s(r)$ satisfies the following Dini assumption
    \begin{equation}\label{eq:dini-Cs}
        \int_0^{*}\frac{C_s(r)}{r}\,dr<+\infty,
    \end{equation}
    then $e$ is independent of the subsequence and $u_\lambda \to |\langle x, e \rangle|$ locally uniformly in $(-\infty,0)\times\RR^n$ as $\lambda\to0$.
\end{theorem}

Consequently, a free boundary point at which $C_s(r)\to0$ as $r\to0$ is either a regular point, or a point where two pieces of the positivity set, with nearly flat boundaries, collide (see Example \ref{eg:exp}). The Dini condition \eqref{eq:dini-Cs} is used only to fix the direction $e$: without it, the collision direction could conceivably rotate slowly between the admissible time windows.

\medskip

Since singularity formation at Type I singular points is modeled by self-shrinkers, it is desirable to classify the solutions of the profile equation \eqref{eq:profile} that can arise. The second group of results of this paper concerns \emph{radial} profiles, $\phi(y) = \phi(|y|)$---by the usual abuse of notation, the same symbol denotes the profile and its radial representative---for which \eqref{eq:profile} reduces to the ODE
\begin{equation}\label{eq:ode}
    \phi''(r) + \left(\frac{n-1}{r} - \frac{r}{2}\right) \phi'(r) + \frac{1}{2}\,\phi(r) = 0.\tag{RP}
\end{equation}
We classify the radial profiles completely. 
%In the statement, $s_*<s_0$ denote the unique positive zeros of the Kummer functions $U(-\tfrac12,\tfrac n2,\cdot)$ and $M(-\tfrac12,\tfrac n2,\cdot)$, respectively (see Lemma ref{lem:MU}).

\begin{theorem}[Classification of radial profiles]\label{thm:radial}
    In any dimension $n\geq 2$, every nonzero radial solution to~\eqref{eq:profile} with connected positivity set $\Omega$ is one of the following, with dimensional constants $r_- < r_* < r_0 < r_+$:
    \begin{enumerate}[label=(\roman*), topsep=0pt]
        \item $\phi_B$, the unique radial profile supported on a ball centered at the origin with radius $r_0$,
        \item\label{it:annular} $\phi_A$, the unique radial profile supported on an annulus $A=\{r_-<|x|<r_+ \}$;
        \item the radial profile $\phi _r$ supported on the complement of a ball of radius $r$ for any $r \geq r_*$. 
        Moreover, near $\infty$ we have
        \[ \phi_r(x) \sim \left\{\begin{array}{ll}
            |x| & \text{ if } r = r_*, \\
            e ^{\left\lvert x \right\rvert^2/4} \cdot \left\lvert x \right\rvert ^{-(n+1)} & \text{ if } r > r_*. 
        \end{array} \right. \]
    \end{enumerate}
    In dimension $n=1$, the annular solution in \ref{it:annular} does not exist, and the radial self-shrinker with compact positivity set is unique.\footnote{In fact, in 1D the solution to \eqref{eq:profile} with compact support has to be even/radial, and thus unique, by personal communication with Nikola Kamburov and Dennis Kriventsov. See also \cite[Lemma 7.5]{radial}.}
\end{theorem}

The existence and uniqueness of the radial profile supported on a ball were established in \cite[Proposition 1.1]{CV}, and its stability was proved in \cite[Section 4.2]{stability}. The construction of an annular profile also appears in \cite[Section 5]{radial}; its uniqueness is new. We remark that the annular solutions in dimensions $n\geq 2$ bear a loose analogy with Angenent's torus $\mathbb{S}^1\times \mathbb{S}^{n-1}$ ($n\geq 2$), a closed embedded rotationally symmetric self-shrinker of the mean curvature flow in $\RR^{n+1}$ \cite{Ang} (see also \cite{DN}). In \cite{KM}, embedded rotationally symmetric self-shrinkers of the mean curvature flow are classified as the plane, the round sphere, the round cylinder, or a smooth embedding of $\mathbb{S}^1\times \mathbb{S}^{n-1}$; however, it remains an open question whether the compact genus-one self-shrinker is unique, even under the additional assumption of rotational symmetry. By contrast, in the parabolic Bernoulli setting the radial classification, including the uniqueness of the annular profile, is complete.

The radii $r_*$ and $r_0$ are given explicitly as the roots of some special functions. In comparison, the radii $r_\pm$ of the annular solution are determined only implicitly by the construction in Theorem \ref{thm:radial}. Our next result locates them sharply, in every dimension. The remaining results of the paper depend crucially on these delicate estimates, which we believe to be of independent interest as well. 

\begin{theorem}[Sharp geometry of the annular profile]\label{thm:annulus-geometry}
    Let $[r_-,r_+]$ be the support of the annular profile $\phi_A$.
    \begin{enumerate}[label=(\roman*), topsep=0pt]
        \item\label{it:geom-loc} \emph{Location of the radii:} For $n\geq3$,
        \[
            2\left(\sqrt{n-1}-\kappa\right)^2<r_-^2<2\left(\frac{n-1}{\sqrt{n-1}+\kappa}\right)^2,
            \qquad
            2\left(\sqrt{n-1}+\kappa\right)^2<r_+^2<2\left(\frac{n-1}{\sqrt{n-1}-\kappa}\right)^2,
        \]
        where $\kappa=1.1623\ldots$ is the first positive zero of the even solution of the Hermite-type equation $Y''-2xY'+Y=0$. In particular,
        \begin{equation}\label{eq:fine-radius-estimate}
            r_\pm=\sqrt{2(n-1)}\pm\sqrt2\,\kappa+O(n^{-1/2}).
        \end{equation}
        When $n=2$, one still has $r_-^2<2(1+\kappa)^{-2}$ and $r_+^2>2(1+\kappa)^2$.
        \item\label{it:geom-width} \emph{Width and product:} For every $n\geq2$,
        \[
            r_+-r_-<\sqrt{14} \qquad \text{and} \qquad r_-\,r_+<2(n-1).
        \]
        \item\label{it:geom-bias} \emph{Outward bias:} For every $n\geq2$, with $\delta:=\tfrac12(r_+-r_-)$ the half-width of the annulus,
        \[
            \frac{\int_0^{\delta}s^4e^{-s^2/2}\,ds}{\int_0^{\delta}s^2e^{-s^2/2}\,ds}
            <\left(\frac{r_-+r_+}{2}\right)^2-2(n-1)
            <\delta^2.
        \]
    \end{enumerate}
\end{theorem}

Theorem \ref{thm:annulus-geometry} shows that, in every dimension, the annulus concentrates around the sphere of radius $\sqrt{2(n-1)}$ --- the zero of the drift coefficient of \eqref{eq:ode} --- and that its width converges to the dimension-free value $2\sqrt{2}\,\kappa$. The constant $\kappa$ arises from a limiting problem: recentering \eqref{eq:ode} at $\sqrt{2(n-1)}$ and letting $n\to\infty$ yields the Hermite-type equation in \ref{it:geom-loc}. Part \ref{it:geom-bias} captures a finer feature which \ref{it:geom-loc} alone cannot detect: the midpoint of the annulus lies strictly beyond $\sqrt{2(n-1)}$, by an excess bounded above and below in terms of the width alone. This outward bias is precisely the quantity that decides the sign of the borderline eigenvalue in the spectral analysis below, see Theorem \ref{prop:spectrum-annular}.

Theorem \ref{thm:annulus-geometry} is proved in Appendix \ref{sec:radius-est}: part \ref{it:geom-loc} is Proposition \ref{prop:annulus-localization}, and part \ref{it:geom-width} combines Proposition \ref{prop:annulus-width} and Lemma \ref{lem:annular-phase}. Part \ref{it:geom-bias} is Lemma \ref{lem:annular-midpoint}, stated there in rescaled variables.
\medskip

For the stability analysis, self-shrinkers are best viewed as equilibria of a rescaled flow. Define the parabolic rescaling
\begin{equation}\label{def:prescaling}
	v(s,y) := \frac{u(t, y\sqrt{-t})}{\sqrt{-t}}, \qquad \text{where } s=-\log(-t) \in [-\log T,\infty),
\end{equation}
which satisfies the rescaled equation
\begin{equation}\label{eq:respar}
    \left\{\begin{array}{ll}
        v_s = \Delta v -\frac12 y\cdot \nabla v + \frac12 v & \text{ in } \{ v>0\}, \\
        |\nabla v(s, \cdot)| = 1 & \text{on } \partial\{v(s, \cdot)>0\}.
    \end{array} \right.\tag{E}
\end{equation}
The self-similar solutions \eqref{eq:selfshrinker} correspond exactly to the equilibria of \eqref{eq:respar}, and in these variables Theorem \ref{thm:blowup} states that the rescaled flow $v(s,\cdot)$ converges to equilibria along subsequences $s_j=-2\log\lambda_j\to\infty$.

We now turn to the stability of the compact profiles, which determines their relevance for the dynamics: one expects that only stable profiles should arise from generic solutions, in the spirit of the program of Colding and Minicozzi for the mean curvature flow \cite{CM-stable}. In \cite{stability}, Brauner, Hulshof and Lunardi developed a general framework that reduces the nonlinear stability of an equilibrium $\phi$ of \eqref{eq:respar} with bounded, smooth positivity set $\Omega$ to the eigenvalue problem
\begin{equation}\label{eq:spectrum-full}
    \left\{\begin{array}{ll}
       \Delta v - \frac{1}{2} y \cdot \nabla v + \frac{1}{2} v = \lambda v,  & \text{ in } \Omega,  \\[5pt]
       \frac{\partial v}{\partial n} -
    	\frac{\partial^2 \phi}{\partial n^2} v = 0,  & \text{ on } \partial \Omega,
    \end{array} \right.
\end{equation}
whose boundary condition is inherited from the Bernoulli condition. For the annular profile, separating variables into spherical harmonics of degree $\ell$ reduces \eqref{eq:spectrum-full} to a family of Sturm--Liouville problems on $(r_-,r_+)$, see \eqref{eq:spectrum}; for each $\ell\geq0$, the resulting eigenvalues $\lambda_{\ell,1}>\lambda_{\ell,2}>\cdots\to-\infty$ are simple. Our next main result determines the sign of every eigenvalue, in every dimension.

\begin{theorem}[The spectrum of the annular profile]\label{prop:spectrum-annular}
	The spectrum of the linearized problem \eqref{eq:spectrum-full} at the annular self-shrinker $\phi_A$ is classified as follows:
\begin{equation}\label{table:spectrum}
	\begin{array}{c|c|c|c|c|c}
& \ell=0 & \ell=1 & \ell=2 &\ell=3 & \ell\geq4 \\
\hline
k=1 & >1 & 1 & >0 & >0 &  <0 \\
\hline
k=2 & 1 & \frac12 & <0 & <0 &<0 \\
\hline
k\geq 3 & <0 & <0 & <0 & <0 &<0
\end{array}
\end{equation}
\end{theorem}

\begin{cor}[Stability of the ball, instability of the annulus]\label{cor:stability}
    The ball profile $\phi_B$ is dynamically stable, and the annular profile $\phi_A$ is dynamically unstable.
\end{cor}

Here dynamical stability is understood modulo the symmetries of the problem. The ball is not a stable equilibrium of \eqref{eq:respar}: it possesses the nonnegative modes $\lambda_{0,1}=1$ and $\lambda_{1,1}=\frac12$, but these are generated by translations of the singular point in time and space and do not change the self-similar solution (rotations fix the radial profiles altogether and contribute no modes), while all the remaining eigenvalues are negative (the table for $\phi_B$ is \eqref{table:spectrum} with the row $k=1$ removed). The annulus, by contrast, carries unstable modes with angular frequencies $\ell=0,1,2,3$, and every mode with $\ell\geq4$ is stable. The first eigenfunction for $\ell=0$ is radial and widens or narrows the annulus; this instability is illustrated by the solutions of \cite[Theorem 8.2]{radial}, which are radially symmetric, have initial data supported in an annulus, and become extinct on a sphere in finite time (see Example \ref{eg:extinct-on-sphere}). The mode $\lambda_{1,1}=1$ widens the annulus on one side and narrows it on the opposite side. The deduction of Corollary \ref{cor:stability} from the table, via the invariant manifold theory of \cite{stability}, is discussed in Section \ref{sec:spectrum-applications}. The stability of the ball and the instability of the annulus parallel the situation for compact self-shrinkers of the mean curvature flow, where the round sphere is the only compact self-shrinker that is entropy stable \cite{CM-stable}.

The proof of Theorem \ref{prop:spectrum-annular} occupies Section \ref{sec:instability} and depends essentially on Theorem \ref{thm:annulus-geometry}: the table reduces to three critical signs, each of which is equivalent to a quantitative statement about the location of the annulus (see the discussion in Section \ref{ssec:main-ideas}). In dimensions $2\leq n\leq4000$, the borderline mode $\lambda_{3,1}$ is treated by a computer-assisted proof---the only computer-assisted ingredient of the paper---carried out in rigorous interval arithmetic \cite{arb}; the verifier and its certified output accompany the paper as ancillary files.
\medskip

Our final result ties the spectral analysis back to the singularity formation. The limit produced by Theorem \ref{thm:blowup} may a priori depends on the subsequence, and the uniqueness of blow-up limits is a fundamental question in geometric measure theory; see \cite{AC, DS, ESV20, DSS-AltP, DPSV-twophase, AKN-ACF, ERZ25} etc. for uniqueness results for the elliptic Bernoulli problems. For the mean curvature flow, uniqueness is usually proven using {\L}ojasiewicz--Simon inequalities (see \cite{Schulze, CM-uniqueness, CM-highercodim, Chodosh-Schulze, Lee-Zhao, Zhu} and \cite{Bamler-Lai} etc.). 
We prove the corresponding statement for the two compact radial profiles.

\begin{theorem}[Uniqueness for ball and annular tangent flows]\label{thm:compact-radial-tangent-uniqueness}
    Under the hypotheses of Theorem~\ref{thm:blowup}, suppose that a profile $\phi_*\in\{\phi_B,\phi_A\}$ is the profile of one tangent flow of $u$ at $(0,0)$. Then
    \[
        v(s,\cdot)\longrightarrow \phi_*
        \qquad\text{locally uniformly in $\RR^n$ as $s\to\infty$};
    \]
    equivalently, $u_\lambda(t,x)\to\sqrt{-t}\,\phi_*\!\left(\tfrac{x}{\sqrt{-t}}\right)$ locally uniformly as $\lambda\to0$. In particular, the tangent flow at $(0,0)$ is unique.
\end{theorem}

Rather than a {\L}ojasiewicz-type inequality, the proof relies on the spectral information of Theorem \ref{prop:spectrum-annular}: the linearized operator at $\phi_B$ or $\phi_A$ has trivial kernel, and an inverse function theorem argument upgrades this to the isolation of $\phi_B$ and $\phi_A$ among all tangent profiles at the singular point. The connectedness of the set of subsequential limits then yields the convergence of the full rescaled flow.

An intriguing open problem concerning \eqref{eq:profile} is whether the ball profile $\phi_B$ is the only self-shrinker supported on a compact convex domain. The corresponding rigidity statement for mean curvature flow already holds under the weaker assumption of mean convexity: every compact mean-convex self-shrinker is a round sphere \cite{Huisken}. A straightforward modification of the isolation argument above shows that $\phi_B$ is isolated, in the $C^{2,\alpha}$ topology, among all self-shrinkers: a convex self-shrinker other than the ball, if one exists, cannot arise as a small deformation of the ball. Likewise, any non-radial self-shrinker with annular topology must be separated from $\phi_A$. The global convex rigidity problem remains difficult: neither the usual moving-plane method, obstructed by the drift term in \eqref{eq:profile}, nor the curvature-quotient argument of \cite{Huisken} appears to carry over directly.

\subsection{Main ideas of the proofs}\label{ssec:main-ideas}

Section \ref{sec:prelim} collects the preliminary estimates and the model examples, in particular Examples \ref{eg:exp} and \ref{eg:extinct-on-sphere}. Theorems \ref{thm:blowup} and \ref{thm:blowup-twophase} are proved in Section \ref{sec:blowup}. Weiss's monotonicity formula \cite{Weiss} provides dissipation and compactness along the rescaled flow \eqref{eq:respar}. The key point is that the equation and the free boundary condition pass to the limit only where the positivity sets converge. This is where the Type I bound enters: combined with the Bernoulli condition, it yields, among other things, uniform non-degeneracy at a full-measure set of times,
%uniform control of tubular neighborhoods of the free boundary, a uniform upper bound on $\|\partial_\sigma w_j\|_{L^\infty}$, which is used in \eqref{eq:goodtime-lp} to pass the evolution PDE to the stationary PDE
from which the convergence of the positivity sets and of the free boundaries follows. When $C_s(r)\to0$, the limiting profile has vanishing Hessian, hence it is linear on each component of its positivity set. The half-plane is excluded by the $\epsilon$-regularity result of Andersson and Weiss \cite{AW}, which leaves $\phi=|\langle y,e\rangle|$, and under the Dini condition \eqref{eq:dini-Cs} the direction $e$ stabilizes because its drift is integrable in time.

Theorem \ref{thm:radial} is proved in Section \ref{sec:radial}. In radial coordinates, \eqref{eq:profile} becomes a boundary value problem for Kummer's equation, and the annular solutions form a one-parameter shooting family. Uniqueness reduces to the strict monotonicity of the shooting map, which we prove by a phase-plane argument.

The sharp geometric estimates in Theorem \(\ref{thm:annulus-geometry}\) are established in Appendix \(\ref{sec:radius-est}\) and constitute one of the main technical parts of the paper. The product bound in \ref{it:geom-width} is proved first (Lemma \ref{lem:annular-phase}), in the same phase-plane variables as the uniqueness in Theorem \ref{thm:radial} but by a different argument: the Bernoulli conditions force the oscillator energy to return to its initial value across the annulus, and a symmetrization argument shows that this balance fails unless $r_-\,r_+<2(n-1)$. For part \ref{it:geom-loc}, we then compare the annular profile $\phi_A$, recentered at $\sqrt{2(n-1)}$, with the even solution of the limiting Hermite-type equation \eqref{eq:hermite-kappa} through a Wronskian argument (Lemma \ref{lem:annular-barrier}). The comparison yields the lower bounds for both radii, and the product bound converts them into upper bounds, enclosing each radius in a window of length $O(n^{-1/2})$. For the width bound in \ref{it:geom-width} (Proposition \ref{prop:annulus-width}), we observe that the profile is the first Dirichlet eigenfunction, with eigenvalue one, of a weighted Sturm--Liouville operator on $(r_-,r_+)$; if the annulus were wider than $\sqrt{14}$, explicit trial functions on a suitable subinterval would force this eigenvalue below one. Since the bound is dimension-free, the comparison must close in every dimension down to the worst case $n=2$, where the margin is extremely narrow and the final inequality reads $999/1000<1$.

The deepest of the three geometric estimates is the outward bias \ref{it:geom-bias}, proven in Lemma \ref{lem:annular-midpoint}. The idea is to exploit the mismatch between symmetric data and an asymmetric equation: the boundary conditions --- $\phi_A$ vanishes at $r_\pm$, with $\phi_A'(r_-)=1=-\phi_A'(r_+)$ --- are those of a profile even about the midpoint of the annulus, but the drift coefficient of \eqref{eq:ode} breaks the reflection symmetry. The odd part of the profile measures this asymmetry, and it satisfies the equation \eqref{eq:annular-odd-part}, whose forcing is a positive weight times the difference between the squared distance to the midpoint and the excess $\bigl(\tfrac{r_-+r_+}{2}\bigr)^2-2(n-1)$. Testing against a solution of the corresponding homogeneous equation yields a compatibility identity (see \eqref{eq:annular-midpoint}), which represents the excess as the average of the squared distance to the midpoint against a positive weight; the upper bound in \ref{it:geom-bias} follows at once. The lower bound lies deeper: it is the same average computed against the explicit weight $s^2e^{-s^2/2}$, and the comparison of the two averages combines a correlation inequality with the convexity of the curvature of the profile, itself proved by a double application of the maximum principle.

The spectral analysis is carried out in Section \ref{sec:instability}. Three eigenvalues are explicit (Lemma \ref{lem:explicit-modes}): $\lambda_{0,2}=1$ and $\lambda_{1,2}=\frac12$ arise by differentiating the self-similar solution in time and space, and $\lambda_{1,1}=1$ has the explicit eigenfunction $r^{1-n}e^{r^2/4}$ in the $\ell=1$ sector. Since the eigenvalues are nonincreasing in $\ell$, and every eigenvalue with $k\geq3$ is negative by a domain monotonicity argument (Lemma \ref{lem:mode-structure}), the table \eqref{table:spectrum} reduces to three assertions:
\[
    \lambda_{2,2}<0, \qquad \lambda_{4,1}<0, \qquad \lambda_{3,1}>0.
\]
Each of the three is proved by pairing an explicit construction with one of the geometric estimates of Theorem \ref{thm:annulus-geometry}: the sign of each critical eigenvalue is, in effect, a statement about where the annulus sits.

For $\lambda_{2,2}<0$ (Lemma \ref{lem:lambda22}), we construct an explicit solution of the zero-eigenvalue equation in the $\ell=2$ sector, namely $\phi+\frac{2n}{r}\phi'$, and feed it into an oscillation criterion based on the ground-state representation (Lemma \ref{lem:oscillation}, cf. \cite{PT06}). The required boundary signs hold precisely because the annulus contains the sphere of radius $\sqrt{2n}$ (Proposition \ref{prop:annulus-containment}). For $\lambda_{4,1}<0$ (Proposition \ref{prop:lambda41}), we exhibit an explicit supersolution whose verification reduces to a single polynomial inequality, valid by the product and width bounds of Theorem \ref{thm:annulus-geometry} \ref{it:geom-width}. The argument is elementary, but the margins are tight, and the width bound enters at nearly its exact value.

The positivity of $\lambda_{3,1}$ (Proposition \ref{prop:lambda31}) is the borderline case, and the most demanding. On the annulus, the zeroth-order coefficient of the Rayleigh quotient for the $\ell$-th problem equals $\frac{3-\ell}{2}+O(n^{-1/2})$: the formal limiting problem for $\ell=3$ has largest eigenvalue exactly zero, and the sign of $\lambda_{3,1}$ is determined by lower-order information which coarse bounds on the radii cannot detect. A factorization of the $\ell=3$ operator reduces the positivity of $\lambda_{3,1}$ to the moment inequality
\begin{equation}\label{eq:l3-moment-intro}
    \int_{r_-}^{r_+}\bigl(r^2-2(n+2)\bigr)\,r^{n+3}e^{-r^2/4}\,dr>0
\end{equation}
(Lemma \ref{lem:l3-moment}). At leading order the integral vanishes by symmetry about the midpoint of the annulus. The sign emerges at the next order in $n^{-1/2}$, where the outward bias \ref{it:geom-bias} enters (see Section \ref{ssec:l3}). For $n\geq4001$ we verify the inequality analytically, using the fact that the moment is increasing in the midpoint and in the width (Lemma \ref{lem:l3-monotone}), together with part \ref{it:geom-loc}, rational bounds for $\kappa$ (Lemma \ref{lem:kappa-bounds}), and the lower bound in \ref{it:geom-bias}.  For $2\leq n\leq4000$ the margins are too thin for asymptotics, and \eqref{eq:l3-moment-intro} is instead verified by computer, in rigorous interval arithmetic (Lemma \ref{lem:l3-finite}).

Theorem \ref{thm:compact-radial-tangent-uniqueness} is proved in Section \ref{sec:spectrum-applications}. The key step is a local isolation statement (Lemma \ref{lem:compact-radial-local-isolation}): a tangent profile close to $\phi_*\in\{\phi_B,\phi_A\}$ coincides with it. Convergence of tangent profiles is upgraded to $C^{2,\alpha}$ by a partial hodograph transformation, which converts the Bernoulli condition into an oblique derivative problem \cite{LT}. Since the linearized operator has trivial kernel by Theorem \ref{prop:spectrum-annular}, the inverse function theorem forces nearby tangent profiles to agree with $\phi_*$. The set of subsequential limits of the rescaled flow is compact and connected, so isolation implies uniqueness.

\section{Preliminaries and examples}\label{sec:prelim}

\subsection{H\"{o}lder estimate in time}

The next Lemma shows that uniform Lipschitz bound in space implies $\frac12$-H\"older bound in time. In particular, a classical solution to the parabolic Bernoulli problem which attains the initial data $u(0,\cdot) \in C_0^{0,1}(\RR^n)$ and has uniform Lipschitz bound $\|\nabla u\|_{L^\infty(\RR^n)} \leq L$ belongs to the class of weak solutions studied in \cite{Weiss}, see Definition 6.1. 
The class of weak solutions as in Definition 6.1 is closed under the blow-up process. In general, classical solutions may not be closed after taking the parabolic blow up, but we show that assuming a natural control on the blow-up rate near singularity (see \eqref{eq:typeI}) the blow-up process is well understood, see Theorem \ref{thm:blowup}. 

\begin{lemma}\label{lem:holder}  
There exists $C=C(n)>0$ such that:
\begin{enumerate}[label=\textup{(\roman*)}]
\item For every $x_0\in\RR^n$ and $t_0\in(-T,0)$ with $t_0-r^2\geq -T$,
\begin{equation*}
\int_{t_0-r^2}^{t_0}\int_{B_r(x_0)} u_t^2\,dx\,dt\leq C(1+L^2)\,r^n.
\end{equation*}
\item $u$ is $\frac12$-H\"older continuous in time: for every $x\in\RR^n$ and $t,t+h\in[-T,0)$,
\begin{equation}\label{eq:Holderintime}
    |u(x,t+h)-u(x,t)|\leq C(1+L)\sqrt{|h|}.
\end{equation}
\end{enumerate}
\end{lemma}

\begin{proof}
Fix $(x_0,t_0)\in\RR^n\times(-T,0)$ and $r>0$. Let $\eta\in C_c^\infty(B_{2r}(x_0))$ be a smooth cutoff satisfying $0\leq\eta\leq 1$, $\eta\equiv 1$ on $B_r(x_0)$, and $|\nabla\eta|\leq C/r$.

Multiplying the equation $u_t=\Delta u$ by $u_t\eta^2$ and integrating by parts in space gives
    \begin{multline*}
    \int_{\RR^n} u_t^2 \eta^2\,dx = \int_{\RR^n} u_t\Delta u \,\eta^2 \,dx \\
    =\int_{\RR^n} \left(-\frac12 \partial_t|\nabla u|^2\eta^2 - 2 u_t \nabla u \cdot \nabla \eta\, \eta\right)dx + \int_{\partial \{u>0\}}u_t\, \nabla u\cdot \nu\, \eta^2\, dS.
    \end{multline*}
    By the Reynolds transport formula, noting that $V_n=u_t/|\nabla u|=u_t$ is the normal velocity of $\partial\{u>0\}$ (since $|\nabla u|=1$ there) and $\nabla u\cdot \nu=-1$, this yields
\begin{equation*}
    \int_{\RR^n} u_t^2 \eta^2\,dx=-\frac{d}{dt}\int_{\RR^n}\frac12(|\nabla u|^2+\chi_{\{u>0\}})\eta^2\,dx-2\int_{\RR^n}u_t \nabla u \cdot \nabla \eta\, \eta\, dx.
\end{equation*}
Applying Young's inequality to the cross term and integrating over $[t_0-r^2,t_0]$,
\begin{multline} \label{eq:ut estimate}
    \int_{t_0-r^2}^{t_0}\int_{\RR^n} u_t^2 \eta^2\,dx\,dt \leq \int_{\RR^n}(|\nabla u|^2+\chi_{\{u>0\}})(t_0-r^2,x)\eta^2 \,dx \\
    + C\int_{t_0-r^2}^{t_0}\int_{\RR^n}|\nabla u|^2 |\nabla \eta |^2 \,dx\,dt\leq C(1+L^2)r^n.
\end{multline}
Now set $r=\sqrt{h}$. Since $\|\nabla u\|_\infty\leq L$, for any $y\in B_r(x)$ we have $|u(x,t)-u(y,t)|\leq L|x-y|\leq Lr$. Averaging the triangle inequality over $y\in B_r(x)$,
\begin{equation*}
    |u(x,t+h)-u(x,t)|\leq \frac{1}{|B_r|} \int_{B_{r}(x)}|u(y,t+h)-u(y,t)|\,dy+ 2Lr.
\end{equation*}
Writing $|u(y,t+h)-u(y,t)|\leq\int_t^{t+h}|u_t(y,s)|\,ds$ and applying the Cauchy--Schwarz inequality together with \eqref{eq:ut estimate}, the averaged integral is bounded by $C(1+L)r^{-n/2}\cdot r^{n/2}\cdot r$. We deduce $|u(t+h,x)-u(t,x)|\leq C(1+L)r=C(1+L)\sqrt{h}$.
\end{proof}

\subsection{Monotonicity formula}

Weiss proved the following monotonicity formula in \cite[Theorem 3.1]{Weiss99}. (Without loss of generality, we consider the space-time point $(t_0, x_0) = (0,0)$.) Let
\[ G(t,x) := \frac{1}{(4\pi|t|)^{\frac{n}{2}}} \exp\left(-\frac{|x|^2}{4|t|} \right) \]
and
\begin{align}
    \theta(r) & := \frac{1}{3r^2} \int_{-4r^2}^{-r^2} \int_{\RR^n} \left(|\nabla u|^2 + \chi_{\{u>0\}} \right) G(t,x) \, dx\, dt \nonumber \\
	& \qquad \qquad - \frac12 \cdot \frac{1}{3r^2} \int_{-4r^2}^{-r^2} \frac{1}{-t} \int_{\RR^n}  u^2\, G(t,x) \, dx\, dt,\label{def:theta}
\end{align}
which is the time-averaged free boundary energy of $u$ centered at $(0,0)$.
Then
\begin{equation} \label{eq:theta-monotonicity-classical}
	\theta'(r) \geq \frac{1}{3r^3} \int_{-4r^2}^{-r^2} \int_{\RR^n} \frac{1}{-t}\left|\nabla u \cdot x + 2t \partial_t u - u \right|^2 G(t,x) \, dx\,dt \geq 0.
\end{equation}
Thus $r \mapsto \theta(r)$ is monotone non-decreasing, and it is constant if and only if $u$ satisfies the equation $\nabla u \cdot x + 2t \partial_t u - u \equiv 0$, which implies that $u$ is a self-similar solution, i.e.
\[ u(\lambda^2 t, \lambda x) = \lambda u(t,x), \qquad \text{ for every }\lambda \geq 0,\; t<0,\text{ and } x\in \RR^n. \]
In general for any $(t_0,x_0)$ fixed, we define $r \mapsto \theta_{(t_0,x_0)}(r)$ by modifying \eqref{eq:theta-monotonicity-classical} using the translated Gaussian $G(t-t_0, x-x_0)$ and integrating over $[t_0-4r^2, t_0-r^2]\times \mathbb{R}^n$. Standard argument from the monotonicity formula and continuity shows that the map $(t_0,x_0) \mapsto \theta_{(t_0,x_0)}(0+):= \lim_{r\to 0+} \theta_{(t_0,x_0)}(r)$ is upper semi-continuous (see for example \cite[Lemma 8.1 (4)]{Weiss}).

Note that if $\limsup_{t\to 0} u(t,0)>0$, we have $\theta(r) \leq C_1 - C_2 r^{-2}$, with constants $ C_1, C_2$ depending on $L$, and thus $\lim_{r\to 0}\theta(r) = -\infty$; otherwise, by \eqref{eq:Holderintime} we have
\begin{equation*}
    \int_{\RR^n}  u^2\, G(t,x) \, dx \leq C(L) |t|, 
\end{equation*} 
and thus $\theta(r)$ is uniformly bounded from below as $r\to 0$. In the latter case, the limit $\theta(0+):= \lim_{r\to 0} \theta(r)$ is well-defined and finite.

To study the limiting behavior of the parabolic rescaling $v$ in \eqref{def:prescaling}, we define for each $\tau \geq \log 2 - \log T$ the corresponding free boundary energy
\begin{equation}\label{def:monres}
    \tilde{\theta} (\tau) := \frac23 e^{\tau} \cdot \int^{\tau+\log 2}_{\tau-\log 2} \left( \int_{\RR^n} \left( |\nabla v|^2 + \chi_{\{v>0\}} - \frac12 v^2 \right) \rho(y) dy \right) \, e^{-s} ds =  \theta(r),
\end{equation}
with $r^2= \frac12 e^{-\tau}$ and the weight
\begin{equation*}
	\rho(y) = \frac{1}{(4\pi)^{\frac{n}{2}}} e^{-\frac{|y|^2}{4}}. 
\end{equation*} 
Thus, the monotonicity formula becomes
\begin{equation}\label{eq:thetadissip}
\tilde{\theta}'(\tau) =-\frac12 r\theta'(r) \leq -\frac43 e^{\tau} \int_{\tau-\log 2}^{\tau+\log 2} \int_{\RR^n} |\partial_s v|^2 \rho(y)dy \, e^{-s}ds \leq 0.
\end{equation}

\subsection{Movement of the free boundary}

We also note the following observation about the free boundary: under the Type I assumption, the free boundary passes a $O(\sqrt{-t})$-sized ball centered at the origin as $t\nearrow 0$. Let $\Gamma_t:= \partial\{u(t,\cdot)>0\}$.
Since $(0,0)$ is a singular point, for any $\rho>0$ there exists a sequence $t_k \nearrow 0$ such that $\Gamma_{t_k} \cap \overline{B_\rho} \neq \emptyset$: otherwise $\Gamma_t \cap \overline{B_\rho} = \emptyset$ for all $t\in [t_0,0)$ and some $t_0<0$, which together with $u\geq 0$ violates the parabolic maximum principle on $(t_0,0)\times B_\rho$. Pick $x_k \in \Gamma_{t_k} \cap \overline{B_\rho}$, so $|x_k|\leq \rho$. The free boundary condition $u(t,x(t))\equiv 0$ and the assumption \eqref{eq:typeI} imply
\[ |x'(t)| = |\partial_t u(t,x(t))| = |\Delta u(t,x(t))| \leq \frac{C_s}{\sqrt{-t}}, \qquad -r_0^2< t<0, \]
so for any $-r_0^2 < \tau< t_k$,
\[ \dist(0, \Gamma_\tau) \leq |x_k| + \int_\tau^{t_k} \frac{C_s}{\sqrt{-t}}\, dt < \rho + 2C_s \sqrt{-\tau}. \]
Since $\rho>0$ was arbitrary, it follows that
\begin{equation}\label{eq:fbnear0}
    \dist(0, \Gamma_\tau) \leq 2C_s \sqrt{-\tau}, \qquad \text{ for any } -r_0^2 < \tau <0.
\end{equation}

\subsection{Examples and model behaviors}

\begin{example}\label{eg:exp}
Consider
\[ u(t,x) = \chi_{\{x>-t\}} \left(\exp(x+t)-1 \right) + \chi_{\{x<t\}} \left(\exp(-x+t)-1 \right), \]
which is a classical solution of \eqref{eq:pBern} for $t<0$ and $x\in \RR$, even though it is only \emph{locally} Lipschitz in space. The origin is a Type I singular point (in fact $|u_{xx}(t,x)|$ is bounded on $Q_r^-(0,0)$), and the corresponding limit is $\phi=|x|$.
This is usually regarded as a travelling wave solution: indeed, following the free boundary interface $\xi(t)= \pm t$, the solution can be written as $u(t,x) = U(t,x-\xi(t))$ with $U(t,z) = \exp(|z|)-1$. Focusing at the origin, however, the solution exhibits the typical behavior of solutions that are asymptotic to the self-similar profile $\phi=|x|$.
\end{example}

\begin{example}
    Let $u(t,x) = \sqrt{-t}\,\phi \left( \frac{x}{\sqrt{-t}} \right)$, where the profile $\phi$ is a solution of \eqref{eq:profile} such that $\overline{\{\phi>0\}}$ is compact. As we have observed, the positivity set of the solution satisfies $\{u(t,\cdot)>0\} = \sqrt{-t} \cdot\{\phi>0\}$ and it shrinks to the origin as $t\nearrow 0$. The point $(0,0)$ is an archetypical Type I singular point provided that $\sqrt{-t}|D^2 u(t,x)| = |D^2 \phi(x/\sqrt{-t})|$ is bounded.
\end{example}

\begin{example} \label{eg:non-type-I}
    Any solution of the \emph{elliptic} Bernoulli problem is automatically a solution of the parabolic Bernoulli problem. Consider for instance the one-homogeneous axis-symmetric solution $U_{as}(x)$ to the elliptic Bernoulli problem (see \cite[Example 2.7]{AC}), called the De Silva-Jerison solution in the literature \cite{DSJ09}. It is a classical solution everywhere except near the origin. Since $u(x)$ is one-homogeneous, the origin cannot be a Type I singular point.
    
    Although this example is not exactly a classical solution with a singular point that is not of Type I, we expect the following to be true: When $n\geq 7$ and $U_{as}(x)$ minimizes the Alt-Caffarelli energy, there exists a classical solution such that its Type I rescaling converges to $U_{as}(x)$ in $\RR^n \setminus\{0\}$, and an alternative rescaling converges to the smoothed out complete solution $U_{\pm}(x)$ constructed in \cite{DSJS}. 
\end{example}

\begin{example}\label{eg:extinct-on-sphere}
    In \cite[Theorem 8.2]{radial}, the authors construct a radially-symmetric solution with initial data compactly supported in an annulus, such that $\{u(t,\cdot)>0\}$ becomes extinct on a sphere of radius $r_*$ at time $t=T$, with no focusing at the origin. For any $x_*$ with $|x_*|=r_*$, the point $(T,x_*)$ is a Type I singular point, and the blow-up is the one-dimensional ball profile (the $n=1$ case of $\phi_B$ in Theorem \ref{thm:radial}). This bears a loose analogy with the \emph{marriage ring} example in mean curvature flow, which is a thin rotationally symmetric torus in $\mathbb{R}^3$ that shrinks to a circle.
\end{example}

\section{Blow-up analysis at Type I singular points}\label{sec:blowup}

\subsection{Proof of Theorem \ref{thm:blowup}}

We work in the rescaled variables \eqref{def:prescaling}: with $s_j=-2\log\lambda_j$, one has $u_{\lambda_j}(t,x)=\sqrt{-t}\,v(s_j+\sigma,x/\sqrt{-t})$, where $\sigma=-\log(-t)$. Under this change of variables, Theorem \ref{thm:blowup} is equivalent to the following statement, which we prove in this section: for any sequence $s_j\to\infty$, modulo passing to subsequences, $w_j(\sigma,y):=v(s_j+\sigma,y)$ converges locally uniformly to a non-trivial solution $\phi$ of \eqref{eq:profile} satisfying the stated properties.

{\bf Step 1. Compactness of $\{w_j\}$ and $\{\chi_{\{w_j>0\}} \}$.} Firstly, we show that the sequence of pairs $(w_j, \chi_{\{w_j>0\}})$ converges to $(\phi,\chi)$ in an appropriate sense, and the limiting functions $\phi, \chi$ are independent of the $\sigma$-variable.
Given any sequence $s_j\to\infty$, set $w_j(\sigma,y):=v(s_j+\sigma,y)$. Since $(0,0)$ is a free boundary point, the $C^{1/2}$-in-time estimate of Lemma~\ref{lem:holder} gives $|u(-e^{-s},0)|\leq C(1+L)\,e^{-s/2}$, hence $|v(s,0)|\leq C(1+L)$. Together with $|\nabla_y v|=|\nabla_x u|\leq L$, this yields uniform boundedness and equicontinuity in $y$ on compact sets. A routine computation using $|\nabla v|\leq L$ and Lemma~\ref{lem:holder} gives equicontinuity in the $\sigma$-variable as well. By Arzel\`a--Ascoli, a subsequence of $w_j$ converges locally uniformly to a limit $\phi$.

Since $\tilde\theta$ in~\eqref{def:monres} is non-increasing and bounded from below, $\int_{\tau_0+\log 2}^\infty(-\tilde\theta')\,d\tau<\infty$, where $\tau_0: =-\log T$. Combined with~\eqref{eq:thetadissip}, this gives a global dissipation bound
\begin{equation*}
    \int_{\tau_0}^\infty\!\int|\partial_s v|^2\,\rho\,dy\,ds \leq  \int_{\tau_0+\log 2}^\infty(-\tilde\theta')\,d\tau <\infty 
\end{equation*} 
after exchanging the order of integration via Fubini (the exponential weights $e^{\tau}$ and $e^{-s}$ cancel). In particular, for any bounded interval $I\subset\RR$ and $R>0$, $\partial_\sigma w_j\to 0$ in $L^2_\rho(I\times B_R)$. The limit $\phi$ is therefore independent of $\sigma$: for any compact $K\subset\RR^n$ and $\sigma_1,\sigma_2\in I$, by Cauchy--Schwarz,
\[
\int_K|w_j(\sigma_2,y)-w_j(\sigma_1,y)|\,\rho\,dy
\leq C_K\,|\sigma_2-\sigma_1|^{1/2}
\Bigl(\int_{\sigma_1}^{\sigma_2}\!\int_K|\partial_\sigma w_j|^2\,\rho\Bigr)^{\!1/2}\to 0, \quad \text{ as } s_j\to \infty.
\]

After a further subsequence, $\chi_j:=\chi_{\{w_j>0\}}\overset{\ast}{\rightharpoonup} \chi$ in $L^\infty_{\rm loc}$ in space-time. 

{\bf Step 2. Identifying $\chi=\chi_{\{\phi>0\}}$.} To this end, note that the Type I assumption \eqref{eq:typeI} gives a uniform bound for $|D^2_y w_j|$ on compact cylinders:
\begin{equation}\label{eq:HessianUB}
    \sup_{([-R,R]\times B_R)\cap \{w_j>0\}} |D^2_{y} w_j| \leq \sup_{Q_j^- \cap \{u>0\}} \sqrt{-t} |D_x^2 u(t,x)| \leq C_s
\end{equation} 
where $Q_j^-:= \left[-\exp(R-s_j),0 \right) \times B_{R \exp( (R-s_j)/2)}$ is contained in $Q_{r_0}^-$ for sufficiently large $j$.
Since $w_j$ and $\nabla_y w_j$ are already locally uniformly bounded and $w_j$ satisfies \eqref{eq:respar}, this also yields a uniform bound for $|\partial_\sigma w_j|$. Thus, for any
$R>0$ and $j$ sufficiently large,
\begin{equation}\label{eq:rescaledC2}
    \sup_{([-R,R]\times B_R)\cap \{w_j>0\}}\bigl(|D^2_{y} w_j|+|\partial_{\sigma}w_j|\bigr) \leq C_R.
\end{equation}
Letting $\Omega_{j,\sigma}=\{w_j(\sigma,\cdot)>0\}$ and $\Gamma_{j,\sigma}=\partial \Omega_{j,\sigma}$, this Hessian bound plus the Bernoulli boundary condition $|\nabla_{y}w_j|=1$ on $\Gamma_{j,\sigma}$ implies that, for $y\in \Omega_{j,\sigma}\cap B_R$ and some $r_0 >0$,
\begin{equation}\label{eq:nearbd-lower}
w_j(\sigma,y)\geq \frac{1}{2}\operatorname{dist}(y,\Gamma_{j,\sigma}), \qquad \text{whenever } \operatorname{dist}(y,\Gamma_{j,\sigma})<r_0.
\end{equation}

Besides, by \eqref{eq:fbnear0} it follows that 
\begin{equation}\label{eq:fbnear0-v}
    \dist(0, \Gamma_{j,\sigma}) \leq 2C_s
\end{equation}  
after rescaling, for any $\sigma \in [-R,R]$ and any $j\geq j(R,r_0)$ large enough.
Let $z\in \Gamma_{j,\sigma} \cap B_R $ and let $\nu(z)$ be the unit normal vector pointing towards $\Omega_{j,\sigma}$. Consider the function $g(s):= w_j(\sigma, z+s\nu(z))$. Suppose $s_z$ is the first point such that $g(s_z) = 0$, then $g'(s_z)\leq 0$, and the line segment $(z,z+s_z \nu(z)) \subset \Omega_{j,\sigma}$. It follows from the Hessian bound \eqref{eq:HessianUB} that
\[ 1= \partial_{\nu} w_j(\sigma, z) \leq g'(0) - g'(s_z) \leq s_z \cdot \sup_{[z,z+s_z \nu(z)]} | g^{''}(s)| \leq C_s s_z, \]
so $s_z \geq 1/C_s$. Therefore 
\begin{equation}\label{eq:thickness}
    w_j(\sigma, \cdot)>0 \text{ on a $r_0$-tubular neighborhood of $\Gamma_{j,\sigma}$ in the inner normal direction}
\end{equation} 
 with $r_0 = 1/(2C_s)$. In other words, we say the positive set $\Omega_{j,\sigma}$ is at least $r_0$-thick near each connected component of the free boundary $\Gamma_{j,\sigma}$.

Define the normal map
\begin{equation*}
    \Psi(z,r)=z+r \nu(z), \quad (z,r)\in (B_{R+1}\cap \Gamma_{j,\sigma}) \times (-r_0,r_0),
\end{equation*}
which, thanks to the Hessian bound \eqref{eq:rescaledC2} (after possibly shrinking $r_0$), has Jacobian
\begin{equation*}
|J_{\Psi}(z,r)|=\left|\prod_{i=1}^{n-1}(1-r\kappa_i(z))\right|\leq (1+C|r|)^{n-1}\leq 2,
\end{equation*}
where $\kappa_i$ denote the principal curvatures of $\Gamma_{j,\sigma}$. Thus, for $\delta \in (0,r_0)$, noting that the $\delta$-tubular neighborhood of the free boundary satisfies $\mathcal{N}_{\delta}(\Gamma_{j,\sigma})\cap B_R\subset \Psi(\Gamma_{j,\sigma}\cap B_{R+1} \times (-\delta,\delta))$,
\begin{equation}\label{eq:tube-meas}
|\mathcal{N}_{\delta}(\Gamma_{j,\sigma})\cap B_R |\leq  4\delta \mathcal{H}^{n-1}(\Gamma_{j,\sigma}\cap B_{R+1}).
\end{equation}
Now let $\eta \in C^{\infty}_c(B_{R+1})$ be a smooth non-negative bump function in space with $\eta \equiv 1$ on $B_R$. We may rewrite the rescaled equation \eqref{eq:respar} in divergence form as
\begin{equation*}
    \operatorname{div}(\rho \nabla w_j)=\rho (\partial_{\sigma}w_j-w_j/2).
\end{equation*}
Integrating $\operatorname{div}(\eta \rho \nabla w_j)$ over $\Omega_{j,\sigma}$ and using the divergence theorem, we get
\begin{equation*}
    \int_{\Gamma_{j,\sigma}}\eta \rho \, d\mathcal{H}^{n-1}=- \int_{\Omega_{j,\sigma}}\nabla \eta \cdot \nabla w_j \rho\, dy - \int_{\Omega_{j,\sigma}} \eta (\partial_{\sigma}w_j -w_j/2)\rho \,dy,
\end{equation*}
and, thus, 
\begin{equation*}
 \int_{B_R\cap \Gamma_{j,\sigma}} \rho\, d\mathcal{H}^{n-1} \leq C_R(1+ \int_{B_{R+1}}|\partial_{\sigma}w_j|\rho \, dy).
\end{equation*}
Integrating in $\sigma$ and using that $\rho \geq c_R>0$ on $B_R$, we get the perimeter upper bound
\begin{equation}\label{eq:spacetime-area}
   \int_{-R}^R \mathcal{H}^{n-1}(B_R \cap \Gamma_{j,\sigma})d\sigma =\int_{-R}^{R} \int_{B_R\cap \Gamma_{j,\sigma}}  d\mathcal{H}^{n-1} \leq C_R\Bigl(2R+ \int_{-R}^{R}\int_{B_{R+1}}|\partial_{\sigma}w_j|\rho\, dy\, d\sigma\Bigr)\leq C'_R,
\end{equation}
where in the last inequality we used the fact that $\partial_s v \in L^2_{\rho}((\tau_0,\infty)\times \RR^n)$.

Let 
\[ \lambda_j:= \|\partial_{\sigma}w_j\|_{L^2_\rho((-R,R)\times \RR^n)}^{\frac12}.\]
Since $\int_{\tau_0}^\infty\!\int|\partial_s v|^2\,\rho\,dy\,ds<\infty$, we have $\lambda_j \to 0$ as $j\to\infty$.
We define the set of good times by
\begin{equation*}
    G_j=\{\sigma \in [-R,R]: \|\partial_{\sigma} w_j(\sigma,\cdot)\|_{L^2_{\rho}(\RR^n)}\leq \lambda_j\}.
\end{equation*}
Note that, by Chebyshev's inequality,
\begin{equation}\label{eq:bad-times}
|[-R,R]\setminus G_j| \leq \frac{\|\partial_{\sigma}w_j\|_{L^2_{\rho}((-R,R)\times \RR^n)}^2}{\lambda_j^2}=\|\partial_{\sigma}w_j\|_{L^2_{\rho}((-R,R)\times \RR^n)}\to 0 
\end{equation}
as $j\to \infty$. Moreover, by \eqref{eq:rescaledC2}, for any $p>n$ and $K>0$ fixed, and every $\sigma \in G_j$,
\begin{equation}\label{eq:goodtime-lp}
    \|\partial_{\sigma}w_j(\sigma,\cdot)\|_{L^p(B_K)}\leq C_{K,p} \|\partial_{\sigma}w_j(\sigma,\cdot)\|_{L^2_\rho(\RR^n)}^{2/p}\leq C_{K,p}\lambda_j^{2/p} \to 0, 
\end{equation}
as $j\to \infty$.

 We claim that there exists $c_R>0$ such that, for all large $j$, all $\sigma\in G_j$, and all $y\in \Omega_{j,\sigma}\cap B_R$ with $\operatorname{dist}(y,\Gamma_{j,\sigma})\geq r_0/2$, 
\begin{equation}\label{eq:nondeg}
w_j(\sigma,y)\geq c_R.
\end{equation}
In other words $w_j$ is \emph{uniformly non-degenerate}.
To see this, fix such $j, \sigma, y$, and let $d=\operatorname{dist}(y,\Gamma_{j,\sigma})$.
By \eqref{eq:fbnear0-v}, we know $\Gamma_{j,\sigma}\cap \overline{B_{2C_s}} \neq \emptyset$, and thus $d\leq R+2C_s$, so $B_d(y)\subset B_{2R+2C_s}$.  
Moreover, choose $y_0\in \Gamma_{j,\sigma}$ with $|y-y_0|=d$, and let $z\in [y_0,y]\cap \partial B_{d-r_0/200}(y)$. Then $\operatorname{dist}(z,\Gamma_{j,\sigma})=r_0/200<r_0$, and by \eqref{eq:nearbd-lower},
\begin{equation*}
    w_j(\sigma,z)\geq \frac{r_0}{400}.
\end{equation*}
Since $|z-y|=d-r_0/200\leq (1-\frac{r_0}{200(R+C_s)})d$, the interior Harnack inequality for inhomogeneous elliptic PDEs \cite[Theorems 8.17 and 8.18]{GT}, applied to \eqref{eq:respar}, yields
\begin{equation*}
w_j(\sigma,z)\leq C_{p,R}(w_j(\sigma,y)+ d^{2-n/p}\|\partial_{\sigma}w_j(\sigma,\cdot)\|_{L^p(B_d(y))}).
\end{equation*}
Since $p>n$ and $B_d(y)\subset B_{2R+2C_s}$, \eqref{eq:goodtime-lp} (applied with radius $2R+2C_s$) therefore implies \eqref{eq:nondeg} for $j$ sufficiently large.

Now let
\begin{equation*}
    \delta_j:=\|w_j-\phi\|_{L^{\infty}((-R,R)\times B_R)} \to 0, \quad E_j:=\{(\sigma,y)\in [-R,R]\times B_R: 0<w_j(\sigma,y)\leq 2\delta_j\}.
\end{equation*}
Combining \eqref{eq:nearbd-lower}, \eqref{eq:tube-meas}, \eqref{eq:spacetime-area}, and \eqref{eq:nondeg}, we deduce, for large $j$ with $\delta_j \ll \min(c_R,r_0)$,
	\begin{multline}\label{eq:Ej-good-times}
	    |E_j\cap (G_j\times B_R)|\leq \int_{-R}^{R}|\mathcal{N}_{4\delta_j}(\Gamma_{j,\sigma})\cap B_{R}|d\sigma 
	\leq 8\delta_j \int_{-R}^R \mathcal{H}^{n-1}(\Gamma_{j,\sigma}\cap B_{R+1})d\sigma \leq C_R \delta_j,
	\end{multline}
	whereas \eqref{eq:bad-times} implies 
	\begin{equation}\label{eq:Ej-bad-times}
	    |E_j \cap (((-R,R)\setminus G_j)\times B_R)|\leq |B_R|\cdot |(-R,R)\setminus G_j|\to 0.
	\end{equation}
On the other hand, by the definition of $\delta_j$ we have
\begin{equation}\label{eq:indicator-split}
|\chi_{\{w_j>0\}}-\chi_{\{\phi>0\}}|\leq |\chi_{\{w_j>2\delta_j\}}-\chi_{\{\phi>0\}}|+ \chi_{\{0<w_j\leq 2\delta_j\}}\leq \chi_{\{0<\phi \leq 3\delta_j\}}+\chi_{E_j}.    
\end{equation}
	Combining \eqref{eq:indicator-split} with \eqref{eq:Ej-good-times} and \eqref{eq:Ej-bad-times}, and applying the dominated convergence theorem to $\chi_{\{0<\phi\leq 3\delta_j\}}$, we finally conclude that $\chi_{\{w_j>0\}} \to \chi_{\{\phi>0\}}$ in $L_{\operatorname{loc}}^1$ in space-time, and thus $\chi=\chi_{\{\phi>0\}}$, as desired.

{\bf Step 3. Showing the limit function $\phi$ satisfies \eqref{eq:profile}.}
Now, modulo shrinking the set $G_j$ slightly we may choose a sequence $\sigma_j\in G_j$ of good times such that
\begin{equation}\label{eq:periUB}
    \{w_j(\sigma_j, \cdot)>0\} \to \{\phi>0\}, \qquad \text{ and } \qquad \mathcal{H}^{n-1}(B_R \cap \Gamma_{j,\sigma_j})\leq C_R,
\end{equation}
where we used \eqref{eq:spacetime-area} for the second statement.
Passing to the limit in \eqref{eq:respar}, using \eqref{eq:goodtime-lp}, readily implies that $\phi$ solves the equation in \eqref{eq:profile} in $\{\phi>0\}$. 

We now verify that 
\begin{equation}\label{eq:CVHausdorff}
    \Gamma_{j,\sigma_j} \to \partial\{\phi>0\}\cap B_R \qquad \text{ locally in the Hausdorff sense}. 
\end{equation} 
First, we prove that for any $\epsilon>0$,
\begin{equation}
   \partial\{\phi>0\} \cap B_{R} \subset \mathcal{N}_{\varepsilon}( \Gamma_{j,\sigma_j})
\end{equation} 
when $j \geq j(\epsilon)$ is sufficiently large. 
Suppose $y_0\in \partial\{\phi>0\} \cap B_R \neq \emptyset$. Then by the continuity of $\phi$, for any $\epsilon>0$ there exists $y\in B_{\frac{\epsilon}{4}}(y_0)\cap B_R$ such that $0<\phi(y)< \min\{\frac{c_R}{2}, \frac{\epsilon}{4} \}$. By the uniform convergence of $w_j(\sigma_j,\cdot)$ to $\phi$ we have that $w_j(\sigma_j, y)< \frac32 \phi(y)$ for $j$ sufficiently large. Since $w_j(\sigma_j, y) < \frac32 \phi(y) < c_R$, by \eqref{eq:nondeg} $\operatorname{dist}(y,\Gamma_{j,\sigma_j}) < r_0/2$. Thus \eqref{eq:nearbd-lower} implies that $\operatorname{dist}(y,\Gamma_{j,\sigma_j}) \leq 2w_j(\sigma_j,y) < \frac34 \epsilon$, and moreover $\operatorname{dist}(y_0,\Gamma_{j,\sigma_j}) < \epsilon$.

Conversely, we now show that for any $\delta>0$ and any $\epsilon \in (0,\delta R)$,
\begin{equation}
    \Gamma_{j,\sigma_j}\cap B_{(1-\delta)R} \subset \mathcal{N}_{\varepsilon}( \partial\{\phi>0\})
\end{equation}
when $j\geq j(\epsilon,\delta)$ sufficiently large. Suppose by contradiction there exists $x_j \in \Gamma_{j,\sigma_j}\cap B_{(1-\delta)R}$ such that $B_{\varepsilon}(x_j)\cap \partial\{\phi>0\}=\emptyset$.
By considering a segment starting at $x_j$ along the inner normal direction, there exists $\eta \in (0, r_0 )$, depending only  on $R$ and the Hessian bound \eqref{eq:rescaledC2}, such that
\begin{equation}
\sup_{B_{r}(x_j)}w_j(\sigma_j,\cdot)\geq \frac{1}{2}r, \quad r \in (0,\eta).
\end{equation}
Thus, after extracting a subsequence so that $B_r(x_j)$ accumulate, we have $x_j\to x \in \overline{\{\phi>0\}}$. Since $\phi(x)=\lim_{x_j\to x}w_j(\sigma_j,x_j) = 0$, we have $x\in \partial \{\phi>0\} \cap B_{\varepsilon}(x_j)$ for 
$j$ sufficiently large, a contradiction. This finishes the proof of \eqref{eq:CVHausdorff}. We remark that by the nonemptiness of $\Gamma_{j, \sigma_j} \cap B_R(0)$ \eqref{eq:fbnear0-v} and \eqref{eq:CVHausdorff} one must have $\partial \{\phi > 0\} \neq \emptyset$ and $\phi \not\equiv 0$.

By the curvature and perimeter upper bounds \eqref{eq:periUB}, the number of connected components of $\Gamma_{j,\sigma_j} \cap B_R$ is uniformly bounded. By the Arzel\`a-Ascoli theorem, after extracting a subsequence, each connected component $\Gamma_{j,\sigma_j}^k$ converges in $C^{1}$ norm to a $C^1$ hypersurface $\Gamma^k$, and by \eqref{eq:CVHausdorff} $\partial\{\phi>0\} \cap B_R =\bigcup_{k}\Gamma^k$. Note that \eqref{eq:thickness} indicates that the connected components $\Gamma_{j,\sigma_j}^k$ are at least $r_0$-distance apart in the positive direction (and there may not be any sheet separation in the opposite direction, as is shown in Example \ref{eg:exp}). Together with the non-degeneracy \eqref{eq:nondeg} and the uniform convergence $w_j(\sigma_j,\cdot) \to \phi$, this implies that $\Gamma^k$'s are also at least $r_0$-distance apart in the positive direction, but they may even coincide in the opposite direction.

We now check that the Bernoulli condition holds classically, namely $|\nabla \phi(y)|\to 1$ as $y\to \partial\{\phi>0\}$ from within $\{\phi>0\}$. Fix $\delta\in(0,\tfrac12)$, let $y\in \{\phi>0\}\cap B_{(1-2\delta)R}$, with $d:=\operatorname{dist}(y,\partial\{\phi>0\})<\delta R/4$. For $j$ sufficiently large, the Hausdorff convergence of the free boundaries, the uniform convergence of $w_j(\sigma_j,\cdot) \to \phi$ and $B_{3d/4}(y) \subset \{\phi>0\}$ give
\[
B_{d/2}(y)\subset \{w_j(\sigma_j,\cdot)>0\}, \qquad \operatorname{dist}(y,\Gamma_{j,\sigma_j})\leq 2d.
\]
Since $w_j(\sigma_j,\cdot)\to\phi$ uniformly and \eqref{eq:HessianUB} yields a uniform $C^{1,1}$ bound on $w_j(\sigma_j,\cdot)$ in $B_{d/2}(y)$, we have $\nabla w_j(\sigma_j,\cdot)\to \nabla \phi$ uniformly on $B_{d/4}(y)$. 
In fact, we have that $w_j(\sigma_j, \cdot) \to \phi$ in $C^{1,\alpha}$ for any $\alpha<1$, and that $\|D^2\phi\|_\infty \leq C_s$ on $B_{d/4}(y)$.
If $z_j\in \Gamma_{j,\sigma_j}$ satisfies $|y-z_j|=\operatorname{dist}(y,\Gamma_{j,\sigma_j})$, then by the Bernoulli boundary condition and \eqref{eq:HessianUB},
\[
\bigl||\nabla w_j(\sigma_j,y)|-1\bigr|
\leq C_s |y-z_j|
\leq 2C_s d.
\]
Passing to the limit $j\to\infty$, we obtain $||\nabla \phi(y)|-1|\leq 2C_s\,\operatorname{dist}(y,\partial\{\phi>0\})$, which proves the claim.

{\bf Step 4. Energy of the limit function $\phi$.}
By the definition of $w_j$ and a change of variables, we have
\[ \tilde{\theta}(s_j) = \frac23 \int_{-\log 2}^{\log 2} \left[ \int_{\RR^n} \left(|\nabla w_j(\sigma, \cdot)|^2 + \chi_{\{w_j(\sigma,\cdot)>0\} } - \frac12 w_j^2(\sigma,\cdot)  \right) \rho \, dy \right] e^{-\sigma} d\sigma. \]
Passing to an unlabeled subsequence, we may assume that \(\sum _{j=1}^{\infty} \left\lvert [-\log 2, \log 2] \setminus G_j \right\rvert < \infty\) so that \(\bigcup _{j_0 \in \mathbb{N}} \bigcap _{j > j_0} G_j\) has full measure. 
For each $\sigma \in\bigcup _{j_0 \in \mathbb{N}}\bigcap _{j > j_0} G_j$ 
we have
\[ \int_{\RR^n} \left(|\nabla w_j(\sigma, \cdot)|^2 + \chi_{\{w_j(\sigma,\cdot)>0\} } - \frac12 w_j^2(\sigma,\cdot)  \right) \rho dy \to \int_{\RR^n} \left(|\nabla \phi|^2 + \chi_{\{\phi>0\} } - \frac12 \phi^2  \right) \rho dy \]
as $j\to \infty$. Note that by construction, the limit $\phi$ depends on the sequence $s_j \to \infty$ that defines $w_j$, but it is the same for any choice of $\sigma $.
Combined with the uniform bounds on $w_j, |\nabla w_j|$, we conclude that
\[ \tilde{\theta}(s_j) \to \int_{\RR^n} \left(|\nabla \phi|^2 + \chi_{\{\phi>0\} } - \frac12 \phi^2  \right) \rho \, dy, \qquad \text{ as } s_j \to \infty. \]
By the equation \eqref{eq:profile} of $\phi$, written in divergence form, and the divergence theorem, we have
\begin{align*}
    0 = \int_{\RR^n} \left(\divg(\rho \nabla \phi) + \frac12 \rho\phi \right) \phi \, dy = - \int_{\RR^n} \rho|\nabla \phi|^2 + \frac12 \int_{\RR^n} \rho \phi^2.
\end{align*}
Combined with the monotonicity formula, we conclude
\[ \theta(0+) = \lim_{r\to 0} \theta(r) = \lim_{s_j \to \infty} \tilde{\theta}(s_j) = \int_{\RR^n} \chi_{\{\phi>0\} } \rho\, dy \leq 1.  \]

To show $\theta(0+)> 1/2$, we consider the self-similar solution $u_s(t,x) := \sqrt{-t} \phi(x/\sqrt{-t})$. For any $y_0\in \partial \{\phi>0\}$ and any $t_0<0$, the point $(t_0, x_0:=\sqrt{-t_0}y_0)$ is a regular point on the free boundary $\partial\{u_s(t_0,\cdot)>0\}$, with uniform Hessian bound. Thus applying the same blowup process at $(t_0,x_0)$ to the sequence
\[ \tilde{u}_{\lambda_j}(t,x) := \frac{u_s(t_0 + \lambda_j^2 t, x_0 + \lambda_j x)}{\lambda_j}, \qquad \lambda_j \to 0 \]
gives the half-plane solution, and in particular $\theta_{(t_0,x_0)}(0+) = 1/2$. Taking $t_0$ sufficiently close to $0$, the upper semi-continuity of $\theta_{(\cdot, \cdot)}(0+)$ implies that $\theta(0+) = \theta_{(0,0)}(0+) \geq 1/2$. However, $\theta(0+) = 1/2$ only if $\phi$ is the half-plane solution, which is impossible at singular points (by \cite[Theorem 8.4]{AW}), and hence $\theta(0+)> 1/2$.

\subsection{Proof of Theorem \ref{thm:blowup-twophase}}

Suppose in addition that $u$ satisfies \eqref{eq:typeI} with the upper bound $C_s(r)$ tending to $0$ as $r\to 0$. This implies that on each compact set the Hessian bound for $w_j$ in \eqref{eq:HessianUB} is given by $C_s(R r_j)$, with $r_j = e^{-s_j/2} \to 0$. Hence
\[ \sup_{B_R} |D^2\phi| \leq \lim_{j\to \infty} \sup_{B_R} |D^2_y w_j(\sigma_j, \cdot)| \leq \lim_{j\to \infty} C_s(R r_j) = 0. \]
In particular $\phi$ is linear in each connected component of $\{\phi>0\}$. Hence the equation \eqref{eq:profile} implies that $y\cdot \nabla \phi = \phi$, which implies $\phi$ is $1$-homogeneous with respect to the origin, i.e. $\phi(\lambda y) = \lambda \phi(y)$ for every $y\in \{\phi>0\}$ and $\lambda>0$. Thus each connected component of $\{\phi>0\}$ is a cone with vertex at the origin. Since $\phi$ is linear, it can only be a flat cone. Combined with the Bernoulli boundary condition, we have that modulo rotations, either $\phi= (y_1)_+$ or $\phi=|y_1|$. However if $\phi=(y_1)_+$, the point $(0,0)$ can not be a singular point. Indeed, the local uniform convergence in space-time, the Hessian bound tending to zero, and the thickness estimate \eqref{eq:thickness} give local $C^1$ convergence up to the free boundary from within the positivity set and place the parabolic dilations $u_{\lambda_j}$ in the flatness class of \cite[Theorem 8.4]{AW} for all sufficiently large $j$ (see also \cite[Theorem 8.1]{AC}). Thus the theorem implies that $\partial\{u(0,\cdot)>0\}$ is regular. 

From now on assume $\phi=|y_1|$. 
Recall that as $s_j\to\infty$,  $w_j(\sigma , y) = v \left( s_j + \sigma , y \right) \to \phi (y)$ locally uniformly in $\sigma$ and $y$. 
Let $\sigma \in I :=[-\log 2, \log 2]$, and we define
\[ \delta_j:= \|w_j - \phi\|_{L^\infty(I\times B_4 )} \to 0, \qquad \kappa_j:= \sup_{\left( I\times B_4 \right) \cap \{w_j>0\}} |D^2 w_j| \leq C_s(3\sqrt{2} e^{-s_j/2}) \to 0. \]
Let $p\in \Gamma_{j,\sigma} \cap B_3$, and $\nu$ be the unit normal vector at $p$ pointing in the positivity set $\Omega_{j,\sigma}$. Firstly, combining \eqref{eq:fbnear0-v} and the assumption $C_s(r) \to 0$, we know that $\Gamma_{j,\sigma} \cap B_3 \neq \emptyset$. We claim there is a uniform constant $\tau \in (0,1)$ such that 
\begin{equation}\label{eq:twophase-cv}
    |p_1| \leq \delta_j, \qquad |\nu_1| = |\langle \nu, e_1 \rangle| \geq 1- \epsilon_j  \text{ with } \epsilon_j:= \frac{\kappa_j}{2} \tau + \frac{\delta_j}{\tau} \to 0.
\end{equation}
The first statement follows easily from
\[ |p_1|=\phi(p) = |w_j(\sigma, p) - \phi(p)| \leq \delta_j. \]
Recall that \eqref{eq:thickness} implies $p + t \nu \in \Omega_{j,\sigma} \cap B_4$ for all $0<t<\tau$, where $\tau:= \min\{r_0, 1\}$ is independent of $j$. Thus the uniform Hessian bound applies, together with the Bernoulli boundary condition, to show
\[ w_j(\sigma, p+\tau\nu) \geq \tau - \frac{\kappa_j}{2} \tau^2. \]
On the other hand
\[ w_j(\sigma, p+\tau\nu) \leq \phi(p+\tau\nu) + \delta_j = |p_1 + \tau \langle \nu, e_1 \rangle | + \delta_j \leq \tau |\langle \nu, e_1 \rangle | + 2\delta_j. \]
Combined, we conclude that $|\langle \nu, e_1 \rangle| \geq 1- \epsilon_j$.

Next, we show that $\Omega_{j,\sigma} \cap B_2$ has exactly two connected components. More precisely, let $y' \in B_1 \cap \mathbb{R}^{n-1}$ be fixed, we claim
\begin{equation}\label{cl:twophase-2components}
    Z_{y'}:= \{t \in [-2,2]: w_j(\sigma, (t,y'))=0 \} \text{ is a non-empty single interval contained in } [-\delta_j, \delta_j]. 
\end{equation} 
If $Z_{y'}$ contains more than one interval, then there is an non-empty interval $(a,b) \subset [-2,2] \setminus Z_{y'} $ with $a, b \in Z_{y'} $ such that $h:= w_j(\sigma, (\cdot ,y')) >0 $ on $(a,b)$. By \eqref{eq:twophase-cv}, 
\[ h'(a+) = \langle \nabla w_j(\sigma, (a,y')), e_1 \rangle \geq 1-\epsilon_j, \qquad h'(b-) = \langle \nabla w_j(\sigma, (b,y')), e_1 \rangle  \leq -( 1-\epsilon_j).  \] 
On the other hand, the uniform Hessian bound gives $h'(a+) - h'(b-) \leq \kappa_j$, and we get a contradiction since both $\epsilon_j, \kappa_j \to 0$.
To show $Z_{y'}$ is always non-empty, recall we have $\Gamma_{j,\sigma} \cap B_1 \neq \emptyset$ for $j$ sufficiently large. Then $Z_{y'} \neq \emptyset$ for some $y'$ (depending on $j$). Moreover, each $\Gamma_{j,\sigma}$ is locally a $C^{1,1}$ hypersurface satisfying \eqref{eq:twophase-cv}. So the condition $Z_{y'} \neq \emptyset$ is both open and closed, and thus it holds on all of the connected set $B_1 \cap \mathbb{R}^{n-1}$. 
This finishes the claim \eqref{cl:twophase-2components}. In particular, $\Gamma_{j,\sigma}\cap B_2$ consists of the graphs of two $C^1$ functions over the hyperplane $\{y_1=0\}$. After a change of variables, this gives the desired statement, with the sequence $\lambda_j:= e^{-s_j/2} \to 0$.

It remains to prove {\bf the uniqueness of the tangents} under the assumption \eqref{eq:dini-Cs}. 
We showed above for every sequence $s_j \to \infty$ there is a subsequence that converges locally uniformly to a limit function
\[
    \phi \in \mathcal Z:=\{|\langle y,e\rangle|:e\in\mathbb S^{n-1} / \left\lbrace\pm 1\right\rbrace\},
\]
but $\phi$ may depend on the sequence. 
For any $R\geq 10$ fixed, we have
\begin{equation}\label{eq:Z-trapping-simple}
    \delta_R(s):%=\operatorname{dist}_{L^\infty(B_{4R})}(v(s,\cdot),\mathcal Z)
    = \inf_{\phi\in \mathcal{Z} } \|v(s,\cdot) - \phi\|_{L^\infty(B_{4R})} \to 0
    \qquad\text{as } s\to\infty.
\end{equation}
Because if this is not the case, there exist $\epsilon>0$ and a sequence $s_j \to \infty$ such that $\delta_R(s_j) \geq \epsilon$. However, since we can extract a subsequence $\{s_{j_k}\}$ from $s_j$ such that $v(s_{j_k},\cdot)$ converges to some $\phi \in \mathcal{Z}$, it is impossible that $\delta_R(s_{j_k}) \geq \epsilon$.

Recall that
\begin{equation}\label{eq:kappaR-simple}
    \kappa_R(s):=\sup_{\{v(s,\cdot)>0\}\cap B_{4R}} |D_y^2v(s,\cdot)| \leq C_s(4Re^{-s/2}).
\end{equation}
Hence the assumption \eqref{eq:dini-Cs} implies 
\begin{equation}\label{eq:kappa-dini-simple}
    \int_{s_0}^\infty \kappa_R(s)\,ds
    \leq 2\int_0^{4Re^{-s_0/2}}\frac{C_s(r)}{r}\,dr<+\infty,
    \qquad
    \lim_{s_0 \to\infty}\int_{s_0}^\infty \kappa_R(s)\,ds=0.
\end{equation}
Suppose that
\begin{equation}\label{hyp:positive}
    \text{the line segment $[y_1,y_2]$ lies in the positivity set $\{v(s,\cdot)>0\}\cap B_{3R}$}.
\end{equation}
Then by the equation \eqref{eq:respar}, we have
\begin{align*}
 & \frac{d}{ds} \left( v(s,y_2) - v(s,y_1) \right) \\
 & =\Delta v(s, y_2)-\Delta v(s,y_1)
    -\frac12\bigl(y_2\cdot\nabla v(s,y_2)
            - y_1 \cdot\nabla v(s,y_1)\bigr)
        +\frac12\bigl(v(s,y_2)-v(s,y_1)\bigr) \\
  &   =\Delta v(s,y_2)-\Delta v(s,y_1)
    +\frac12\bigl[ v(s,y_2)-v(s,y_1)
            -(y_2-y_1) \cdot\nabla v(s,y_1)\bigr] -\frac12 y_2\cdot
        \bigl(\nabla v(s,y_2)-\nabla v(s,y_1)\bigr).
\end{align*}
By the Hessian bound \eqref{eq:kappaR-simple} and Taylor's theorem, we obtain the estimates
\[
\begin{aligned}
&|\Delta v(s,y_2)|+|\Delta v(s,y_1)|\leq C\kappa_R(s),\\
&\left|
    v(s,y_2)-v(s,y_1)-(y_2 -y_1)\cdot\nabla v(s,y_1)
\right|
    \leq C\kappa_R(s),\\
&\left|
    y_2 \cdot(\nabla v(s,y_2)-\nabla v(s,y_1))
\right|
    \leq C_R\kappa_R(s).
\end{aligned}
\]
Therefore, as long as the hypothesis \eqref{hyp:positive} holds for all $s\in (s_1,s_2)$, it follows that
\begin{equation}\label{eq:spatial-change-v}
    \left|(v(s_2,y_2)-v(s_2,y_1) - (v(s_1,y_2)-v(s_1,y_1) ) \right| \leq C_R \int_{s_1}^{s_2} \kappa_R(s) \, ds.
\end{equation}

Let $s_0 >0$ be a large time to be chosen later. Modulo one fixed rotation, 
we may assume that
\begin{equation}\label{eq:initial-double-plane-close}
    \delta_R(s_0) = \|v(s_0,\cdot)-|y_1|\|_{L^\infty(B_{4R})}.
\end{equation}
Let \(e_1,\ldots,e_n\) denote the standard basis of $\mathbb{R}^n$, set
\begin{equation*}
    q_0:=2Re_1,\qquad q_i:=q_0+e_i \text{ for each } i=1,\ldots,n,
\end{equation*}
and define
\begin{equation*}
    a_i(s):=v(s,q_i)-v(s,q_0),\qquad
    a(s):=(a_1(s),\ldots,a_n(s)).
\end{equation*}
Note that for $s$ large enough, the vector $a(s)$ encodes the information  which $\phi\in \mathcal{Z}$ is $v(s,\cdot)$ close to. We aim to show that $a(s)$ has a unique limit as $s\to \infty$.
By \eqref{eq:initial-double-plane-close}, at the initial time $s_0$ each of the line segments $[q_0,q_i]$ lies in the positivity set of $v(s_0,\cdot)$ far from the free boundary. This is because for any \(y\in [q_0,q_i]\) for some \(i\), we have \(y_1\geq 2R\), and thus $v(s_0,y)\geq y_1-\delta_R(s_0) \geq R > 0$, by choosing $s_0$ large so that $\delta_R(s_0) \leq R$.
Let
\begin{equation}\label{def:extension-in-time}
    \tau_*:=\sup\{\tau \geq s_0:\ [q_0,q_i]\subset\{v(s,\cdot)>0\}
    \text{ for every }s\in[s_0,\tau]\text{ and }i=1,\ldots,n\}.
\end{equation}
The lower bound \(v(s_0,y)\geq R\) on the segments, together with continuity in time, clearly implies \(\tau_*>s_0\). Moreover, the definition of \eqref{def:extension-in-time} implies that each line segment $[q_0,q_i]$ satisfies the hypothesis \eqref{hyp:positive} on the time window $ [s_1,s_2] \subset [s_0,\tau_*)$, and thus \eqref{eq:spatial-change-v} implies that
\begin{equation}\label{eq:change-in-ai}
    |a_i(s_2 ) - a_i(s_1)| \leq C_R \int_{s_1}^{s_2} \kappa_R(s) \, ds. 
\end{equation} 

In particular, recall that at the initial time $s_0$ we have \eqref{eq:initial-double-plane-close}. This implies that
\begin{equation}\label{eq:initial-a}
    |a(s_0) - e_1| \leq C\delta_R(s_0), 
\end{equation} 
where the constant $C$ depends on $R$ and the dimension $n$.
For any $s \in (s_0,\tau_*)$, choose
\(\hat e(s)\in\mathbb S^{n-1}\) such that
\begin{equation*}
    \|v(s,\cdot)-|\langle y,\hat e(s)\rangle|\|_{L^\infty(B_{4R})}
    \leq \frac32 \delta_R(s).
\end{equation*}
Without loss of generality we may assume that $\langle \hat e(s), e_1 \rangle \geq 0$. (If not, simply replace $\hat e(s)$ by $-\hat e(s)$.) Combining
\begin{equation}\label{tmp:small-in-e1}
     \left|a_1(s) - (|\langle q_1, \hat e(s) \rangle| - |\langle q_0, \hat e(s) \rangle| ) \right| \leq 3\delta_R(s), 
\end{equation}
\begin{equation}\label{tmp:proj-in-e1}
    |\langle q_1, \hat e(s) \rangle| - |\langle q_0, \hat e(s) \rangle| = (2R+1)\langle \hat e(s), e_1 \rangle - 2R \langle \hat e(s), e_1 \rangle = \langle \hat e(s), e_1 \rangle,  
\end{equation} 
and \eqref{eq:change-in-ai}, \eqref{eq:initial-a}, we know that
\begin{align}
    |\langle \hat e(s), e_1 \rangle - 1| & \leq |\langle \hat e(s), e_1 \rangle - a_1(s)| + |a_1(s) - a_1(s_0)| + |a_1(s_0)-1| \nonumber \\
    & \leq 3\delta_R(s) + C_R \int_{s_0}^s \kappa_R(s)\, ds + \delta_R(s_0).\label{eq:inner-product-e1} 
\end{align} 
By choosing $s_0$ sufficiently large, we know that $\delta_R(s_0), \delta_R(s)$ are sufficiently small by \eqref{eq:Z-trapping-simple}, and $\int_{s_0}^s \kappa_R(s)\, ds$ is also sufficiently small by \eqref{eq:kappa-dini-simple}. Thus \eqref{eq:inner-product-e1} implies that $\langle \hat e(s), e_1 \rangle \geq \frac{17}{18}$ for all $s\in [s_0,\tau_*)$. In particular
\begin{equation}\label{tmp:new-direction}
    |\hat e(s) - e_1|^2 = 2-2\langle \hat e(s), e_1 \rangle \leq \frac19. 
\end{equation} 
Recall that for any $i$ and $y\in [q_0,q_i]$, we have $y_1\geq 2R$ and $|y|\leq 3R$. Hence \eqref{tmp:new-direction} implies
\[
    \langle y,\hat e(s)\rangle
    =y_1+\langle y,\hat e(s)-e_1\rangle
    \geq 2R-3R|\hat e(s)-e_1|
    \geq R > 0.
\]
As a consequence,
\[ v(s,y) \geq \langle y,\hat e(s)\rangle - \frac32 \delta_R(s) \geq \frac12 R, \]
and $v(\tau_*,y) = \lim_{s\to \tau_* -} v(s,y) \geq R/2 > 0$. 
If $\tau_*< +\infty$, then the continuity in time implies that the defining property for $\tau_*$ holds slightly above $\tau_*$, which is a contradiction. Hence $\tau_* = +\infty$.

Since the estimates \eqref{eq:change-in-ai} hold on $[s_0, +\infty)$, the condition \eqref{eq:kappa-dini-simple} implies that each $a_i(s)$ has a finite limit as \(s\to\infty\). Denote $e:=\lim_{s\to\infty}a(s) \in \RR^n$.
We claim that $\hat e(s)$ also converges to $e$ as $s\to \infty$.
We have seen that \eqref{tmp:small-in-e1} and \eqref{tmp:proj-in-e1} combined gives 
\begin{equation}\label{tmp:CVa1}
    |a_1(s) - \langle \hat e(s), e_1 \rangle |\leq 3\delta_R(s) \to 0. 
\end{equation} 
For each $i\in \{2, \cdots, n\}$, we also have
\[ \langle q_i, \hat e(s) \rangle = 2R \langle \hat e(s) , e_1 \rangle + \langle \hat e(s), e_i \rangle \geq 2R \cdot \frac{17}{18} - 1 > 0, \]
and thus we can remove the absolute values and obtain
\[ |\langle q_i, \hat e(s) \rangle| - |\langle q_0, \hat e(s) \rangle| = \langle q_i, \hat e(s) \rangle - \langle q_0, \hat e(s) \rangle = \langle \hat e(s), e_i \rangle. \]
This immediately implies
\begin{equation}\label{tmp:CVai}
    |a_i(s) - \langle \hat e(s), e_i \rangle | \leq 3\delta_R(s) \to 0, \qquad \text{ for each } i\in \{2, \cdots, n\}.
\end{equation} 
Combining \eqref{tmp:CVa1} and \eqref{tmp:CVai} we conclude that  $\hat e(s) \to e $ as $s\to \infty$. In particular $e\in \mathbb{S}^{n-1}$, $\langle e, e_1 \rangle \geq \frac{17}{18}$ and
\[ |\langle y, \hat e(s) \rangle | \to | \langle y, e \rangle |, \qquad \text{ uniformly in $B_{4R}$ as } s\to \infty.  \]
It follows that
\[ \|v(s,\cdot) - |\langle y, e \rangle | \|_{L^\infty(B_{4R})} \leq  2\delta_R(s) + \| |\langle y, \hat e(s) \rangle |- |\langle y, e \rangle | \|_{L^\infty(B_{4R})} \longrightarrow 0, \qquad \text{ as } s\to \infty.  \]
A priori, the unit vector $e$ depends on $R$; but the limit of $v(s,\cdot)$ in a larger ball clearly agrees with the limit in a smaller ball, so $e$ remains the same for all radii $R\geq 10$. In other words, we have proven that there is a unique $e\in \mathbb{S}^{n-1}/\mathbb{Z}_2$ such that
\[
    v(s,\cdot)\to |\langle y,e\rangle|
    \qquad\text{locally uniformly in $\RR^n$ as }s\to\infty.
\]

\subsection{Remarks on hypotheses}\label{sec:rmk-hp}

We indicate where the assumption that $u$ is classical enters the proof and how it might be weakened.
\begin{itemize}
    \item The global Lipschitz bound on $u$ is only used to deduce the $1/2$-H\"older regularity in time, and both are used to show the density $\theta(r)$ is bounded from below as $r\to 0+$. This can be replaced by some weak growth assumption on $u$ near infinity in space and local $1/2$-H\"older regularity in time.
    
    \item The property \eqref{eq:fbnear0} follows from the continuous movement of the free boundary, but it can be weakened to a parabolic corkscrew condition, such as: There exists a constant $K>0$ such that for any $r$ sufficiently small,
    \[ \partial\{u(t,\cdot)>0\} \cap B_{Kr} \neq \emptyset, \qquad \text{ for some } t\in (-4r^2,-r^2). \]
\end{itemize}

The weak solution theories of \cite{CV, Weiss99, Weiss} begin with the singular limit of the regularized semilinear equations $u_t - \Delta u = -\beta_\epsilon(u)$, where $\beta_\epsilon(s) := \frac{1}{\epsilon} \beta(\frac{s}{\epsilon})$ and $\beta$ is positive on $(0,1)$, vanishes elsewhere, and satisfies $\int \beta(s)\, ds = \frac12$. In \cite{CV} the authors construct weak solutions to \eqref{eq:pBern} assuming the initial condition $u_0$ is strictly mean concave in its support, which implies that the support of the solution $u(t)$ is receding with time; in particular, traveling wave solutions (see Example \ref{eg:exp}) are excluded. 

Under mild regularity and growth assumptions, Weiss \cite{Weiss99} proved the closure under the blow-up process for domain-variation solutions satisfying the scale-invariant clearing condition
\begin{equation}\label{eq:weiss-clearing}
    \sup_{Q_r(t,x)}u\leq cr \quad\Longrightarrow\quad u=0\quad\text{in }Q_{r/2}(t,x)
\end{equation}
for all sufficiently small parabolic cylinders in a fixed compact subset of $(0,\infty) \times \mathbb{R}^n$. 

To expand the class of admissible solutions, Weiss \cite{Weiss} studies pairs $(u,\chi)$, where $u$ is a domain-variation solution that exists globally in time, and $\chi$ satisfies $\chi \in \{0,1\}$ a.e.\ and $\chi_{\{u>0\}} \leq \chi$; see \cite[Definition 6.1]{Weiss} for the precise definition\footnote{We remark that in item 7) a) of the definition, the right hand side of the inequality is missing a square root, see also the proof of \cite[Lemma 5.1]{Weiss}; this is nevertheless harmless for the arguments of \cite{Weiss}.}, and see \cite{KW25} for analogous theory in the elliptic setting. The flexibility of permitting $\chi \neq \chi_{\{u>0\}}$ introduces technicalities to the analysis, but it also allows the author to show that the class of pairs $(u,\chi)$ is closed with respect to the blow-up process, without imposing the clearing condition \eqref{eq:weiss-clearing}. In both works, the closure of the class of weak solutions under blow-up enables the use of geometric measure theory tools, in particular the dimension reduction argument; as a consequence, the author shows that for almost every time $t$, the free boundary $\partial\{u(t)>0\}$ consists of an $(n-1)$-dimensional rectifiable set and a degenerate singular set. 

Technical details aside, one of the main differences between the work of Weiss and our blow-up procedure above lies in whether the goal is to prove \emph{partial regularity} of the free boundary of weak solutions, or to analyze the \emph{formation of singularities}. For example, in \cite{Weiss}, \emph{horizontal points} (points at which the solution's behavior in the time direction is dominant) can be ignored since, on almost every time slice, they belong to a low-dimensional set in space (see \cite[Lemmas 7.2 and 11.3]{Weiss}). However, these are precisely the points where extinction, collision, or a change of topology takes place. Moreover, to study how singularities form and how the solution behaves past them, we need to better understand the possible singularity models -- knowing that the blow-up limits are backward self-similar variational solutions is not enough. In contrast, we have shown above that solutions near Type I singular points are modeled by self-shrinkers whose profile solves the equation in \eqref{eq:profile} on its positivity set and attains the Bernoulli condition classically on each sheet of the free boundary, and we have also identified a critical threshold: $C_s(r)\to0$ forces the profile to be a double plane.

\section{Classification of radial self-shrinkers}\label{sec:radial}

We next prove the classification stated in Theorem \ref{thm:radial}.
Assuming radial symmetry, the profile equation \eqref{eq:profile} reduces to the ODE \eqref{eq:ode}.
After a change of variables $s = r^2/4$, the function $F(s) := \phi(2\sqrt{s})$ satisfies Kummer's (confluent hypergeometric) equation \eqref{eq:K}, which has been studied extensively. We highlight the properties of its solutions below.

\begin{lemma}\label{lem:MU}
The confluent hypergeometric equation
\begin{equation}\label{eq:K}
sF''(s) + \left(\frac{n}{2}-s \right)F'(s) + \frac{1}{2}\,F(s) = 0, \qquad s > 0,\tag{K}
\end{equation}
admits a pair of solutions $M(s)=M(-\tfrac{1}{2},\tfrac{n}{2},s)$ and $U(s)=U(-\tfrac{1}{2},\tfrac{n}{2},s)$ with the following properties:
\begin{enumerate}
\item[\textup{(i)}] $M$ and $U$ are linearly independent, with Wronskian
\[
W(s) := M(s)\,U'(s) - M'(s)\,U(s) = 
\frac{\Gamma(n/2)}{2\sqrt{\pi}}\,\frac{e^s}{s^{n/2}}\,;
\]
\item[\textup{(ii)}] $M'(s) = -\frac1n M(\tfrac{1}{2},\tfrac{n}{2}+1,s) < 0$ and $U'(s) = \frac12 U(\frac12, \frac{n}{2}+1,s) > 0$ for all $s > 0$;
\item[\textup{(iii)}] $U(s) = \sqrt{s}$ when $n=1$; for any $n>1$, the function $U$ has a unique zero $s_* > 0$, with $U < 0$ on $(0,s_*)$ and $U > 0$ on $(s_*,\infty)$;
\item[\textup{(iv)}] $M$ has a unique zero $s_0$, and $s_0 > s_*$ if $n>1$;
\item[\textup{(v)}] $M(0) = 1$ and as $s \downarrow 0$,
\[ U(s) \sim \left\{\begin{array}{ll}
    -\frac{\Gamma(n/2-1)}{2\sqrt{\pi}}\,s^{1-n/2}, & n \neq 2 \\
    \frac{1}{2\sqrt{\pi}}\,\log s, & n=2;
\end{array} 
\right. \]
\item[\textup{(vi)}] as $s \to \infty$, $U(s) \sim \sqrt{s}$ and $M(s) \sim -\frac{\Gamma(n/2)}{2\sqrt{\pi}}\,e^{s}\,s^{-(n+1)/2}$.
\end{enumerate}
\end{lemma}
\begin{proof}
    Most of the properties are standard, see \cite[{\S13.2, 13.3, 13.7}]{DLMF}; the Wronskian is \cite[(13.2.34)]{DLMF} with $\Gamma(-\tfrac12)=-2\sqrt{\pi}$. The uniqueness of zeros follows from the monotonicity of $M, U$ and their asymptotics at $0$ and $\infty$. Note that $U(s) \to 0$ for $n=1$, and $U(s) \to -\infty$ for any $n>1$ as $s\to 0$. Hence $U(s)$ remains positive when $n=1$, and it has a unique zero when $n>1$. On the other hand, since the Wronskian $W(s)$ is always positive, evaluating at $s=s_0$ we get that $U(s_0)>0$. Thus $s_0>s_*$.
\end{proof}
We now prove some additional properties of the fundamental pair of solutions, which will be used in the proof of uniqueness and spectral analysis of the annular solution.

    \begin{lemma}\label{lem:MUroots}
Let $U(s)=U(-\tfrac12,\tfrac n2,s)$ and $M(s)=M(-\tfrac12,\tfrac n2,s)$ be the fundamental pair from
Lemma~\ref{lem:MU}. Let $s_*$ be the unique zero of $U$, and $s_0>s_*$ be the unique zero of $M$.
Then
\[
s_*<\frac{n-1}{2}<\frac n2<s_0 .
\]
\end{lemma}

\begin{proof}
We first prove that $s_*<\frac{n-1}{2}$ by showing that $U\!\left(\frac{n-1}{2}\right)>0$.
By the contiguous relation \cite[Eq.~13.3.7]{DLMF} (specialized to $a=\tfrac12$, $b=\tfrac n2$),
\[
U\!\left(-\tfrac12,\tfrac n2,s\right)
=\Bigl(s+1-\tfrac n2\Bigr)\,U\!\left(\tfrac12,\tfrac n2,s\right)
+\frac{\tfrac n2-\tfrac32}{2}\,U\!\left(\tfrac32,\tfrac n2,s\right).
\]
Moreover, for $a>0$ we have the integral representation \cite[Eq.~13.4.4]{DLMF}
\[
U(a,\tfrac n2,s)=\frac{1}{\Gamma(a)}\int_{0}^{\infty}e^{-st}\,t^{a-1}(1+t)^{\,\tfrac n2-a-1}\,dt,
\qquad s>0.
\]
Setting $s=\frac{n-1}{2}$ gives $s+1-\tfrac n2=\tfrac12$, hence
\begin{align*}
U\!\left(\tfrac{n-1}{2}\right)
&=\frac{1}{2\Gamma(\tfrac12)}\int_{0}^{\infty}e^{-\frac{n-1}{2}t}\,t^{-1/2}(1+t)^{(n-3)/2}\,dt \\
&\quad+\frac{\tfrac n2-\tfrac32}{2\Gamma(\tfrac32)}
\int_{0}^{\infty}e^{-\frac{n-1}{2}t}\,t^{1/2}(1+t)^{(n-5)/2}\,dt .
\end{align*}
If $n\ge 3$, both coefficients are nonnegative and the first integral is strictly positive, so
$U\!\left(\frac{n-1}{2}\right)>0$. If $n=2$, using $\Gamma(\tfrac32)=\Gamma(\tfrac12)/2$ we obtain
\begin{align*}
    U \left(\tfrac12\right) & = \frac{1}{2\Gamma(\tfrac12)} \int_{0}^{\infty} e^{-t/2} \Bigl( t^{-1/2}(1+t)^{-1/2}-t^{1/2}(1+t)^{-3/2} \Bigr)\,dt \\
    & =\frac{1}{2\Gamma(\tfrac12)} \int_{0}^{\infty} e^{-t/2}\, t^{-1/2}(1+t)^{-3/2} \,dt>0.
\end{align*}
In either case, $U\!\left(\frac{n-1}{2}\right)>0$, and since $U$ changes sign only at $s_*$ with
$U>0$ on $(s_*,\infty)$ (Lemma~\ref{lem:MU}(iii)), it follows that $s_*<\frac{n-1}{2}$.

Next we prove $s_0>\frac n2$ by showing $M(\tfrac n2)>0$.
Using the series \cite[Eq.~13.2.2]{DLMF},
\[
M(s)=M\!\left(-\tfrac12,\tfrac n2,s\right)
=1+\sum_{k=1}^{\infty}\frac{(-\tfrac12)_k}{(\tfrac n2)_k}\frac{s^k}{k!},
\]
where $(x)_k:= x(x+1) \cdots (x+k-1)$ is the Pochhammer symbol.
Setting $s=\tfrac n2$ and using $(-\tfrac12)_k=-\tfrac12( \tfrac12)_{k-1}$ for $k\ge1$ yields
\[
\sum_{k=1}^{\infty}\frac{(-\tfrac12)_k}{(\tfrac n2)_k}\frac{(\tfrac n2)^k}{k!}
=-\frac12\sum_{k=1}^{\infty}\frac{(\tfrac12)_{k-1}}{k!}\,\frac{(\tfrac n2)^k}{(\tfrac n2)_k}
\ge -\frac12\sum_{k=1}^{\infty}\frac{(\tfrac12)_{k-1}}{k!},
\]
since $(\tfrac n2)_k=(\tfrac n2)(\tfrac n2+1)\cdots(\tfrac n2+k-1)\ge(\tfrac n2)^k$ for $n\ge2$.
Now
\[
\frac{(\tfrac12)_{k-1}}{k!}
=\frac{(2k-2)!}{4^{k-1}(k-1)!\,k!}
=\frac{1}{4^{k-1}}\cdot\frac{1}{k}\binom{2k-2}{k-1}
=\frac{C_{k-1}}{4^{k-1}},
\]
where $C_j=\frac{1}{j+1}\binom{2j}{j}$ is the $j$-th Catalan number. Using the generating function
$\sum_{j=0}^{\infty}C_j x^j=\frac{1-\sqrt{1-4x}}{2x}$ and evaluating at $x=\tfrac14$ gives
\[
\sum_{k=1}^{\infty}\frac{(\tfrac12)_{k-1}}{k!}=\sum_{j=0}^{\infty}\frac{C_j}{4^j}=2.
\]
Therefore $M(\tfrac n2)\ge 1-\frac12\cdot 2=0$. Moreover, since $(\tfrac n2)_k>(\tfrac n2)^k$ for
every $k\ge2$, the above inequality is strict, hence $M(\tfrac n2)>0$.
Finally, $M'<0$ and $M(s_0)=0$ (Lemma~\ref{lem:MU}(ii),(iv)), so $M(\tfrac n2)>0$ implies $\tfrac n2<s_0$.
\end{proof}

The following result establishes the existence and uniqueness of a radial solution of~\eqref{eq:profile} compactly supported on an annulus.

\begin{proposition}\label{prop:odesol}
If $n>1$, there exist radii $0 < r_- < r_+$ and a function $\phi \in C^2((r_-,r_+)) \cap C^1([r_-,r_+])$ such that:
\begin{enumerate}
\item[\textup{(i)}] $\phi > 0$ on $(r_-,r_+)$ and $\phi(r_-) = \phi(r_+) = 0$;
\item[\textup{(ii)}] $\phi$ solves the radial profile ODE \eqref{eq:ode}, i.e.
\[
\phi''(r) + \left(\frac{n-1}{r} - \frac{r}{2}\right)\phi'(r) + \frac{1}{2}\,\phi(r) = 0, \qquad r \in (r_-,r_+);
\]
\item[\textup{(iii)}] the Bernoulli conditions hold: $\phi'(r_-) = 1$ and $\phi'(r_+) = -1$.
\end{enumerate}
The function defined by $\phi(y) := \phi(|y|)$ for $y$ in the annulus $\{r_- < |y| < r_+\}$ and $\phi \equiv 0$ otherwise is the unique radial solution of~\eqref{eq:profile} compactly supported on the annulus.
\end{proposition}

\begin{proof}
{\bf Existence.} Let $M,U$ be the fundamental pair from Lemma~\ref{lem:MU}, with unique zeros $s_*$ of $U$ and $s_0$ of $M$, and $s_0> s_*$. The ratio
\begin{equation} \label{eq:R'sign}
R(s):=\frac{M(s)}{U(s)}\quad\text{satisfies}\quad R'(s)=-\frac{W(s)}{U(s)^2}=-\frac{\Gamma(n/2)\,e^s}{2\sqrt{\pi}\,s^{n/2}\,U(s)^2}<0
\end{equation}
wherever $U\neq0$. Note that in the special case $n=1$, since $U(s)=\sqrt{s}>0$, $R$ is strictly monotone and thus no linear combination of $M$ and $U$ can vanish at two different points. 
Assume $n>1$. Since $M(s_*)>0$ and $U$ changes sign at $s_*$, we have $R(s_*^-)=-\infty$ and $R(s_*^+)=+\infty$. The asymptotics in Lemma~\ref{lem:MU}(v) (vi) give $R(s)\to 0^-$ as $s\downarrow 0$ and $R(s)\to-\infty$ as $s\to\infty$.

For $\alpha\in(0,s_*)$, set $c:=R(\alpha)<0$ and $F_\alpha(s):=M(s)-cU(s)$. Then $F_\alpha(s)=0$ iff $R(s)=c$ on $\{U\neq0\}$, so by strict monotonicity of $R$ on $(s_*,\infty)$, there is a unique $\beta(\alpha)\in(s_0,\infty)$ with $F_\alpha(\beta)=0$, and $F_\alpha>0$ on $(\alpha,\beta)$ (since $F_\alpha(s_*)=M(s_*)>0$).

Set $\phi(r)=A\,F_\alpha(r^2/4)$ on $[2\sqrt{\alpha},\,2\sqrt{\beta}]$, with $A>0$ chosen so that $\phi'(2\sqrt{\alpha})=1$. At a zero $s$ of $F_\alpha$, one has $c=M(s)/U(s)$ and hence
$F_\alpha'(s)=M'(s)-cU'(s)=\frac{M'(s)U(s)-M(s)U'(s)}{U(s)}=-\frac{W(s)}{U(s)}=-\frac{\Gamma(n/2)\,e^s}{2\sqrt{\pi}\,s^{n/2}\,U(s)}$.
By the chain rule, $\phi'(r)=A\,F_\alpha'(r^2/4)\cdot r/2$, so at a zero $s$ of $F_\alpha$,
\[
\phi'(2\sqrt{s})=A\,F_\alpha'(s)\,\sqrt{s}=-\frac{A\,\Gamma(n/2)\,e^s}{2\sqrt{\pi}\,s^{(n-1)/2}\,U(s)}.
\]
The condition $\phi'(2\sqrt{\alpha})=1$ gives
$A=-\frac{2\sqrt{\pi}\,\alpha^{(n-1)/2}\,U(\alpha)\,e^{-\alpha}}{\Gamma(n/2)}$, so
\[
p(\alpha):=\phi'(2\sqrt{\beta})=-\frac{A\,\Gamma(n/2)\,e^\beta}{2\sqrt{\pi}\,\beta^{(n-1)/2}\,U(\beta)}=e^{\beta-\alpha}\!\left(\frac{\alpha}{\beta}\right)^{\!(n-1)/2}\!\frac{U(\alpha)}{U(\beta)}<0.
\]
Continuity of $p$ follows from the implicit relation $R(\beta)=R(\alpha)$ and $R'<0$. As $\alpha\downarrow0$, Lemma~\ref{lem:MU}(v) implies $\alpha^{(n-1)/2}\,|U(\alpha)|\to0$; while $R(\alpha)\to0^-$ forces $\beta(\alpha)\to s_0$, hence $p(\alpha)\to0$. As $\alpha\uparrow s_*$, $R(\alpha)\to-\infty$ implies $\beta(\alpha)\to\infty$, and using $U(\alpha)/U(\beta)=M(\alpha)/M(\beta)$ together with $M(s)\sim -\frac{\Gamma(n/2)}{2\sqrt{\pi}}\,e^{s}\,s^{-(n+1)/2}$ gives $p(\alpha)\to-\infty$.

By the intermediate value theorem, there exists $\alpha_0 \in (0,s_*)$ with $p(\alpha_0) = -1$. Setting $r_- := 2\sqrt{\alpha_0}$, $r_+ := 2\sqrt{\beta(\alpha_0)}$, and $\phi(r) := A\,F_{\alpha_0}(r^2/4)$ on $[r_-,r_+]$, extended by zero outside, we obtain a solution of~\eqref{eq:ode} with $\phi(r_\pm)=0$, $\phi > 0$ on $(r_-,r_+)$, $\phi'(r_-)=1$, and $\phi'(r_+)=-1$. Thus $\phi(y) = \phi(|y|)$ is the desired annular solution of~\eqref{eq:profile}.
\bigskip

{\bf Uniqueness.} Firstly, we claim that the desired solution must be a constant multiple of $M+cU$ for some $c>0$. By Lemma \ref{lem:MU} (iii) (iv), a constant multiple of $M$ or $U$ alone has only one root. Besides, $M+cU$ has two roots if and only if there exist $0<\alpha<\beta$ such that $\frac{M(\alpha)}{U(\alpha)} =-c = \frac{M(\beta)}{U(\beta)}$.
By the properties of $R(s)=M(s)/U(s)$, this can only occur when $c>0$, $\alpha\in (0,s_*)$ and $\beta \in (s_0, \infty)$.
On the other hand, for each $c>0$ there are exactly two solutions $\alpha(c)<\beta(c)$ to $R(s) = -c$; and these two solutions satisfy $\alpha(c)< s_* < s_0 < \beta(c)$.
In other words, any solution with two roots is a positive multiple of $F_\alpha$ for some $\alpha\in(0,s_*)$. 
Moreover, after the normalization $\phi(r)=A\,F_\alpha(r^2/4)$ such that $\phi'(r_-)=1$, the Bernoulli conditions hold exactly when $p(\alpha)=-1$. Therefore to prove uniqueness, it suffices to show that the shooting slope $p$ is strictly decreasing on $(0,s_*)$.

For $\alpha\in(0,s_*)$, set
\begin{equation}\label{tmp:annular-scaled-notation}
    a(\alpha):=\sqrt{2\alpha},
 \qquad
 b(\alpha):=\sqrt{2\beta(\alpha)},
 \qquad
 y_\alpha(t):=\frac1{\sqrt2}\,\phi\bigl(\sqrt2\,t\bigr).
\end{equation}
By \eqref{eq:ode}, the rescaled profile $y_\alpha$ satisfies the damped oscillator equation
\begin{equation}\label{eq:y-alpha}
    y_\alpha''-\left(t-\frac{n-1}{t}\right)y_\alpha'+y_\alpha=0
 \qquad\text{on }(a(\alpha),b(\alpha)),
\end{equation}
with
\[
 y_\alpha>0\ \text{ on }(a(\alpha),b(\alpha)),
 \qquad
 y_\alpha'(a(\alpha))=1,
 \qquad
 y_\alpha'(b(\alpha))=p(\alpha)<0,
\]
and Lemma~\ref{lem:MUroots} gives
\[
 0<a(\alpha)<\sqrt{n-1}<\sqrt n<b(\alpha).
\]
For the remainder of the proof we fix $\alpha$ and write $a=a(\alpha)$, $b=b(\alpha)$, $y=y_\alpha$.

\begin{claim}\label{cl:concave}
    We first check that $y''<0$ on $(a,b)$. In particular, the annular profile is concave.
\end{claim}
\noindent \emph{Proof of Claim \ref{cl:concave}.}
Differentiating the equation shows that $K:=-y''$ satisfies
\begin{equation}\label{eq:y-2ndder}
    K''+\left(\frac{n+1}{t}-t\right)K'-3K=0,
\end{equation}
while evaluating the equation at the endpoints gives
\[
 K(a)=\frac{n-1}{a}-a>0,
 \qquad
 K(b)=-\left(b-\frac{n-1}{b}\right)p(\alpha)>0,
\]
since $a<\sqrt{n-1}<b$. As the zeroth-order coefficient is negative, the maximum principle gives $K>0$ on $(a,b)$.

We use a phase-plane argument to prove the monotinicity of $p(\alpha)$. For solutions to \eqref{eq:y-alpha}, set
\[
 E:=y'^2+y^2, \qquad h(t):=\frac12\left(t-\frac{n-1}{t}\right),
\]
and introduce the phase
\[
 y'=\sqrt E\cos\theta,
 \qquad
 y=\sqrt E\sin\theta,
 \qquad \text{ with initial phase }
 \theta(a)=0.
\]
Since $y>0$ and $y''<0$ on $(a,b)$,
\[
 \theta'=\frac{y'^2-yy''}{E}>0,
\]
so $t$ may be regarded as a function $t_\alpha(\theta)$ on $[0,\pi]$, with $t_\alpha(0)=a$ and $t_\alpha(\pi)=b$. Using the equation \eqref{eq:y-alpha} to eliminate $y''$, the same computation gives $\theta'=1-h(t)\sin(2\theta)$, while $E'=4h(t)\,y'^2$ gives $(\log E)'=4h(t)\cos^2\theta$. Hence
\begin{equation}\label{eq:annular-shooting-phase}
 \frac{dt_\alpha}{d\theta}
 =\frac1{1-h(t_\alpha)\sin(2\theta)},
 \qquad
 \frac{d}{d\theta}\log E
 =\frac{4h(t_\alpha)\cos^2\theta}{1-h(t_\alpha)\sin(2\theta)},
\end{equation}
and the positivity of $\theta'$ means that the common denominator is positive along every trajectory.

Let now $0< \alpha_1<\alpha_2< s_*$. The functions $t_{\alpha_1}$ and $t_{\alpha_2}$ solve the same scalar equation in \eqref{eq:annular-shooting-phase}, whose right-hand side is continuously differentiable in $t$ wherever the denominator is positive, and
$t_{\alpha_1}(0)=a(\alpha_1)<a(\alpha_2)=t_{\alpha_2}(0)$. Two solutions cannot cross by ODE uniqueness, so
\[
 t_{\alpha_1}(\theta)<t_{\alpha_2}(\theta)
 \qquad(0\leq\theta\leq\pi);
\]
in particular, $b(\alpha)$ is strictly increasing with $\alpha$.

Since $E(a)=1$ and $E(b)=p(\alpha)^2$, integrating the second equation in \eqref{eq:annular-shooting-phase} gives
\[
 \log p(\alpha)^2
 =\int_0^\pi
 \frac{4h(t_\alpha(\theta))\cos^2\theta}
 {1-h(t_\alpha(\theta))\sin(2\theta)}\,d\theta.
\]
Fix $\theta\in [0,\pi]$ and consider the integrand as a function of $t$ between the two trajectories. The denominator is monotone in $h(t)$ and $h$ is increasing, so the denominator lies between its values at $t_{\alpha_1}(\theta)$ and $t_{\alpha_2}(\theta)$; in particular it is positive there. Moreover, the integrand satisfies
\[
 \frac{\partial}{\partial t}\,
 \frac{4h(t)\cos^2\theta}{1-h(t)\sin(2\theta)}
 =\frac{4h'(t)\cos^2\theta}{\bigl(1-h(t)\sin(2\theta)\bigr)^2} 
 \geq0,
 \qquad
 h'(t)=\frac12+\frac{n-1}{2t^2}>0,
\]
with strict inequality except at $\theta=\pi/2$. Hence $p(\alpha)^2$ is strictly increasing, and since $p<0$, the slope $p(\alpha)$ is strictly decreasing:
\begin{equation}\label{eq:annular-shooting-monotonicity}
 \alpha_1<\alpha_2
 \quad\Longrightarrow\quad
 b(\alpha_1)<b(\alpha_2)
 \quad\text{and}\quad
 p(\alpha_1)>p(\alpha_2).
\end{equation}
In particular, $p(\alpha)=-1$ for exactly one $\alpha\in(0,s_*)$, which completes the proof of uniqueness.

\end{proof}

\begin{proof}[Proof of Theorem \ref{thm:radial}]
    Given a radial solution $\phi $ to~\eqref{eq:profile}, 
its radial representative $\phi (r)$ solves~\eqref{eq:ode}.
If the support of $\phi (r)$ is $\left[ a, b \right]$
for some $0 < a < b < \infty$, 
we must have that $\phi $ is the solution described in Proposition~\ref{prop:odesol}. 
In what follows, let $F(s) = \phi \left( 2\sqrt s \right)$
as in the proof of Proposition~\ref{prop:odesol}. 

If $\phi $ is supported on a ball containing the origin, 
then $F$ has a finite value near $0$. 
But expressing $F(s) = aU(s) + b M(s)$, if $n>1$ we must have $a = 0$, 
as $U$ is unbounded as $s\to 0$ while $M$ is bounded by Lemma~\ref{lem:MU}. In the special case $n=1$, we still express $F(s)=a\sqrt{s}+bM(s)$. Recall that $F$ satisfies the condition $\lim_{s\to0} \sqrt{s}F'(s) = \phi'(0)= 0$. Combined with $M'(0) \in \RR$ this also implies $a=0$.  
Thus $F(s) = bM(s)$. In particular $F \left( s_0 \right) = 0$, 
and the size of the support is determined. The value of $b$ is also uniquely determined, so that
\[ \phi'\left( 2\sqrt{s_0} \right) = \sqrt{s_0} F'(s_0) = b\cdot \sqrt{s_0} M'(s_0) = -1. \]

If the positivity set of $F $ is an unbounded interval $[s,\infty)$, 
then $F(s) = 0$ implies $R(s)= \frac{M(s)}{U(s)} = -\frac{a}{b}$. 
If $s < s_*$, then by the proof of Proposition~\ref{prop:odesol} 
we know $R(s) < 0$, 
and there is a second point $s' > s$ such that $R\left( s' \right) = R(s)$.
This contradicts the positivity set of $F$ being $[s,\infty)$. 
Thus $s \geq s_*$. 
When $s = s_*$, $U\left( s_* \right) = 0$ forces $b = 0$, and thus $F(s) = aU(s)$ for a unique value of $a$ such that it satisfies the Bernoulli boundary condition at $2\sqrt{s_*}$.
In any cases, the claimed asymptotic growth of $\phi $ 
follows from Lemma~\ref{lem:MU} (vi). 

\end{proof}

Lastly, we note the following consequence of the construction and Lemma~\ref{lem:MUroots}, which locates the annulus relative to the zeros $s_*$ and $s_0$.
\begin{proposition}\label{prop:annulus-containment}
Let $\phi$ be the annular profile from Proposition~\ref{prop:odesol}, with support $[r_-,r_+]$. Then
\begin{equation}\label{eq:annulus-containment}
    r_-^2 < 4s_* < 2(n-1) < 2n < 4s_0 < r_+^2.
\end{equation}
In particular, the annulus $\{r_-<|y|<r_+\}$ contains the spheres of radius $\sqrt{2(n-1)}$ and $\sqrt{2n}$.
\end{proposition}
\begin{proof}
In the construction in Proposition~\ref{prop:odesol}, $r_-=2\sqrt{\alpha_0}$ and $r_+=2\sqrt{\beta(\alpha_0)}$ with $\alpha_0\in(0,s_*)$ and $\beta(\alpha_0)\in(s_0,\infty)$, which gives the two outer inequalities. The middle inequalities are Lemma~\ref{lem:MUroots}.
\end{proof}

The finer description of the annulus stated in Theorem \ref{thm:annulus-geometry} --- the two-sided bounds for the radii \eqref{eq:fine-radius-estimate}, the width and product bounds, and the outward bias of the midpoint --- is proved in Appendix \ref{sec:radius-est}, see Propositions \ref{prop:annulus-localization} and \ref{prop:annulus-width} and Lemmas \ref{lem:annular-phase} and \ref{lem:annular-midpoint}. The classification of the modes in Section \ref{sec:instability} depends crucially on these estimates.

\section{Spectral analysis of the annular self-shrinker}\label{sec:instability}
In \cite[Section 2]{stability}, the authors derived the linearized operator for the evolution equation \eqref{eq:respar} at an equilibrium $\phi$, assuming $\Omega :=\{\phi>0\}$ is bounded, which is
\begin{equation}\label{eq:linear-stability}
    \left\{\begin{array}{ll}
       \mathscr{L} v := \Delta v - \frac{1}{2} y \cdot \nabla v + \frac{1}{2} v = 0,  & \text{ in } \Omega,  \\[5pt]
       \mathscr{B} v := \frac{\partial v}{\partial n} - 
    	\frac{\partial^2 \phi}{\partial n^2} v = 0,  & \text{ on } \partial\Omega.
    \end{array} \right.\tag{L}
\end{equation}
Here \(\partial/\partial n\) denotes differentiation in the direction of the unit normal pointing into \(\Omega\).  

From now on, let $\phi(r)$ denote the radial annular profile in Proposition \ref{prop:odesol}, and let $A=\{r_-<|y|<r_+ \}$ denote the positivity set of $\phi$. Then $\phi$ is an equilibrium for the evolution equation \eqref{eq:respar}.
By  separation of variables, the operator $\mathscr{L}$ for $v(r\omega) = \psi(r)g(\omega)$ becomes 
\begin{align*}
    \mathscr{L} v & = \frac{\partial^2}{\partial r^2} v + \frac{n-1}{r} \frac{\partial}{\partial r} v + \frac{1}{r^2} \Delta_{\mathbb{S}^{n-1}} v - \frac12 r\omega \cdot \nabla v + \frac12 v \\
    & = \left[\psi''(r) + \left( \frac{n-1}{r} - \frac{r}{2} \right) \psi'(r) + \frac12 \psi(r) \right] g(\omega) + \frac{\psi(r)}{r^2} \Delta_{\mathbb{S}^{n-1}} g(\omega).
\end{align*}
Recall that the eigenfunctions of $\Delta_{\mathbb{S}^{n-1}}$ are fully classified as spherical harmonics, and they form a basis for $L^2(\mathbb{S}^{n-1})$: For each $\ell\in \mathbb{N}\cup\{0\}$, $\mu_{\ell} := \ell(\ell+n-2)$ is an eigenvalue of $-\Delta_{\mathbb{S}^{n-1}}$ with eigenspace of dimension \(d_{n,\ell}=\binom{n+\ell-1}{\ell}-\binom{n+\ell-3}{\ell-2}\).
Therefore it suffices to study the spectrum problem of the following ODE:
\begin{equation}\label{eq:spectrum}
    \left\{\begin{array}{l}
    	\mathscr{L}_{\ell}\psi:=\psi''(r) + \left( \frac{n-1}{r} - \frac{r}{2} \right) \psi'(r) + \left( \frac12 - \frac{\mu_\ell}{r^2} \right) \psi(r) = \lambda \psi(r), \text{ in } (r_-,r_+) \\[5pt]
    	\psi'(r_-) + \left( \frac{n-1}{r_-} - \frac{r_-}{2} \right) \psi(r_-) = 0, \quad \psi'(r_+) + \left( \frac{n-1}{r_+} - \frac{r_+}{2} \right) \psi(r_+) = 0,
    \end{array}  \right.\tag{$L_s$}
\end{equation}
for each $\mu= \mu_\ell$ fixed.
By the Sturm--Liouville theory, \eqref{eq:spectrum} has simple eigenvalues $\lambda_{\ell,1} > \lambda_{\ell,2} > \cdots \to -\infty$, and the eigenfunction of $\lambda_{\ell,k}$ has $(k-1)$ zeros.

The proof of Theorem \ref{prop:spectrum-annular} occupies the remainder of the section: Section~\ref{ssec:modes} identifies the explicit eigenfunctions and proves the monotonicity of $\lambda_{\ell,k}$ in $\ell$ and the row $k\geq3$; Section~\ref{ssec:criteria} records two ground-state criteria, which are in turn used in Section~\ref{ssec:negative} to show $\lambda_{2,2}<0$ and $\lambda_{4,1}<0$; Section~\ref{ssec:l3} proves $\lambda_{3,1}>0$ in every dimension and assembles the table.

\subsection{Explicit eigenfunctions and the monotonicity of eigenvalues}\label{ssec:modes}

To write down the Rayleigh quotient for the spectrum problem \eqref{eq:spectrum}, we rewrite the ODE as follows. Denote
\[ \rho(r) := r^{n-1} e^{-\frac{r^2}{4}}, \qquad \alpha(r) := \tfrac{\rho'(r)}{\rho(r)} = \tfrac{n-1}{r} - \tfrac{r}{2}.  \]
Then
\begin{equation}\label{eq:p-transform}
    \zeta:=\rho\psi
\end{equation}
satisfies the Neumann boundary condition:
\begin{equation}\label{tmp:BC}
	\zeta' = \rho (\psi'+ \alpha \psi) = 0 \qquad \text{ at the endpoints } r=r_-, r_+;
\end{equation}
and simple computation shows that the ODE becomes
\begin{equation}\label{tmp:ODE}
	\left( \frac{\zeta'}{\rho} \right)' + \left(1-\frac{\mu-(n-1)}{r^2} \right) \frac{\zeta}{\rho} = \lambda \frac{\zeta}{\rho}.\tag{$N_s$}
\end{equation}
Thus the Rayleigh quotient is
\begin{equation}\label{eq:rayleigh-p}
    \mathcal{R}_{\mu}(\zeta)= \dfrac{-\int (\zeta')^2 \rho^{-1} \, dr + \int\left(1-\frac{\mu-(n-1)}{r^2} \right) \zeta^2 \rho^{-1} \, dr}{\int \zeta^2 \rho^{-1} \, dr},
\end{equation}
with boundary condition $\zeta'(r_-) = \zeta'(r_+)=0$. In particular, the Robin problem \eqref{eq:spectrum} for $\psi$ and the Neumann problem \eqref{tmp:ODE} for $\zeta=\rho\psi$ have the same eigenvalues $\lambda_{\ell,k}$ (when $\mu$ is chosen to be $\mu_\ell$), and we work with the latter from now on.

\begin{lemma}\label{lem:mode-structure}
The eigenvalues $\lambda_{\ell,k}$ of \eqref{eq:spectrum} satisfy:
\begin{enumerate}[label=\textup{(\roman*)}]
\item For each fixed $k$, the eigenvalue $\lambda_{\ell,k}$ is nonincreasing in $\ell$.
\item For every $\ell\geq0$ and every $k\geq3$, $\lambda_{\ell,k}<0$.
\end{enumerate}
\end{lemma}

\begin{proof}
By the min-max characterization of eigenvalues $\lambda_{\ell,k}$ using the Rayleigh quotient \eqref{eq:rayleigh-p} with $\mu=\mu_\ell$, we know that for each $k$ fixed, $\lambda_{\ell,k}$ is monotone decreasing as $\ell$ and $\mu_{\ell}$ increase.

By (i), it suffices to show $\lambda_{0,k}<0$ for $k\geq3$.
By the Sturm--Liouville theory, when $k\geq 3$ the corresponding eigenfunction $\psi$ has at least two zeros in $(r_-,r_+)$. For any two consecutive zeros $a<b$, $\psi$ is the first Dirichlet eigenfunction for the domain $(a,b)$. Recall that $\phi$ is the first Dirichlet eigenfunction of the same operator (recall that $\ell=0$ and thus $\mu_\ell=0$) in a strictly larger domain $(r_-,r_+)$, with eigenvalue $0$. Hence by the domain monotonicity of first eigenvalues, the eigenvalue for $\psi$ satisfies $\lambda_{0,k} < 0$.
\end{proof}

\begin{lemma}\label{lem:explicit-modes}
Three eigenvalues for the full spectrum problem \eqref{eq:spectrum-full} are explicit:
\begin{enumerate}[label=\textup{(\roman*)}]
\item $\lambda_{1,1}=1$, with eigenfunction $\langle \omega, a \rangle/\rho(r)$ for any $a\in \RR^n$; 
\item $\lambda_{1,2}=\frac12$, with eigenfunction $\langle \omega, a \rangle \phi'(r)$ for any $a\in \RR^n$; 
\item $\lambda_{0,2}=1$, with eigenfunction $\phi(r)-r\phi'(r)$.
\end{enumerate}
\end{lemma}

\begin{remark}\label{rem:geometric-modes}
The geometric origins of these modes are discussed after Corollary \ref{cor:stability} in the introduction. For comparison, when $\phi= \phi_B$ is the ball profile solution, by sign considerations $\phi(r)-r\phi'(r)$ is an eigenfunction corresponding to the index $\lambda_{0,1}$, and $\langle \omega, a \rangle \phi'(r)$ is an eigenfunction corresponding to the index $\lambda_{1,1}$.
\end{remark}

\begin{proof}
In divergence form, the profile equation \eqref{eq:ode} reads
\begin{equation}
\label{eq:radial-profile}
    (\rho\,\phi')'=-\frac12\,\rho\,\phi
    \quad\text{on }A,
    \qquad
    \phi=0
    \quad\text{on }\partial A.
\end{equation}

(i) For any $a\in \RR^n$, $\langle \omega,a \rangle /\rho(r)$ is an eigenfunction of \eqref{eq:spectrum-full}. Let $\psi(r) := 1/\rho(r)$. Since $\zeta:= \rho\psi \equiv 1$, it clearly satisfies the ODE \eqref{tmp:ODE} with $\mu=n-1$ and $\lambda = 1$ and the vanishing Neumann boundary condition \eqref{tmp:BC}. This is equivalent to saying $\psi(r)$ is an eigenfunction to \eqref{eq:spectrum} with $\mu=\mu_1$ and $\lambda=1$.
Moreover, since $\psi=1/\rho>0$, it has to be the first eigenfunction of \eqref{eq:spectrum} (with $\mu=\mu_1$ fixed); and $\omega \mapsto \langle \omega, a\rangle$ is a spherical harmonic with $\ell=1$, so $\lambda = \lambda_{1,1}$.

(ii) The equation \eqref{eq:radial-profile} implies $\zeta_1:=\rho\phi'$ satisfies $\zeta_1'(r_\pm)=-\frac12\rho\phi(r_\pm)=0$, which is the Neumann condition \eqref{tmp:BC}, and
\[
    \Bigl(\frac{\zeta_1'}{\rho}\Bigr)'+\frac{\zeta_1}{\rho}
    =-\frac12\,\phi'+\phi'
    =\frac12\cdot\frac{\zeta_1}{\rho},
\]
which is \eqref{tmp:ODE} with $\mu=\mu_1$ and $\lambda=\frac12$.
We claim that the radial profile $\phi$ has a unique critical point in $(r_-,r_+)$. Since $\phi'(r_-) = 1, \phi'(r_+) = -1$, by the intermediate value theorem $\phi$ must have at least one critical point. At any critical point $r_c$, we have $\phi''(r_c) = -\frac12 \phi(r_c) < 0$. Hence $\phi$ can not have more than one critical point in $(r_-,r_+)$. By the Sturm--Liouville theory, this indicates that $\phi'(r)$ is the second eigenfunction, i.e. $\lambda = \lambda_{1,2}$.

(iii) We now verify that $\psi(r) = \phi(r) - r \phi '(r)$
is an eigenfunction of \eqref{eq:spectrum}
with $\mu = 0$ and $\lambda = 1$. Let $\zeta:=\rho\psi$. By \eqref{eq:radial-profile} and $\rho'=\bigl(\tfrac{n-1}{r}-\tfrac{r}{2}\bigr)\rho$,
\[
    \zeta'
    =\rho'\phi+\rho\phi'-\rho\phi'-r(\rho\phi')'
    =\rho'\phi+\frac{r}{2}\,\rho\,\phi
    =\frac{n-1}{r}\,\rho\,\phi,
\]
so $\zeta'(r_\pm)=0$, which is the Neumann condition \eqref{tmp:BC}, and
\[
    \Bigl(\frac{\zeta'}{\rho}\Bigr)'+\Bigl(1+\frac{n-1}{r^2}\Bigr)\frac{\zeta}{\rho}
    =(n-1)\Bigl(\frac{\phi'}{r}-\frac{\phi}{r^2}\Bigr)+\Bigl(1+\frac{n-1}{r^2}\Bigr)\bigl(\phi-r\phi'\bigr)
    =\frac{\zeta}{\rho},
\]
which is \eqref{tmp:ODE} with $\mu=0$ and $\lambda=1$.
Here $\mu=0$, so $\ell=0$. The eigenfunction has different signs at $r_-$ and $r_+$, so by the count of zeros the index $k$ is at least $2$. On the other hand since $\lambda = 1>0$ in this case, Lemma~\ref{lem:mode-structure} (ii) makes $k\geq 3$ impossible, and thus $k=2$ and $\lambda=\lambda_{0,2}$.
\end{proof}

\subsection{Two ground-state criteria}\label{ssec:criteria}

A key step in establishing the signs of $\lambda_{2,2}$ and of the modes $\ell\geq4$ is the following oscillation-type lemma, which converts a solution of the zero-eigenvalue ODE having a single sign change and favorable endpoint behavior into a sign bound for the second Neumann eigenvalue. It is a consequence of the ground-state representation (see \cite[Lemma~2.4]{PT06}). Note that condition (ii) below states precisely that, at each endpoint, the one-sided derivative of $|V|$ taken in the direction pointing inside $(a,b)$ is strictly negative (for comparison, a Neumann eigenfunction would have vanishing derivative there).

\begin{lemma}\label{lem:oscillation}
Let $\rho\in C^1([a,b])$ be positive, let $q\in C([a,b])$, and let $\lambda_1>\lambda_2>\cdots$\footnote{The eigenvalues are simple: a solution of \eqref{eq:osc-spectrum} with $\zeta'(a)=0$ is determined up to a scalar.} denote the eigenvalues of the Neumann problem
\begin{equation}\label{eq:osc-spectrum}
	\left(\frac{\zeta'}{\rho}\right)' + q\,\frac{\zeta}{\rho} = \lambda\,\frac{\zeta}{\rho} \ \text{ in } (a,b), \qquad \zeta'(a)=\zeta'(b)=0.
\end{equation}
Suppose there exists $V\in C^2([a,b])$ solving $\left(\frac{V'}{\rho}\right)' + q\,\frac{V}{\rho} = 0$ in $(a,b)$, such that:
\begin{enumerate}[label=\textup{(\roman*)}]
	\item $V$ has a unique zero $z\in(a,b)$, and the zero is simple;
	\item $V(a)V'(a)<0$ and $V(b)V'(b)>0$.
\end{enumerate}
Then $\lambda_2<0$.
\end{lemma}

\begin{proof}
By Sturm--Liouville theory, the eigenvalues of \eqref{eq:osc-spectrum} admit the min-max characterization
\[
	\lambda_2 = \inf_{\operatorname{codim}Y=1}\,\sup_{\zeta\in Y\setminus\{0\}} \frac{\mathcal{Q}[\zeta]}{\int_a^b \zeta^2\rho^{-1}\,dr},
	\qquad
	\mathcal{Q}[\zeta] := -\int_a^b(\zeta')^2\rho^{-1}\,dr + \int_a^b q\,\zeta^2\rho^{-1}\,dr,
\]
where $Y$ ranges over closed subspaces of $H^1(a,b)$ of codimension one. Since point evaluation is bounded on $H^1(a,b)$,
\[
	X:=\{\zeta\in H^1(a,b):\zeta(z)=0\}
\]
is such a subspace, and it suffices to show that the Rayleigh quotient is negative on $X\setminus\{0\}$.

For an interval $J\subset(a,b)$ on which $V$ does not vanish, let $\mathcal{Q}_J$ denote the quadratic form $\mathcal{Q}$ with the integrals restricted to $J$, and write $\zeta=Vw$. Integration by parts using $q\,\frac{V}{\rho}=-\left(\frac{V'}{\rho}\right)'$ gives the ground-state representation
\begin{multline}\label{eq:osc-factorization}
    \mathcal Q_{J}[\zeta]
    =
    -\int_J \frac{\bigl((Vw)'\bigr)^2}{\rho}\,dr
    -\int_J \left(\frac{V'}{\rho}\right)' Vw^2\,dr \\
    =
    -\int_J \frac{\bigl((Vw)'\bigr)^2}{\rho}\,dr
    -\left[\frac{V'V}{\rho}w^2\right]_{\partial J}
    +\int_J \frac{V'}{\rho}(Vw^2)'\,dr \\
    =
    -\int_J
        \frac{(V'w+Vw')^2
        -V'(V'w^2+2Vww')}{\rho}\,dr
    -\left[\frac{V'V}{\rho}w^2\right]_{\partial J}
    =
    -\int_J\frac{V^2}{\rho}
        (w')^2\,dr
    -
    \left[\frac{V'}{V}\frac{\zeta^2}{\rho}\right]_{\partial J}.
\end{multline}
For $\zeta\in X$ and $\varepsilon>0$ small, applying \eqref{eq:osc-factorization} on $(a,z-\varepsilon)$ and $(z+\varepsilon,b)$ gives
\begin{align}
    \mathcal Q_{(a,z-\varepsilon)}[\zeta]
    +\mathcal Q_{(z+\varepsilon,b)}[\zeta]
    &=
    -\int_{a}^{z-\varepsilon}\frac{V^2}{\rho}
        \left(\left(\frac{\zeta}{V}\right)'\right)^2\,dr
    -\int_{z+\varepsilon}^{b}\frac{V^2}{\rho}
        \left(\left(\frac{\zeta}{V}\right)'\right)^2\,dr \nonumber\\
    &\qquad
    +\frac{V'(a)}{V(a)}\frac{\zeta(a)^2}{\rho(a)}
    -\frac{V'(b)}{V(b)}\frac{\zeta(b)^2}{\rho(b)} \nonumber\\
    &\qquad
    -\frac{V'(z-\varepsilon)}{V(z-\varepsilon)}
        \frac{\zeta(z-\varepsilon)^2}{\rho(z-\varepsilon)}
    +\frac{V'(z+\varepsilon)}{V(z+\varepsilon)}
        \frac{\zeta(z+\varepsilon)^2}{\rho(z+\varepsilon)}.
    \label{eq:osc-factorization-cut}
\end{align}
We now let $\varepsilon\downarrow0$. The inner boundary terms in \eqref{eq:osc-factorization-cut} vanish. Indeed, since the zero of $V$ at $z$ is simple,
\[
    \frac{V'}{V}=\frac{1}{r-z}+O(1)
    \qquad\text{as }r\to z,
\]
and, using the absolutely continuous representative of $\zeta\in H^1(a,b)$ and the condition $\zeta(z)=0$, we have
\[
    |\zeta(r)|^2
    \leq |r-z|\int_z^r|\zeta'(t)|^2\,dt=o(|r-z|).
\]
Passing to the limit in \eqref{eq:osc-factorization-cut}, we obtain
\begin{equation}\label{eq:osc-limit}
    \mathcal Q[\zeta]
    =
    -\int_{a}^{z}\frac{V^2}{\rho}
        \left(\left(\frac{\zeta}{V}\right)'\right)^2\,dr
    -\int_{z}^{b}\frac{V^2}{\rho}
        \left(\left(\frac{\zeta}{V}\right)'\right)^2\,dr
    +\frac{V'(a)}{V(a)}\frac{\zeta(a)^2}{\rho(a)}
    -\frac{V'(b)}{V(b)}\frac{\zeta(b)^2}{\rho(b)}.
\end{equation}
By assumption (ii), both endpoint terms are nonpositive, hence $\mathcal{Q}[\zeta]\leq0$ for every $\zeta\in X$. Moreover, if $\mathcal{Q}[\zeta]=0$, then $(\zeta/V)'\equiv0$ on each side of $z$, that is, $\zeta=c_-V$ on $(a,z)$ and $\zeta=c_+V$ on $(z,b)$; the endpoint terms in \eqref{eq:osc-limit} then equal $c_-^2\,V'(a)V(a)/\rho(a)$ and $-c_+^2\,V'(b)V(b)/\rho(b)$, which by (ii) are strictly negative unless $c_-=c_+=0$. Therefore $\mathcal{Q}[\zeta]<0$ for every nonzero $\zeta\in X$.

Finally, by the compact embedding $H^1(a,b)\hookrightarrow C([a,b])$, the supremum of the Rayleigh quotient over $X\setminus\{0\}$ is attained, and is therefore negative. The min-max characterization then gives $\lambda_2<0$.
\end{proof}

The same representation yields a criterion for the sign of the first eigenvalue: a positive supersolution with the favorable endpoint behavior forces $\lambda_1<0$.

\begin{lemma}\label{lem:supersolution}
Let $\rho$, $q$, and $(\lambda_k)_{k\geq1}$ be as in Lemma~\ref{lem:oscillation}. Suppose there exists $V\in C^2([a,b])$ with $V>0$ on $[a,b]$, such that:
\begin{enumerate}[label=\textup{(\roman*)}]
	\item $\left(\frac{V'}{\rho}\right)' + q\,\frac{V}{\rho}<0$ on $[a,b]$;
	\item $V'(a)\leq0\leq V'(b)$.
\end{enumerate}
Then $\lambda_1<0$.
\end{lemma}

\begin{proof}
Let $\zeta\in H^1(a,b)$ be nonzero and write $\zeta=Vw$. Repeating the integration by parts in \eqref{eq:osc-factorization} on $J=(a,b)$, without using an equation for $V$, gives
\[
    \mathcal Q[\zeta]
    =
    -\int_a^b\frac{V^2}{\rho}(w')^2\,dr
    +\int_a^b\left[\left(\frac{V'}{\rho}\right)'+q\,\frac{V}{\rho}\right]Vw^2\,dr
    +\frac{V'(a)}{V(a)}\frac{\zeta(a)^2}{\rho(a)}
    -\frac{V'(b)}{V(b)}\frac{\zeta(b)^2}{\rho(b)}.
\]
By (i) and (ii), every term is nonpositive, and the second is strictly negative unless $w\equiv0$. Hence $\mathcal Q[\zeta]<0$ for every nonzero $\zeta\in H^1(a,b)$, and the min-max characterization gives $\lambda_1<0$ as in the proof of Lemma~\ref{lem:oscillation}.
\end{proof}

\subsection{Negativity of $\lambda_{2,2}$ and $\lambda_{4,1}$}\label{ssec:negative}

\begin{lemma}\label{lem:lambda22}
For every $n\geq2$, $\lambda_{2,2}<0$.
\end{lemma}

\begin{proof}
The strategy is to apply Lemma~\ref{lem:oscillation}: we construct an explicit solution $g$ of the zero-eigenvalue equation $\mathscr{L}_2 g=0$ in $(r_-,r_+)$ (with no boundary condition imposed) such that $V:=\rho g$ has a unique, simple zero in $(r_-,r_+)$ and the inward one-sided derivative of $|V|$ is strictly negative at both endpoints. The endpoint behavior of $V$ encodes the geometric fact that the annulus contains the sphere of radius $\sqrt{2n}$ (Proposition~\ref{prop:annulus-containment}). Recall from \eqref{eq:spectrum} that
\begin{equation}
    \mathscr{L}_{2}u=u''+\left(\frac{n-1}{r}-\frac{r}{2}\right)u'+ \left( \frac12 - \frac{2n}{r^2} \right)u.
\end{equation}

We construct $g$ by briefly returning to the variable $s=r^2/4$. 
Using $\partial_r =\frac{r}{2}\partial_s$ and $r^2=4s$, the operator $\mathscr{L}_2$ becomes
\begin{equation}
    \mathcal{L}_2H=sH_{ss}+\left(\frac{n}{2}-s\right)H_s+\left(\frac12 -\frac{n}{2s}\right)H.
\end{equation}
Recalling that $F(s) =\phi(2\sqrt{s})$ solves \eqref{eq:K}, we have
\begin{equation} \label{eq:L_2F}
    \mathcal{L}_2F = -\frac{n}{2s}F.
\end{equation}
Differentiating \eqref{eq:K}, we get
\begin{equation}
    sF_{sss}+\left( \frac{n}{2}+1-s \right)F_{ss}-\frac12F_s=0,
\end{equation}
so that $F_s$ satisfies
\begin{equation} \label{eq:L_2G}
    \mathcal{L}_2 F_s=-F_{ss}+ \left(1 -\frac{n}{2s}\right)F_s=-\frac{1}{s}\left(sF_{ss}+\left(\frac{n}{2}-s\right)F_s\right)=\frac{F}{2s}.
\end{equation}
Thus, from \eqref{eq:L_2F} and \eqref{eq:L_2G}, we obtain
\begin{equation}
    \mathcal{L}_2(F+nF_s)=0.
\end{equation}
Hence, returning to the $r$--coordinates, we deduce that
\begin{equation}\label{eq:L2-g-definition}
g(r):=\phi(r)+\frac{2n}{r}\phi'(r),
\end{equation}
satisfies
\begin{equation}
        \mathscr{L}_2g=0.
\end{equation}
 We claim that $g$ has a unique root in $(r_-,r_+)$. Indeed, by the Bernoulli condition, we have $g(r_-)>0$ and $g(r_+)<0$, so $g$ has at least one root in $(r_-,r_+)$. On the other hand, if there were two consecutive roots $x_0<x_1$, then $g$ would be the principal Dirichlet eigenfunction for $\mathscr{L}_2$ on $I:=(x_0,x_1) \subset (r_-,r_+)$, with eigenvalue $\lambda^{\operatorname{Dir}}_{2,1}(I)=0$. Meanwhile, $\phi$ satisfies $\mathscr{L}_0\phi=0$ and is the principal Dirichlet eigenfunction for $\mathscr{L}_0$ on $(r_-,r_+)$, with eigenvalue $\lambda_{0,1}^{\operatorname{Dir}}(r_-,r_+)=0$. Thus, by the fact that $\mathscr{L}_2=\mathscr{L}_0-\frac{2n}{r^2}$,
 \begin{equation}
    0=\lambda^{\operatorname{Dir}}_{2,1}(I), \qquad \lambda^{\operatorname{Dir}}_{2,1}(r_-,r_+)<\lambda^{\operatorname{Dir}}_{0,1}(r_-,r_+)=0,
 \end{equation}
 a contradiction with the strict domain monotonicity of eigenvalues. Let $z\in(r_-,r_+)$ denote the unique zero of $g$.
The zero is simple, since otherwise ODE uniqueness for $\mathscr L_2g=0$ would force $g\equiv0$.

Set $V:=\rho g$. The equation $\mathscr L_2g=0$, which is equivalent to \eqref{tmp:ODE}
with $\mu=\mu_2=2n$ and $\lambda=0$, gives
\begin{equation}\label{eq:L2-V-equation}
    \left(\frac{V'}{\rho}\right)'
    +\left(1-\frac{n+1}{r^2}\right)\frac{V}{\rho}=0.
\end{equation}
We compute $V'$ at the two endpoints. Since
$\phi(r_\pm)=0$ and $\phi'(r_\pm)=\mp1$, we have
\begin{equation}\label{eq:L2-g-endpoint-values}
    g(r_\pm)=\mp\frac{2n}{r_\pm}.
\end{equation}
Writing $\alpha=\rho'/\rho$, the ODE \eqref{eq:ode} gives $\phi''(r_\pm)=-\alpha(r_\pm)\phi'(r_\pm)$, and hence, differentiating \eqref{eq:L2-g-definition},
\[
    g'(r_\pm)
    =
    \phi'(r_\pm)
    +
    2n\left(\frac{\phi''(r_\pm)}{r_\pm}
    -\frac{\phi'(r_\pm)}{r_\pm^2}\right)
    =
    \phi'(r_\pm)\left(1-\frac{2n\alpha(r_\pm)}{r_\pm}
    -\frac{2n}{r_\pm^2}\right).
\]
Substituting these values into $V'=\rho(g'+\alpha g)$, the terms involving $\alpha$ cancel, and we obtain
\begin{equation}\label{eq:L2-Vprime-endpoint-values}
    V'(r_\pm)
    =
    \mp\rho(r_\pm)\,\frac{r_\pm^2-2n}{r_\pm^2}.
\end{equation}
By Proposition~\ref{prop:annulus-containment}, $r_-^2<2n<r_+^2$.
Combining this with \eqref{eq:L2-Vprime-endpoint-values}, we obtain
\begin{equation}\label{eq:L2-V-endpoint-sign}
    V'(r_-)<0,\qquad V'(r_+)<0.
\end{equation}

We can now conclude. By \eqref{eq:L2-g-endpoint-values}, $V(r_-)=\rho(r_-)g(r_-)>0$ and $V(r_+)<0$, so \eqref{eq:L2-V-endpoint-sign} gives $V(r_-)V'(r_-)<0$ and $V(r_+)V'(r_+)>0$.
Moreover, since $\rho>0$, the function $V$ has the same unique, simple zero $z$ as $g$. Thus $V$ satisfies the hypotheses of Lemma~\ref{lem:oscillation} on $(r_-,r_+)$, with weight $\rho$ and potential $q(r)=1-\frac{n+1}{r^2}$: it solves the zero-eigenvalue equation by \eqref{eq:L2-V-equation}, and conditions (i) and (ii) hold. Since the Neumann problem \eqref{eq:osc-spectrum} for this choice of $\rho$ and $q$ is exactly \eqref{tmp:ODE} with $\mu=\mu_2$, which is equivalent to \eqref{eq:spectrum}, its eigenvalues are $(\lambda_{2,k})_{k\geq 1}$. Lemma~\ref{lem:oscillation} therefore gives $\lambda_{2,2}<0$.
\end{proof}

The first mode with $\ell=4$ is treated by an explicit supersolution.

\begin{proposition}\label{prop:lambda41}
For every $n\geq2$, $\lambda_{4,1}<0$.
\end{proposition}

\begin{proof}
Work in the scaled variables \eqref{tmp:annular-scaled-notation} or \eqref{eq:annular-scaled-notation} and set
$\varrho(t):=t^{n-1}e^{-t^2/2}$. Since
$\mu_4-(n-1)=3(n+3)$, the change of variables
$V(t):=\zeta(\sqrt2\,t)$ transforms the equation \eqref{tmp:ODE} into
\begin{equation}\label{eq:l4-scaled}
    \left(\frac{V'}{\varrho}\right)'
    +\left(2-\frac{3(n+3)}{t^2}\right)\frac{V}{\varrho}
    =2\lambda\frac{V}{\varrho},
    \qquad V'(a)=V'(b)=0,
\end{equation}
whose eigenvalues are $(2\lambda_{4,k})_{k\geq1}$. Thus, by
Lemma~\ref{lem:supersolution}, it is enough to find a positive supersolution $V$ with
$V'(a)\leq0\leq V'(b)$ such that
\[
    \mathcal HV:=\varrho\left(\frac{V'}{\varrho}\right)'
    +\left(2-\frac{3(n+3)}{t^2}\right)V<0.
\]

Set $p:=ab$ and define
\[
    V(t):=t^{-2p/5}
    \exp\left(\frac{2(a+b)t-t^2}{5}\right),
    \qquad
    g(t):=\frac{V'(t)}{V(t)}
    =\frac{2(t-a)(b-t)}{5t}.
\]
Then $V>0$, while $g\geq0$ on $[a,b]$ and $g(a)=g(b)=0$; hence
$V'(a)=V'(b)=0$. By Lemma~\ref{lem:annular-phase},
$\delta:=n-1-p>0$. Since $\varrho'/\varrho=(n-1)/t-t$,
\[
    \frac{\mathcal HV}{V}
    =F_p-\delta\left(\frac gt+\frac3{t^2}\right)<F_p,
    \qquad
    F_p:=g'+g^2+\left(t-\frac pt\right)g
      +2-\frac{3(p+4)}{t^2},
\]
because $g\geq0$ and $t>0$.

For $t\in[a,b]$, put
\[
    x:=t-a,\qquad y:=b-t,\qquad u:=xy,\qquad v:=x-y.
\]
Then $u'=-v$, and
\[
    p=(t-x)(t+y)=t^2-vt-u,
    \qquad
    g=\frac{2u}{5t},
    \qquad
    g'=-\frac{2v}{5t}-\frac{2u}{5t^2}.
\]
Substitution into the definition of $F_p$ gives
\[
    t^2F_p
    =-t^2+\frac{v(2u+13)}5\,t
      +\frac{14u^2+65u-300}{25}.
\]
Completing the square in $t$ and using Proposition~\ref{prop:annulus-width},
\[
    -t^2+\frac{v(2u+13)}5\,t
    \leq \frac{v^2(2u+13)^2}{100},
    \qquad
    v^2=(x-y)^2=(x+y)^2-4xy
    =(b-a)^2-4u<7-4u.
\]
Since $u=xy\geq0$, it follows that
\[
\begin{split}
    t^2F_p
    &<\frac{(7-4u)(2u+13)^2}{100}
      +\frac{14u^2+65u-300}{25}\\
    &=-\frac{16u^3+124u^2+52u+17}{100}<0.
\end{split}
\]
Thus $F_p<0$ on $[a,b]$. It follows that
$\mathcal HV/V<F_p<0$, and hence $\mathcal HV<0$ there. Since
$V>0$ and $V'(a)=V'(b)=0$, Lemma~\ref{lem:supersolution} applied
to \eqref{eq:l4-scaled} gives $2\lambda_{4,1}<0$, and hence
$\lambda_{4,1}<0$.
\end{proof}

\subsection{Positivity of $\lambda_{3,1}$}\label{ssec:l3}

We now prove that $\lambda_{3,1}>0$ for every $n\geq2$. This is the borderline case described in Section \ref{ssec:main-ideas}: by \eqref{eq:fine-radius-estimate}, the zeroth-order coefficient $1-\frac{\mu_\ell-(n-1)}{r^2}$ of the Rayleigh quotient \eqref{eq:rayleigh-p} equals $\frac{3-\ell}{2}+O(n^{-1/2})$ uniformly on the annulus, so the formal limiting problem for $\ell=3$ has largest eigenvalue zero, and the sign of $\lambda_{3,1}$ is decided by the finer geometric information of Theorem \ref{thm:annulus-geometry}.

The first step is a reduction to a moment inequality for the annular radii.

\begin{lemma}\label{lem:l3-moment}
Let $n\geq2$, and suppose that
\begin{equation}\label{eq:l3-moment}
    \int_{r_-}^{r_+}\bigl(r^2-2(n+2)\bigr)\,r^4\rho(r)\,dr>0 .
\end{equation}
Then $\lambda_{3,1}>0$.
\end{lemma}

\begin{proof}
Let $\mathcal Q_3$ denote the numerator of the Rayleigh quotient \eqref{eq:rayleigh-p} for
$\mu=\mu_3$. Since $\mu_3-(n-1)=2(n+2)$, the corresponding differential
expression factors as
\[
    \mathscr H_3:=\partial_r^2 -\left(\frac{n-1}{r}-\frac r2\right)\partial_r
    +1-\frac{2(n+2)}{r^2}
    =
    \left(\partial_r+\frac r2-\frac{n+1}{r}\right)
    \left(\partial_r+\frac2r\right).
\]
Noting that $r^{-2}$ is the integrating factor for the second operator on the right, writing $\zeta=r^{-2}g$ gives
\begin{equation} \label{eq:l3-reduced}
    \mathcal Q_3[r^{-2}g]
    =2\left[\frac{g^2}{r^5\rho}\right]_{r_-}^{r_+}
    -\int_{r_-}^{r_+}\frac{(g')^2}{r^4\rho}\,dr .
\end{equation}
Set $I:=\int_{r_-}^{r_+}r^4\rho\,dr$, $u:=r_-^5\rho(r_-)$, and
$v:=r_+^5\rho(r_+)$, and choose
\[
    g(r):=\frac u2+\int_{r_-}^r s^4\rho(s)\,ds.
\]
We note that (by the Cauchy-Schwarz inequality), $g'=r^4\rho(r)$ is chosen to minimize the value of the negative term in \eqref{eq:l3-reduced} given the endpoint values $g(r_{\pm})$.
Since $g(r_-)=u/2$, $g(r_+)=u/2+I$, and $g'=r^4\rho$, direct substitution
gives
\[
    \mathcal Q_3[r^{-2}g]
    =\frac{(u+2I)^2}{2v}-\frac u2-I
    =\frac{(u+2I)(u+2I-v)}{2v}.
\]
Moreover,
\[
    u+2I-v
    =\frac12\int_{r_-}^{r_+}
       \bigl(r^2-2(n+2)\bigr)r^4\rho(r)\,dr>0
\]
by the assumption \eqref{eq:l3-moment}. Thus $\mathcal Q_3[r^{-2}g]>0$, and the variational
characterization gives $\lambda_{3,1}>0$.
\end{proof}

Thanks to Lemma \ref{lem:l3-moment}, the sign of $\lambda_{3,1}$ can be determined by the location of the radii $r_{\pm}$, which in turn depends on the dimension. However, we will explain below why the moment inequality is difficult to verify numerically and what it depends on. In fact, the integral in \eqref{eq:l3-moment} has leading order in $O(1)$ and does not grow with dimensions, which matches with the heuristics above that the largest eigenvalue of the formal limiting problem is
zero. 

To see this, we center the interval of integration $[r_-,r_+]$ by setting
\[ r_0:= \frac{r_+ + r_-}{2}, \qquad \bar{d}:= \frac{r_+ - r_-}{2}. \]
By the asymptotics of the radii in \eqref{eq:fine-radius-estimate} and proven in Appendix \ref{sec:radius-est}, we have that
\begin{equation}\label{tmp:center-shift}
    r_0 = \sqrt{2(n-1)} + O(n^{-1/2}), \qquad \bar{d} = \sqrt{2}\kappa + O(n^{-1/2}) = O(1). 
\end{equation}
Note that the boundary data, $\phi(r_-) = \phi(r_+)$ and $\phi'(r_-) = -\phi'(r_+)$, are those of a profile even about the midpoint $r_0$, but the equation \eqref{eq:ode} is not invariant under the reflection, because the damping coefficient $\frac{n-1}{r} - \frac{r}{2}$ is not odd with respect to its unique root $\sqrt{2(n-1)}$.
To quantify the mismatch between the symmetric boundary data and asymmetric equation, we also set
\[ \bar{R}:= r_0^2 - 2(n-1), \]
which is bounded thanks to \eqref{tmp:center-shift}. (In fact, it is also bounded from below by Lemma \ref{lem:annular-midpoint}.)

By letting $g_n(r):= r^4 \rho(r) = r^{n+3} e^{-r^2/4}$, the integral in \eqref{eq:l3-moment} is the sum of
\begin{equation}\label{tmp:integral-center}
    \int_0^{\bar{d}} \left[ (r_0 \pm s)^2 - 2(n+2) \right] g_n(r_0\pm s) ds. 
\end{equation} 
We analyze the integral on the right half, and the left half is analogous. For bounded $s$, Taylor expansion of the normalized weight gives
\[
 \log\frac{g_n(r_0+s)}{g_n(r_0)}
 =(n+3)\log\left(1+\frac{s}{r_0}\right)
   -\frac{r_0 s}{2}-\frac{s^2}{4}
 = -\frac{s^2}{2}  + \frac{s}{6r_0} \left(24-3\bar{R} + s^2 \right)  +O(r_0^{-2}).
\]
In particular the two terms of order $r_0$ cancels.
Thus exponentiating gives
\[
 \frac{g_n(r_0+s)}{g_n(r_0)}
 =e^{-s^2/2}\left[
 1+\frac{s}{6r_0} \left(24-3\bar{R} + s^2 \right) +O(r_0^{-2})
 \right].
\]
Therefore the integral on the right half \eqref{tmp:integral-center} has the following asymptotics
\begin{equation}\label{tmp:integral-asymptotics}
    \int_0^{\bar{d}} e^{-s^2/2} \left[ 2sr_0 + \left( \bar{R}-6 + (9-\bar{R}) s^2 + \frac{s^4}{3} \right) + O(r_0^{-1}) \right] ds. 
\end{equation} 
Note that the leading order term $2sr_0$ is odd, so it cancels with the integral on the left half. Hence the sign of the full integral is determined by the sign of the next order expansion, the $O(1)$ term
\begin{equation}\label{tmp:integral-leading-order}
    2 \int_0^{\bar{d}} e^{-s^2/2} \left( \bar{R}-6 + (9-\bar{R}) s^2 + \frac{s^4}{3} \right) ds.
\end{equation}

This formula makes the sign behavior visible. Proposition~\ref{prop:annulus-width}
and Lemma~\ref{lem:annular-midpoint}, after rescaling, give $\bar{R}<\bar{d}^2<7/2$. Hence the polynomial
in parentheses is negative near $s=0$ and strictly increasing as a function of
$s^2$, since its derivative with respect to $s^2$ is
$9-\bar{R}+\frac23 s^2>0$. Thus the negative contribution near the midpoint of the
annulus competes with positive contributions nearer the two endpoints. For
$n\geq4001$, the sharper radius and midpoint estimates give explicit lower bounds for $\bar{R}$ and $d$; % $R>3/5$ and $d>8/7$; 
moreover, the exact moment is increasing in both parameters in this
range by Lemma~\ref{lem:l3-monotone}. It is therefore enough to verify the positivity of the integral \eqref{tmp:integral-leading-order} at the lower bounds of $\bar{R}$ and $\bar{d}$, see \eqref{eq:l3-J0}.
In particular, at the worst admissible corner of $(\bar{R},\bar{d})$ the
positive contributions near the endpoints outweigh the negative contribution
near the midpoint.

As we go further down in the asymptotics of the integral \eqref{tmp:integral-asymptotics}, note that every odd power of $r_0^{-1}$ corresponds to an odd function in $s$, so the integrals on the left and right half intervals cancel each other. The next surviving term is therefore of order $r_0^{-2} = O(n^{-1})$.
For $n\geq4001$, this error is smaller than
the positive margin gained from the leading order term \eqref{tmp:integral-leading-order}.

This cancellation illustrates the borderline nature of
the $\ell=3$ mode: The largest eigenvalue of the formal limiting problem is
zero, and here the sign comes from the first asymmetric correction to the
Gaussian weight, which a leading-order estimate discards. This is why the
available margins are small and coarse radius bounds do not suffice.

Appendix~\ref{sec:l3-details} makes the above argument quantitative.
Lemma~\ref{lem:l3-monotone} proves the monotonicity used above. For
$n\geq4001$, the sharper radius and midpoint estimates give explicit lower bounds of $\bar{R}$ and $\bar{d}$ (modulo rescalings), %$(R,d)$ above the corner $(3/5,8/7)$, 
while Lemma~\ref{lem:l3-corner} controls the
$O(n^{-1})$ remainder there. These ingredients are combined in
Lemma~\ref{lem:l3-tail}. Finally, Lemma~\ref{lem:l3-finite} handles the
dimensions $2\leq n\leq4000$ by a rigorous finite computation.

\begin{proposition}\label{prop:lambda31}
For every $n\geq2$, $\lambda_{3,1}>0$.
\end{proposition}

\begin{proof}
Proposition~\ref{prop:l3-moment-positive} verifies
\eqref{eq:l3-moment}; hence Lemma~\ref{lem:l3-moment} gives the result.
\end{proof}

\begin{proof}[{Proof of Theorem~\ref{prop:spectrum-annular}}]
The row $k\geq3$ is Lemma~\ref{lem:mode-structure}(ii). By Lemma~\ref{lem:explicit-modes}, $\lambda_{1,1}=1$, $\lambda_{1,2}=\frac12$, and $\lambda_{0,2}=1$; since the eigenvalues for fixed $\ell$ are strictly decreasing in $k$, this also gives $\lambda_{0,1}>\lambda_{0,2}=1$. In the row $k=2$, Lemma~\ref{lem:lambda22} gives $\lambda_{2,2}<0$, and then $\lambda_{3,2}\leq\lambda_{2,2}<0$ and $\lambda_{4,2}\leq\lambda_{2,2}<0$ by Lemma~\ref{lem:mode-structure}(i). In the row $k=1$, Proposition~\ref{prop:lambda31} gives $\lambda_{3,1}>0$, and then $\lambda_{2,1}\geq\lambda_{3,1}>0$ by Lemma~\ref{lem:mode-structure}(i); Proposition~\ref{prop:lambda41} gives $\lambda_{4,1}<0$, and the same monotonicity gives $\lambda_{\ell,k}<0$ for every $\ell\geq4$ and $k\geq1$.
\end{proof}

\section{Applications of spectral analysis: uniqueness of tangent flow and stability}\label{sec:spectrum-applications}

We briefly explain the setup in \cite{stability} as follows, to illustrate the implications of the linear spectral analysis in Theorem \ref{prop:spectrum-annular}. The authors study solutions $v(s,y)$ to \eqref{eq:respar} near an equilibrium $\phi$.
For each $s$ such that the boundary of $\Omega_s:= \{v(s,\cdot)>0 \}$ is a graph over $\partial\Omega=\partial\{\phi>0\}$ with small norm, there is a vector field $\Psi(s,\cdot)$ supported in a small tubular neighborhood of $\partial\Omega$ such that $\operatorname{Id} + \Psi(s,\cdot): \Omega \to \Omega_s$ is a diffeomorphism. % mapping $\Omega$ onto $\Omega_s$. 
% More precisely, suppose $\partial\Omega_s = \Graph_{\partial\Omega} h_s$ and let $\nu_\xi$ denote the unit normal derivative at $\xi \in \partial\Omega$ pointing in the positivity set. We define
% \[ \Psi(s,\xi):= h_s(\xi) \nu_{\xi}, \qquad \text{ for each } \xi\in \partial\Omega; \]
% and for each $\xi \in \mathcal{N}(\Omega) := \{\xi\in  \Omega: \dist(\xi, \partial\Omega)<\tau_0\}$ for some $\tau_0$ sufficiently small (depending on the $C^2$ norm of $\partial\Omega$), letting $(\xi', \tau)$ denote the Fermi coordinate of $\xi$, that is, $\xi'$ denote a projection of $\xi$ onto $\partial\Omega$ and $\tau=|\xi-\xi'|$, we define
% \begin{equation}\label{def:transf}
%     \Psi(s,\xi) := \psi(\tau) h_s(\xi') \nu_{\xi'}, 
% \end{equation} 
% where $\psi$ is a given smooth cut-off function taking value $1$ near the origin and vanishing after $\tau_0$.

To fix notation, for each $s$ in a suitable time window $I$ and $y\in \Omega_s$, we write $\xi\in \Omega$ such that $y=\xi + \Psi(s,\xi)$. Then the difference
\begin{equation}\label{def:perturb-domain-transf}
    \eta(s,\xi) := v(s,y) - \left( \phi(\xi) + \Psi(s,\xi) \cdot \nabla \phi(\xi) \right) 
\end{equation} 
satisfies a nonlinear equation in the \emph{fixed domain} $\Omega$
\begin{equation}\label{eq:perturb}
    \left\{\begin{array}{ll}
    \partial_s \eta = \mathscr{L} \eta + F(\eta(s,\cdot)) & \text{ in } \Omega \\
    \mathscr{B} \eta = G(\eta(s,\cdot)) &  \text{ in } \partial\Omega
\end{array} \right. 
\end{equation}
where $F$ and $G$ map $C^2(\overline{\Omega})$ to $C(\overline{\Omega})$ and $C^1(\partial\Omega)$ respectively, with $F(0)=0, F'(0)=0, G(0)=0, G'(0)=0$. In other words, the nonlinear equation \eqref{eq:perturb} is the linear equation \eqref{eq:linear-stability} plus quadratic error terms.

The spectral analysis for the annular solution in Theorem \ref{prop:spectrum-annular} in particular implies that, subject to the Robin boundary condition $\mathscr{B} \eta =0$, the operator $\mathscr{L}$ has no kernel element. Together with the nonlinear equation \eqref{eq:perturb}, this gives the local rigidity needed for
Theorem~\ref{thm:compact-radial-tangent-uniqueness}. 

\begin{lemma}[Local isolation of the ball and annular profiles]\label{lem:compact-radial-local-isolation}
Let $\mathcal T$ be the collection of tangent profiles given by
Theorem~\ref{thm:blowup}. Each $\phi_*\in\{\phi_B,\phi_A\}$ is isolated in
$\mathcal T$ for the topology of locally uniform convergence.
\end{lemma}

\begin{proof}
Fix $\phi_*\in\{\phi_B,\phi_A\}$, set $\Omega_*:=\{\phi_*>0\}$, and suppose
that $\phi_j\in\mathcal T$ and $\phi_j\to\phi_*$ locally uniformly.
Recall that tangent profiles $\phi_j$ in $\mathcal{T}$ satisfy a Hessian bound with uniform upper bound $C_s$; and the estimates \eqref{eq:nearbd-lower} and \eqref{eq:thickness} use only the
uniform Hessian bound and the classical Bernoulli condition, hence their proofs apply verbatim to $\phi_j$'s. Applying the proof of \eqref{eq:CVHausdorff} and then the
Arzel\`a--Ascoli argument using the uniform curvature bound and the local
perimeter estimate in \eqref{eq:periUB} likewise shows that, near
$\partial\Omega_*$, the free boundaries
$\Gamma_j:=\partial\{\phi_j>0\}$ are locally finite disjoint unions of $C^1$
hypersurfaces converging to $\partial\Omega_*$ in $C^1$.

Fix a connected component $\Sigma_*$ of $\partial\Omega_*$ and choose
$\delta>0$ so small that its $\delta$-tubular neighborhood is disjoint from those of
the other components. Using the normal parametrization from the proof of
\eqref{eq:tube-meas}, write
\[
    \mathcal N_\delta(\Sigma_*)
    =\{p+t\nu_*(p):p\in\Sigma_*,\ |t|<\delta\},
\]
where $\nu_*$ is the unit normal pointing into $\Omega_*$. For large $j$, denote the connected hypersurface
components of $\Gamma_j\cap\mathcal N_\delta(\Sigma_*)$ by
$\Sigma_j^1,\ldots,\Sigma_j^{N_j}$.
Together with the $C^1$ convergence of the free boundaries, the one-sided tubular estimate \eqref{eq:thickness} and local uniform convergence $\phi_j \to \phi_*$ force the positive-side normal of each $\Sigma_j^k$ to converge to $\nu_*$: a component with the opposite orientation would carry a positive collar into the interior of $\RR^n\setminus\overline\Omega_*$. For all sufficiently large $j$, suppose that $\Sigma_j^k$ meets the normal fiber over $\Sigma_*$ through $p$ at $x_0=p+t_0\nu_*(p)$. Then, for every
sufficiently small $\varepsilon>0$,   
\[  p+(t_0+\varepsilon)\nu_*(p)\in\{\phi_j>0\}. \]
On the other hand, suppose $\varepsilon_j>0$ is the largest value such that $p+(t_0-\varepsilon)\nu_*(p)\in\{\phi_j=0\}$ for all $0<\varepsilon< \varepsilon_j$. Then either $x_j:= p+(t_0-\varepsilon_j)\nu_*(p)\notin \mathcal{N}_\delta(\Sigma_*)$, or $x_j \in \Gamma_j \cap \mathcal{N}_\delta(\Sigma_*)$. The latter contradicts the fact that $|\nu_j(x_j)- \nu_*(p)| \ll 1$.
To sum up, a normal fiber cannot meet $\bigcup_{k=1}^{N_j}\Sigma_j^k$ in two distinct
points with this orientation: between two transitions from $\{\phi_j=0\}$ to
$\{\phi_j>0\}$ there would have to be a transition in the opposite direction.
It follows that $\Gamma_j\cap\mathcal N_\delta(\Sigma_*)$ has exactly one
connected component, which is a normal graph over $\Sigma_*$ and converges to
$\Sigma_*$ in $C^1$.

To upgrade the convergence, we work in a boundary chart in which the $x_n$-axis
points into $\Omega_*$ and consider the partial hodograph map
\[
    H_j(x',x_n):=(x',\phi_j(x',x_n)).
\]
The $C^1$ convergence of these normal graphs and the uniform Hessian bound
allow the chart and its positive-side collar to be chosen so that
\[
    \frac12\leq\partial_{x_n}\phi_j\leq\frac32
\]
there, uniformly in $j$. Hence $\det DH_j=\partial_{x_n}\phi_j$ is bounded away
from zero. Write the inverse as $H^{-1}_j(z',z_n)=(z',q_j(z',z_n))$, and define $q_*$ analogously using $H_*(x',x_n)=(x',\phi_*(x',x_n))$. The profile equation becomes a uniformly
elliptic quasilinear equation for $q_j$, while the boundary Bernoulli condition becomes
the uniformly oblique boundary condition
\[
    \partial_{z_n}q_j
    =\sqrt{1+|\nabla_{z'}q_j|^2}
    \qquad\text{on $\{z_n=0\}$}.
\]
The boundary regularity estimate for nonlinear oblique problems
\cite[Theorem~6.2 and the following remark]{LT} gives uniform
$C^{2,\alpha_0}$ bounds for $q_j$ on smaller boundary charts.
Together with
the locally uniform convergence, this yields $q_j\to q_*$ in $C^{2,\alpha}$ on the upper half space for some
$\alpha\in(0,\alpha_0)$. Consequently, the perturbations $\eta_j$ associated to these normal graphs through
\eqref{def:perturb-domain-transf} satisfy
$\eta_j\to0$ in $C^{2,\alpha}(\overline\Omega_*)$.

For stationary profiles, \eqref{eq:perturb} becomes
\[
    \mathscr L\eta+F(\eta)=0\quad\text{in $\Omega_*$},
    \qquad
    \mathscr B\eta-G(\eta)=0\quad\text{on $\partial\Omega_*$}.
\]
The fixed-domain construction in \cite[Sections~2--3]{stability} gives a
$C^1$ map
\[
\mathcal H(\eta):=(\mathscr L\eta+F(\eta),\mathscr B\eta-G(\eta))
\colon C^{2,\alpha}(\overline\Omega_*)
\longrightarrow C^\alpha(\overline\Omega_*)\times C^{1,\alpha}(\partial\Omega_*)
\]
near zero, with $D\mathcal H(0)=(\mathscr L,\mathscr B)$. Theorem
\ref{prop:spectrum-annular} and the ball spectral analysis in
\cite[Section~4.2]{stability} give
$\ker D\mathcal H(0)=\{0\}$. By
\cite[Theorem~2.30]{Lieberman}, $D\mathcal  H(0)$ is Fredholm of index zero and
hence an isomorphism. The inverse function theorem gives
$\eta_j=0$ for all sufficiently large $j$. Consequently, $\Omega_*$ is a
connected component of $\{\phi_j>0\}$, and $\phi_j=\phi_*$ on $\Omega_*$.

The Gaussian-volume identity in Theorem~\ref{thm:blowup} gives
\[
    \int_{\{\phi_j>0\}}e^{-|y|^2/4}\,dy
    =\int_{\Omega_*}e^{-|y|^2/4}\,dy.
\]
Any other connected component of $\{\phi_j>0\}$ would be open and would add
strictly positive Gaussian volume. Hence there is no such component, and
$\phi_j=\phi_*$ on $\RR^n$. This proves the isolation.
\end{proof}

\begin{proof}[Proof of Theorem~\ref{thm:compact-radial-tangent-uniqueness}]
Let
\[
    \omega(v):=\bigcap_{S>0}
    \overline{\{v(s,\cdot):s\geq S\}}^{\,C_{\rm loc}(\RR^n)}.
\]
The compactness in Theorem~\ref{thm:blowup} makes every tail precompact.
Since $s\mapsto v(s,\cdot)$ is continuous in $C_{\rm loc}(\RR^n)$, each tail
closure is connected, and thus their nested intersection $\omega(v)$ is a
nonempty compact connected set. Every element of $\omega(v)$ belongs to
$\mathcal T$, and the hypothesis gives $\phi_*\in\omega(v)$. By
Lemma~\ref{lem:compact-radial-local-isolation}, the singleton $\{\phi_*\}$ is
both open and closed relative to $\omega(v)$. Connectedness yields
$\omega(v)=\{\phi_*\}$, and precompactness then gives the full convergence.
\end{proof}

Back to the stability analysis in \cite{stability}, the authors prove the following in Theorems 3.1 and 3.3, respectively.

\begin{theorem}
	For any $\eta_0$ with small norm in $C^{2,\alpha}(\overline{\Omega})$ and satisfying the compatibility condition $\mathscr{B} \eta_0 = G(\eta_0(\cdot))$, there is a unique solution to \eqref{eq:perturb} with initial data $\eta_0$.
\end{theorem}

\begin{theorem}
	Let $\sigma(\mathscr{L})$ denote the discrete spectrum of $\mathscr{L}$ with boundary condition $\mathscr{B}v=0$. Suppose that $\sigma^+(\mathscr{L})$, the elements in $\sigma(\mathscr{L})$ with positive real parts, is not empty.  Let $P^+$ denote the spectral projection associated with $\sigma^+(\mathscr{L})$, and let $\beta \in (0, \min \{ \Re \lambda: \lambda \in \sigma^+(\mathscr{L})\} )$ be fixed. Then there exist $M>0$, $\delta_0>0$ and a Lipschitz continuous map
    \[ \mathscr{R}: P^+(C^{2,\alpha}(\overline{\Omega})) \cap B(0,\delta_0) \to (\operatorname{Id}-P^+)(C^{2,\alpha}(\overline{\Omega})) \]
    with $\mathscr{R}'(0)=0$ such that for each $\eta_0 := \varphi+\mathscr{R}\varphi$ with some $\varphi \in P^+(C^{2,\alpha}(\overline{\Omega})) \cap B(0,\delta_0)$, the nonlinear equation \eqref{eq:perturb} has a unique backwards solution with initial data $\eta_0$, and moreover
    \begin{equation}\label{eq:expdecay}
        \|e^{\beta |s|} \eta(s,y) \|_{C^{1+\alpha/2, 2+\alpha}((-\infty,0] \times \overline{\Omega}) } \leq M. 
    \end{equation} 
    In other words,
    \begin{equation*}
        \eta(s,\cdot) \to 0 \text{ exponentially fast as } s\downarrow -\infty, \qquad \eta(s,\cdot) \to \eta_0 \text{ as } s\uparrow 0.
    \end{equation*} 
\end{theorem}

The map $\mathscr{R}$ describes the \emph{unstable manifold} generated by a finite-dimensional linear subspace of unstable eigenfunctions.
As an example, we have seen in Theorem \ref{prop:spectrum-annular} that for any $a\in \RR^n$, $\phi'(r) \langle \omega, a\rangle$ is an eigenfunction with eigenvalue $\frac12>0$, and it corresponds to the following solution to the nonlinear problem \eqref{eq:respar} $v_a(s,y):= \phi(y+a e^{\frac{s}{2}})$. Clearly, if $|a|$ is sufficiently small, each $\partial \Omega_s = \partial A- a e^{\frac{s}{2}}$ is a graph over $\partial A$ for all $s\in (-\infty,0]$, and
\[ \text{ as }s \downarrow -\infty, \qquad v_a(s,\cdot) - \phi \to 0 \text{ exponentially fast,} \]
\[  \text{ as } s\uparrow 0, \qquad v_a(s,\cdot) - \phi \to \phi(y+a) - \phi(y) = \phi'(r) \langle \omega, a \rangle + O(|a|^2), \text{ with } r=|y|, \omega = \frac{y}{|y|}. \]
Similarly, there are explicit solutions to \eqref{eq:respar} with initial data $(1-\epsilon)\phi\left(y/(1-\epsilon) \right)$ generated by the eigenfunction $\phi(r)-r\phi'(r)$, which satisfies an analogous asymptotics.

Conversely, consider any backwards solution $\eta$ constructed as above from an unstable mode. What is the behavior of the corresponding solution $v$ to \eqref{eq:respar}? Recall that the domain transformations argument here demands an extra term $\Psi(s,\xi) \cdot \nabla \phi(\xi)$ in the perturbation \eqref{def:perturb-domain-transf}, which works under the a priori assumption that $\Omega_s$ is sufficiently close to $\Omega$. For each $s \in I$, we can recover the boundary data of the vector field $\Psi(s,\cdot)$ as
\[ \Psi(s,\xi) \cdot \nu_\xi = -\Psi(s,\xi) \cdot \nabla \phi(\xi) = \eta(s,\xi) - v(s,y) + \phi(\xi) = \eta(s,\xi) = O(e^{-\beta|s|}), \qquad \xi \in \partial\Omega,  \]
by \eqref{eq:expdecay}. 
Thus the a priori assumption in fact holds for all $s\in (-\infty,0]$, and this determines the full vector field $\Psi(s,\cdot)$ modulo a smooth cut-off function. Then the corresponding solution $v(s,y)$ satisfies: away from the boundary $v(s,y) - \phi(\xi) = \eta(s,\xi) = O(e^{-\beta|s|})$, and near the boundary 
\[ \left| v(s,y) - \phi(\xi) \right| = \left| v(s,y) - \phi(\xi) - (v(s,y') - \phi(\xi')) \right| \lesssim \|\eta\|_{C^{2,\alpha}} |\xi-\xi'| = O(e^{-\beta|s|}\dist(\xi,\partial\Omega)), \] 
where $\xi'$ denotes the projection of $\xi$ onto $\partial\Omega$, and $y' = \xi' + \Psi(s,\xi') \in \partial\Omega_s$. In other words, $v(s,\cdot)$ tends to the equilibrium $\phi$ exponentially fast as $s\to -\infty$.

The above discussion shows how the spectral analysis in Theorem \ref{prop:spectrum-annular} leads to the instability assertion in Corollary \ref{cor:stability}.

\subsection*{Use of AI} The Python script used for the finite verification of Lemma \ref{lem:l3-finite} was written with the assistance of a coding agent. Moreover, agents were used to run inexpensive numerical explorations at scale. In particular, the monotonicity property of Lemma \ref{lem:l3-monotone}
was initially discovered by using an agent to systematically search for relevant quantities that numerically appear to be signed or monotone. The authors take full responsibility for the entire paper.

\subsection*{Acknowledgements}
The authors started working on the paper during the AIM workshop ``Time-dependent Bernoulli-type free boundary problems'', organized by Max Engelstein, William Feldman and Inwon Kim. We would like to thank the organizers and the American Institute of Mathematics for making this collaboration possible. We are also grateful to Nikola Kamburov and Dennis Kriventsov for introducing us to the problem.
The third-named author is supported by the NSF Grant No. DMS-2350351.

\appendix
\part*{Appendices}
%\addcontentsline{toc}{part}{Appendices}

\section{Radius estimates for the annular solution}\label{sec:radius-est}
Let $\phi$ be the annular profile from Proposition \ref{prop:odesol} with support $[r_-,r_+]$.
To simplify the remaining computations, we introduce some notation (recall that they were already used in the proof of Proposition \ref{prop:odesol}):
\begin{equation}\label{eq:annular-scaled-notation}
    \mu:=\frac{n-1}{2},
    \qquad
    a:=\frac{r_-}{\sqrt2},
    \qquad
    b:=\frac{r_+}{\sqrt2},
    \qquad
    y(t):=\frac1{\sqrt2}\phi(\sqrt2\,t),
\end{equation}
and $h(t):=\frac t2-\frac{\mu}{t}$. Then the equation \eqref{eq:ode} becomes the damped oscillator equation
\begin{equation}\label{eq:annular-y-equation}
    y''-2h(t)y'+y=0
    \qquad\text{on }(a,b),\tag{D}
\end{equation}
and
\begin{equation}\label{eq:annular-y-data}
    y(a)=y(b)=0,\qquad y>0\text{ on }(a,b),
    \qquad y'(a)=1,\qquad y'(b)=-1.
\end{equation}

Recall that $y$ is concave on its support, see Claim \ref{cl:concave}. Besides, the shooting monotonicity \eqref{eq:annular-shooting-monotonicity} yields the following one-sided bounds for $a$ and $b$ from single trajectories of \eqref{eq:annular-y-equation}, which is used in the finite verification in Section~\ref{sec:instability}.

\begin{cor}\label{cor:shooting-bracket}
Let $w$ solve $w''-2h(t)w'+w=0$.
\begin{enumerate}[label=\textup{(\roman*)}]
\item If $w(A)=0$ and $w'(A)=1$ for some $A>0$, and $w$ vanishes again at some $z>A$ with $w>0$ on $(A,z)$ and $w'(z)>-1$, then $A<a$.
\item If $w(B)=0$ and $w'(B)=-1$ for some $B>0$, and $w$ vanishes at some $\hat A<B$ with $w>0$ on $(\hat A,B)$ and $w'(\hat A)>1$, then $B<b$.
\end{enumerate}
\end{cor}

\begin{proof}
In the notation of the proof of Proposition~\ref{prop:odesol}, a solution vanishing at $A$ with unit slope and positive up to a second zero is the member $y_\alpha$ of the shooting family with $\alpha=A^2/2$; the existence of the second zero forces $\alpha\in(0,s_*)$, and then $z=b(\alpha)$ and $w'(z)=p(\alpha)$.

(i) Here $p(\alpha)>-1=p(\alpha_0)$, so $\alpha<\alpha_0$ by \eqref{eq:annular-shooting-monotonicity}, that is, $A<a$.

(ii) The function $w/w'(\hat A)$ vanishes at $\hat A$ with unit slope, is positive on $(\hat A,B)$, and has next zero $B$. Hence, with $\hat\alpha:=\hat A^2/2$,
\[
    p(\hat\alpha)=-\frac1{w'(\hat A)}>-1,
\]
so $\hat\alpha<\alpha_0$ as in (i), and \eqref{eq:annular-shooting-monotonicity} gives $B=b(\hat\alpha)<b(\alpha_0)=b$.
\end{proof}

The estimates \eqref{eq:fine-radius-estimate} involve the constant $\kappa=1.1623422615\ldots$, defined as the first positive zero of the even solution of the Hermite-type equation
\begin{equation}\label{eq:hermite-kappa}
    Y''-2xY'+Y=0,
    \qquad
    Y(0)=1,\quad Y'(0)=0;\tag{H}
\end{equation}
in terms of Kummer's function, $Y(x)=M\!\left(-\tfrac14,\tfrac12,x^2\right)$.
This equation is the formal $n\to\infty$ limit of \eqref{eq:annular-y-equation}, after recentering at the zero $\sqrt{2\mu}$ of the damping coefficient $h$: for $t=\sqrt{2\mu}+x$,
\begin{equation}\label{eq:drift-comparison}
    h(\sqrt{2\mu}+x)
    =
    x-\frac{x^2}{2(\sqrt{2\mu}+x)}
    <x,
    \qquad
    -\sqrt{2\mu}<x,\ x\neq0.
\end{equation}

The following rational bounds for $\kappa$ avoid decimal estimates.

\begin{lemma}\label{lem:kappa-bounds}
The constant defined by \eqref{eq:hermite-kappa} satisfies
\begin{equation}\label{eq:kappa-bounds}
    \frac43<\kappa^2<\frac{23}{17}.
\end{equation}
\end{lemma}

\begin{proof}
Write $Y(x)=1-\sum_{j\geq1}b_jx^{2j}$. Substituting the series in \eqref{eq:hermite-kappa} gives
\[
    b_1=\frac12,
    \qquad
    \frac{b_{j+1}}{b_j}=\frac{j-\frac14}{\bigl(j+\frac12\bigr)(j+1)},
\]
so $b_j>0$ for every $j\geq1$.

Since every term of the sum is positive, each partial sum of the series for $Y$ is an upper bound for $Y$. Evaluating the terms with $j\leq5$ at $x^2=q:=\tfrac{23}{17}$ gives
\[
    Y\!\left(\sqrt{23/17}\right)
    <
    1-\frac q2-\frac{q^2}8-\frac{7q^3}{240}-\frac{11q^4}{1920}-\frac{11q^5}{11520}
    =
    -\frac{16882063}{16356752640}
    <0.
\]
Since $Y(0)=1$, the first positive zero satisfies $\kappa^2<\frac{23}{17}$.

At $x^2=\frac43$, the first four terms of the sum have total $\frac{1186}{1215}$, the fifth is $\frac{44}{10935}$, and the ratio of every subsequent term to its predecessor is at most $\frac29$. Hence
\[
    \sum_{j\geq1}b_j\left(\frac43\right)^j
    \leq
    \frac{1186}{1215}+\frac{44}{8505}<1,
\]
so $Y(2/\sqrt3)>0$ and $\kappa^2>\frac43$.
\end{proof}

\begin{proposition}\label{prop:annulus-localization}
Let $\phi$ be the annular profile from Proposition~\ref{prop:odesol}, with support $[r_-,r_+]$. For $n\geq3$,
\begin{equation}\label{eq:radii-localization}
    2\left(\sqrt{n-1}-\kappa\right)^2<r_-^2<2\left(\frac{n-1}{\sqrt{n-1}+\kappa}\right)^2,
    \qquad
    2\left(\sqrt{n-1}+\kappa\right)^2<r_+^2<2\left(\frac{n-1}{\sqrt{n-1}-\kappa}\right)^2.
\end{equation}
In particular,
\begin{equation}\label{eq:radii-asymptotics}
    r_\pm=\sqrt{2(n-1)}\pm\sqrt2\,\kappa+O(n^{-1/2})
    \qquad\text{as }n\to\infty.
\end{equation}
For $n=2$, where $\sqrt{n-1}<\kappa$, the upper bound for $r_-$ and the lower bound for $r_+$ in \eqref{eq:radii-localization} remain valid.
\end{proposition}

The proof of Proposition \ref{prop:annulus-localization} depends on two lemmas. The first one locates the maximum point of the profile and bounds the product of the two radii; it is proved by a phase-plane argument.

\begin{lemma}\label{lem:annular-phase}
The scaled profile $y$ in \eqref{eq:annular-y-equation} has a unique critical point $c\in(a,b)$. Moreover,
\[
    c\geq\sqrt{2\mu}
    \qquad\text{and}\qquad
    ab<2\mu;
\]
in particular, $r_-r_+<2(n-1)$.
\end{lemma}

\begin{proof}
Recall that $y$ is concave, so $y'$ is strictly decreasing; since $y'(a)=1$ and $y'(b)=-1$, it follows that $y$ has exactly one critical point $c\in(a,b)$, with $y'>0$ on $[a,c)$ and $y'<0$ on $(c,b]$.

Define the oscillator energy $E(t):=y'(t)^2+y(t)^2$. The Bernoulli conditions give
\begin{equation}\label{eq:annular-energy-derivative}
    E(a)=E(b)=1.
\end{equation}
The profile $y$ is the member $y_{\alpha_0}$ of the shooting family in the proof of Proposition~\ref{prop:odesol}, so the phase representation \eqref{eq:annular-shooting-phase} applies: with $y'=\sqrt E\cos\theta$, $y=\sqrt E\sin\theta$, $\theta(a)=0$, the variable $t$ is a strictly increasing function of $\theta\in[0,\pi]$, with $t(0)=a$, $t(\pi)=b$, $t(\pi/2)=c$, and
\begin{equation}\label{eq:annular-theta-derivative}
    \theta'=1-h(t)\sin(2\theta),
\end{equation}
\begin{equation}\label{eq:annular-logE-theta-derivative}
    \frac{d}{d\theta}\log E
    =
    \frac{4h(t(\theta))\cos^2\theta}{\theta'(t(\theta))}
    =:
    I(\theta).
\end{equation}
Integrating and pairing $\theta$ with $\pi-\theta$, \eqref{eq:annular-energy-derivative} gives
\begin{equation}\label{eq:annular-logE-integral}
    0
    =
    \log E(b)-\log E(a)
    =
    \int_0^{\pi/2}
    \bigl(I(\theta)+I(\pi-\theta)\bigr)\,d\theta.
\end{equation}
For $0<\theta<\pi/2$, let
\[
    t_1:=t(\theta),
    \qquad
    t_2:=t(\pi-\theta),
    \qquad
    S:=\sin(2\theta)>0,
\]
so that $a<t_1<c<t_2<b$, and \eqref{eq:annular-theta-derivative} reads $\theta'(t_1)=1-h(t_1)S$, $\theta'(t_2)=1+h(t_2)S$. Since $\cos^2(\pi-\theta)=\cos^2\theta$ and
\[
    h(t_1)\bigl(1+h(t_2)S\bigr)+h(t_2)\bigl(1-h(t_1)S\bigr)
    =
    h(t_1)+h(t_2),
\]
the integrand $I(\theta)+I(\pi-\theta)$ in \eqref{eq:annular-logE-integral} has the same sign as
\begin{equation}\label{eq:annular-h-sum}
    h(t_1)+h(t_2)
    =
    (t_1+t_2)\left(\frac12-\frac{\mu}{t_1t_2}\right),
\end{equation}
that is, the sign of $t_1t_2-2\mu$.

Accordingly, set $P(\theta):=t_1t_2$, so that $P(0^+)=ab$ and $P(\pi/2^-)=c^2$. We claim the following crossing rule: $P'(\theta)>0$ at every $\theta\in(0,\pi/2)$ with $P(\theta)=2\mu$. Indeed,
\begin{equation}\label{eq:annular-P-derivative}
    P'
    =
    \frac{t_2}{1-h(t_1)S}
    -
    \frac{t_1}{1+h(t_2)S}
    =
    \frac{
        t_2-t_1+S\bigl(t_2h(t_2)+t_1h(t_1)\bigr)
    }{
        (1-h(t_1)S)(1+h(t_2)S)
    },
\end{equation}
and when $t_1t_2=2\mu$ we have $h(t_1)=\frac{t_1-t_2}2$ and $h(t_2)=\frac{t_2-t_1}2$, so the numerator equals $t_2-t_1+\frac S2(t_2-t_1)^2>0$.

Suppose now that $c<\sqrt{2\mu}$, so that $P(\pi/2^-)=c^2<2\mu$. Then the crossing rule forces $P<2\mu$ on all of $(0,\pi/2)$: otherwise there would be a last $\theta_0$ with $P(\theta_0)=2\mu$, and $P<2\mu$ on $(\theta_0,\pi/2)$ would give $P'(\theta_0)\leq0$. Hence $h(t_1)+h(t_2)<0$ for every $\theta$ by \eqref{eq:annular-h-sum}, and the integral in \eqref{eq:annular-logE-integral} is strictly negative, a contradiction. Thus $c\geq\sqrt{2\mu}$.

Similarly, suppose that $ab\geq2\mu$. Then $P>2\mu$ for all $\theta$ close to $0$: this follows from continuity if $ab>2\mu$, while if $ab=2\mu$ it follows from \eqref{eq:annular-P-derivative}, whose numerator tends to $b-a>0$ as $\theta\downarrow0$. If $P(\theta_1)\leq2\mu$ for some $\theta_1$, then at the first $\theta_0\leq\theta_1$ with $P(\theta_0)=2\mu$ we would have $P>2\mu$ on $(0,\theta_0)$ and hence $P'(\theta_0)\leq0$, contradicting the crossing rule. Hence $P>2\mu$ on all of $(0,\pi/2)$, and the integral in \eqref{eq:annular-logE-integral} is strictly positive, again a contradiction. Thus $ab<2\mu$.
\end{proof}

The second lemma bounds $a$ and $b$ from below by comparison with the limiting equation \eqref{eq:hermite-kappa}.

\begin{lemma}\label{lem:annular-barrier}
We have
\[
    a>\sqrt{2\mu}-\kappa
    \qquad\text{and}\qquad
    b>\sqrt{2\mu}+\kappa.
\]
\end{lemma}

\begin{proof}
Set
\[
    \widetilde y(x):=y(\sqrt{2\mu}+x),
    \qquad
    H(x):=h(\sqrt{2\mu}+x),
\]
and define the Wronskian
\[
    W(x):=\widetilde y'(x)Y(x)-\widetilde y(x)Y'(x).
\]
The equations for $\widetilde y$ and $Y$ give
\begin{align}
    W'-2xW&=2(H-x)\widetilde y'Y,
    \label{eq:annular-wronskian-left}\\
    W'-2HW&=2(H-x)\widetilde yY'.
    \label{eq:annular-wronskian-right}
\end{align}
Recall from \eqref{eq:drift-comparison} that $H(x)-x<0$ for $x\neq0$.
Moreover, $Y>0$ on $(-\kappa,\kappa)$, with $Y'>0$ on
$(-\kappa,0)$ and $Y'<0$ on $(0,\kappa)$. Indeed, $Y$ is even and
$Y'(0)=0$, while at any critical point in $\{Y>0\}$ its equation gives
$Y''=-Y<0$; hence $Y'$ cannot vanish anywhere else in
$(-\kappa,\kappa)$.

\emph{The bound $a>\sqrt{2\mu}-\kappa$.}
When $n=2$ this is immediate from $a>0$ and $\sqrt{2\mu}=1<\kappa$.
Suppose therefore that $n\geq3$ and, to the contrary, that
$a\leq\sqrt{2\mu}-\kappa$. By Lemma~\ref{lem:annular-phase},
$c\geq\sqrt{2\mu}$, and hence $\widetilde y'>0$ on $(-\kappa,0)$.
Equation~\eqref{eq:annular-wronskian-left} yields
\[
    \bigl(e^{-x^2}W\bigr)'
    =2e^{-x^2}(H-x)\widetilde y'Y<0
    \qquad\text{on }(-\kappa,0).
\]
If $a<\sqrt{2\mu}-\kappa$, then
\[
    W(-\kappa)=-\widetilde y(-\kappa)Y'(-\kappa)<0,
\]
whereas if $a=\sqrt{2\mu}-\kappa$, then both $\widetilde y$ and $Y$
vanish at $-\kappa$, so $W(-\kappa)=0$. In either case the preceding
strict inequality gives $W(0)<0$. On the other hand, $Y'(0)=0$ and
$\widetilde y'(0)\geq0$ because $c\geq\sqrt{2\mu}$, so
\[
    W(0)=\widetilde y'(0)Y(0)\geq0,
\]
a contradiction.

\emph{The bound $b>\sqrt{2\mu}+\kappa$.}
Suppose instead that $b\leq\sqrt{2\mu}+\kappa$, and set
$\delta:=b-\sqrt{2\mu}\in(0,\kappa]$. Let
\[
    M(x):=\exp\left(-2\int_0^x H(s)\,ds\right).
\]
Since $\widetilde y>0$ and $Y'<0$ on $(0,\delta)$,
\eqref{eq:annular-wronskian-right} gives
\[
    (MW)'=2M(H-x)\widetilde yY'>0
    \qquad\text{on }(0,\delta).
\]
If $\delta<\kappa$, then
\[
    W(\delta)=\widetilde y'(\delta)Y(\delta)<0,
\]
while if $\delta=\kappa$, then both $\widetilde y$ and $Y$ vanish at
$\delta$, so $W(\delta)=0$. Since $MW$ is strictly increasing and
$M(0)=1$, it follows that $W(0)<0$. This again contradicts
\[
    W(0)=\widetilde y'(0)Y(0)\geq0.
\]
\end{proof}

\begin{proof}[Proof of Proposition~\ref{prop:annulus-localization}]
Since $r_-=\sqrt2\,a$ and $r_+=\sqrt2\,b$, in the scaled variables \eqref{eq:annular-scaled-notation} the bounds \eqref{eq:radii-localization} read
\[
    \sqrt{2\mu}-\kappa<a<\frac{2\mu}{\sqrt{2\mu}+\kappa},
    \qquad
    \sqrt{2\mu}+\kappa<b<\frac{2\mu}{\sqrt{2\mu}-\kappa}.
\]
The two lower bounds are the content of Lemma~\ref{lem:annular-barrier}. Combining them with $ab<2\mu$ from Lemma~\ref{lem:annular-phase} gives
\[
    a<\frac{2\mu}b<\frac{2\mu}{\sqrt{2\mu}+\kappa},
    \qquad
    b<\frac{2\mu}a<\frac{2\mu}{\sqrt{2\mu}-\kappa}.
\]
The second chain uses $a>\sqrt{2\mu}-\kappa>0$ and hence requires $n\geq3$; the first does not, which proves the claim for $n=2$. 
Finally, the two enclosing bounds for $a$, and likewise those for $b$, differ by $\kappa^2/(\sqrt{n-1}\pm\kappa)=O(n^{-1/2})$, which gives \eqref{eq:radii-asymptotics}.

\end{proof}

Next, we bound the width of the annulus. The bound is dimension-free, and together with the product bound of Lemma~\ref{lem:annular-phase} it is the only radius information used to classify the $\ell\geq4$ modes in Section \ref{sec:instability}. The eigenvalue-comparison strategy is outlined in Section \ref{ssec:main-ideas}.

\begin{proposition}\label{prop:annulus-width}
Let $\phi$ be the annular profile from Proposition~\ref{prop:odesol}, with support $[r_-,r_+]$. For every $n\geq2$,
\[
    r_+-r_-<\sqrt{14};
\]
equivalently, in the scaled variables \eqref{eq:annular-scaled-notation}, $b-a<\sqrt7$.
\end{proposition}

\begin{proof}
Write $p:=ab$ and
\[
    \varrho(t):=t^{2\mu}e^{-t^2/2},
\]
so that $\varrho'=-2h\varrho$ and the equation \eqref{eq:annular-y-equation} becomes divergence form $(\varrho\,y')'+\varrho\,y=0$. By Lemma~\ref{lem:annular-phase}, $p<2\mu$ for every $n\geq2$; for $n\geq3$, Lemmas \ref{lem:annular-barrier} and \ref{lem:kappa-bounds} complement this with a lower bound, so that $p$ lies in the window
\begin{equation}\label{eq:product-window}
    2\mu-\frac{23}{17}
    <2\mu-\kappa^2
    =\bigl(\sqrt{2\mu}-\kappa\bigr)\bigl(\sqrt{2\mu}+\kappa\bigr)
    <p<2\mu.
\end{equation}

Suppose for contradiction that $b-a\geq\sqrt7$. For fixed product $p$, the endpoints of the interval with product $p$ and width $d$ are $\frac12\bigl(\sqrt{d^2+4p}\mp d\bigr)$, the first decreasing and the second increasing in $d$. Hence the interval $[\alpha,\beta]$ determined by
\[
    \beta-\alpha=\sqrt7,
    \qquad
    \alpha\beta=p,
\]
satisfies $[\alpha,\beta]\subset[a,b]$, and $\alpha>0$.

\emph{Step 1: reduction to a trial function.} Let $\lambda_1$ be the first eigenvalue of the Dirichlet problem
\begin{equation}\label{eq:width-eigen}
    -(\varrho\,\eta')'=\lambda\,\varrho\,\eta
    \quad\text{on }(\alpha,\beta),
    \qquad
    \eta(\alpha)=\eta(\beta)=0.
\end{equation}
Since $y>0$ on $(a,b)$ and solves $-(\varrho y')'=\varrho y$ with $y(a)=y(b)=0$, the first Dirichlet eigenvalue on $(a,b)$ is $1$. Hence domain monotonicity and $[\alpha,\beta]\subset[a,b]$ give $\lambda_1\geq1$. It therefore suffices to show $\lambda_1<1$.

The Liouville substitution $v=\varrho^{1/2}\eta$ transforms \eqref{eq:width-eigen} into the Dirichlet problem for $-v''+V_n\,v=\lambda v$ on $(\alpha,\beta)$, with
\begin{equation}\label{eq:width-potential}
    V_n(t)=\frac{t^2}{4}+\frac{\mu(\mu-1)}{t^2}-\mu-\frac12,
\end{equation}
so it suffices to exhibit $\psi\in H^1_0((\alpha,\beta))$ with
\begin{equation}\label{eq:width-rayleigh}
    \frac{\displaystyle\int_\alpha^\beta\bigl[(\psi')^2+V_n\,\psi^2\bigr]\,dt}
         {\displaystyle\int_\alpha^\beta\psi^2\,dt}<1 .
\end{equation}
Write
\[
    L:=\sqrt7,
    \qquad
    m:=\frac{\alpha+\beta}2=\sqrt{p+\frac74},
    \qquad
    t=m+x,
    \quad
    -\frac L2\leq x\leq\frac L2 .
\]

\emph{Step 2: the case $n\geq3$.} Take
\[
    \psi(x):=\cos\frac{\pi x}{L},
\]
and let $\langle F\rangle:=\int F\psi^2\,dx\big/\!\int\psi^2\,dx$ denote the normalized averages. Elementary integration gives
\begin{equation}\label{eq:width-cos-moments}
    \frac{\int(\psi')^2}{\int\psi^2}=\frac{\pi^2}{7},
    \qquad
    \langle t^2\rangle=m^2+\vartheta\,\frac{L^2}4=p+C,
    \qquad
    C:=\frac73-\frac7{2\pi^2},
    \qquad
    \vartheta:=\frac13-\frac2{\pi^2}.
\end{equation}
Since $|x|\leq\frac L2<m$, pairing $x$ with $-x$ gives the convergent expansion
\[
    \langle t^{-2}\rangle
    =\frac1{m^2}\biggl\{1+\sum_{j\geq1}(2j+1)\Bigl\langle\Bigl(\frac xm\Bigr)^{2j}\Bigr\rangle\biggr\},
\]
and $\langle x^{2j}\rangle\leq\bigl(\tfrac L2\bigr)^{2j-2}\langle x^2\rangle=\vartheta\bigl(\tfrac L2\bigr)^{2j}$ yields, after summing the resulting series,
\begin{equation}\label{eq:width-inverse-upper}
    \langle t^{-2}\rangle\leq J_\vartheta(p),
    \qquad
    J_\vartheta(p):=\frac4{4p+7}\left(1+\vartheta\,\frac{7(6p+7)}{8p^2}\right)
    =\frac{4(1-\vartheta)}{4p+7}+\frac\vartheta p+\frac{7\vartheta}{2p^2}.
\end{equation}
Since $\mu(\mu-1)\geq0$ for $n\geq3$, the Rayleigh quotient in \eqref{eq:width-rayleigh} is at most
\begin{equation}\label{eq:width-Q}
    \mathcal Q(p):=\frac{\pi^2}7+\frac{p+C}4+\mu(\mu-1)\,J_\vartheta(p)-\mu-\frac12 .
\end{equation}

We claim that $\mathcal Q$ is increasing throughout the window \eqref{eq:product-window}. The second form of $J_\vartheta$ in \eqref{eq:width-inverse-upper} shows $J_\vartheta''>0$, and $\pi^2<10$ gives $\vartheta<\frac2{15}$. Since $\partial_\vartheta J_\vartheta'(p)=\frac{16}{(4p+7)^2}-\frac1{p^2}-\frac7{p^3}<0$,
\[
    4\,\mathcal Q'(p)
    =1+4\mu(\mu-1)\,J_\vartheta'(p)
    \geq 1+(n-1)(n-3)\,J_{2/15}'(p),
\]
with strict inequality when $n\geq 4$, and the right-hand side is increasing in $p$. For $n=3$ it equals $1$. For $n\geq4$, set $B:=n-4\geq0$; clearing denominators at the left endpoint $p=2\mu-\frac{23}{17}$ of the window gives
\[
    1+(n-1)(n-3)\,J_{2/15}'\!\left(2\mu-\frac{23}{17}\right)
    =\frac{P(B)}{15\,(17B+28)^3\,(68B+231)^2}>0,
\]
where
\[
    P(B):=475067448\,B^4+2882697837\,B^3+5070969134\,B^2+2143250690\,B+10362030 .
\]
Hence $\mathcal Q(p)<\mathcal Q(2\mu)$ throughout the window, by the upper bound in \eqref{eq:product-window}. Finally, since $\partial_\vartheta J_\vartheta=\frac1p+\frac7{2p^2}-\frac4{4p+7}>0$, the bounds $\vartheta<\frac2{15}$ and $\pi^2<10$ give
\[
    4\bigl(1-\mathcal Q(2\mu)\bigr)
    >n+5-\frac{3233}{420}-(n-1)(n-3)\,J_{2/15}(n-1)
    =\frac{2272\,n^2-8555\,n+9027}{420\,(n-1)(4n+3)}>0,
\]
where the quadratic in the numerator has negative discriminant. Thus the Rayleigh quotient of $\psi$ is at most $\mathcal Q(p)<\mathcal Q(2\mu)<1$, and Step~1 applies.

\emph{Step 3: the case $n=2$.} Here $\mu=\frac12$, and two features of Step~2 change. First, the coefficient $\mu(\mu-1)=-\frac14$ of the inverse-square term in \eqref{eq:width-potential} is now negative, so an upper bound such as \eqref{eq:width-inverse-upper} no longer controls the Rayleigh quotient, and we need a lower bound on the average of $t^{-2}$ instead. Second, the lower bound in \eqref{eq:product-window} becomes vacuous, since $2\mu-\kappa^2<0$; the only property of the product $p$ we use is $0<p<2\mu=1$. Put $k:=\pi/L$ and take the two-mode trial function
\[
    \psi(x):=\cos(kx)-\frac1{10}\sin(2kx),
\]
which vanishes at $x=\pm\frac L2$. Let
\[
    K:=\frac{\int(\psi')^2}{\int\psi^2},
    \qquad
    T:=\frac{\int t^2\psi^2}{\int\psi^2},
    \qquad
    S:=\frac{\int t^{-2}\psi^2}{\int\psi^2}.
\]
Orthogonality and direct integration give
\begin{equation}\label{eq:width-two-mode-KT}
    K=\frac{104}{101}\cdot\frac{\pi^2}7,
    \qquad
    T=m^2+\frac7{12}-\frac{2807}{808\,\pi^2}-\frac{640\,mL}{909\,\pi^2}.
\end{equation}
Since $\pi^2>9$ and $L<3$ give $\frac{640L}{909\pi^2}<1<2m$, the expression for $T$ is increasing in $m$; as $m^2=p+\frac74\leq\frac{11}4$, we may evaluate it at $m^2=\frac{11}4$, where $mL=\frac{\sqrt{77}}2$. Using $\pi<\frac{22}7$ and $\sqrt{77}>\frac{877}{100}$,
\[
    K=\frac{104}{707}\,\pi^2<\frac{104}{707}\cdot\frac{484}{49}=\frac{50336}{34643}<\frac{1453}{1000},
    \qquad
    T<\frac{11}4+\frac7{12}-\frac{49}{484}\left(\frac{2807}{808}+\frac{14032}{4545}\right)<\frac{267}{100}.
\]
For a lower bound on $S$, write
\[
    E(x):=\cos^2(kx)+\frac1{100}\sin^2(2kx),
\]
so that $\psi^2=E-\frac15\cos(kx)\sin(2kx)$ and $\int\psi^2=\int E$. The omitted cross term is odd; on $\bigl(0,\frac L2\bigr)$ it is nonpositive, while $(m+x)^{-2}-(m-x)^{-2}<0$ there, so the cross term contributes nonnegatively to $\int t^{-2}\psi^2$. Pairing $x$ with $-x$ as in Step~2 and keeping the first three (positive) terms of the expansion therefore gives
\begin{equation}\label{eq:width-two-mode-S}
    S\geq\frac1{m^2}+\frac{3M_2}{m^4}+\frac{5M_4}{m^6},
    \qquad
    M_2:=\frac{\int x^2E}{\int E},
    \qquad
    M_4:=\frac{\int x^4E}{\int E},
\end{equation}
where direct integration gives
\[
    M_2=\frac7{12}-\frac{2807}{808\,\pi^2},
    \qquad
    M_4=49\left(\frac1{80}-\frac{401}{1616\,\pi^2}+\frac{4803}{3232\,\pi^4}\right).
\]
Both $M_2$ and $M_4$ are positive, so the right-hand side of \eqref{eq:width-two-mode-S} is decreasing in $m^2$ and may be evaluated at $m^2=\frac{11}4$, where it equals
\[
    \frac4{11}+\frac{48}{121}\,M_2+\frac{320}{1331}\,M_4 .
\]
As a function of $y:=\pi^2$, this expression has derivative $\frac{14\,(41303\,y-336210)}{134431\,y^3}$, which is positive for $y\geq9$; substituting the lower bound $\pi>\frac{333}{106}$ therefore yields
\[
    S>\frac{267896812791500}{551004452874117}>\frac{243}{500}.
\]
For $n=2$ the Rayleigh quotient in \eqref{eq:width-rayleigh} is
\[
    K+\frac T4-\frac S4-1
    <\frac{1453}{1000}+\frac{267}{400}-\frac{243}{2000}-1
    =\frac{999}{1000}<1,
\]
and Step~1 applies.
\end{proof}

The next result is the outward bias, Theorem \ref{thm:annulus-geometry} \ref{it:geom-bias}, in the rescaled variables. Together with the width bound, it is the geometric information needed in the study of the $\ell=3$ mode in Section~\ref{sec:instability}. Write
\begin{equation}\label{def:asymmetry}
    m:=\frac{a+b}2,
    \qquad
    d:=\frac{b-a}2,
    \qquad
    R:=m^2-2\mu.
\end{equation}
The proof extracts an identity for the excess $R$ from the mismatch between the even boundary data \eqref{eq:annular-y-data} and the asymmetric equation \eqref{eq:annular-y-equation}; the strategy is outlined in Section \ref{ssec:main-ideas}.

\begin{lemma}\label{lem:annular-midpoint}
For every $n\geq2$,
\[
    0<R<d^2
    \qquad\text{and}\qquad
    R>\gamma_d:=\frac{\int_0^d x^4e^{-x^2}\,dx}{\int_0^d x^2e^{-x^2}\,dx}\,.
\]
Moreover, $d\mapsto\gamma_d$ is strictly increasing.
\end{lemma}

\begin{proof}
The monotonicity of $\gamma_d$ is immediate: $\gamma_d<d^2$, and
\[
    \gamma_d'
    =\frac{d^2e^{-d^2}\,(d^2-\gamma_d)}{\int_0^d x^2e^{-x^2}\,dx}>0.
\]

\emph{Step 1: the curvature is strictly convex.} Let $K:=-y''$, which is positive by Claim \ref{cl:concave} and satisfies the equation \eqref{eq:y-2ndder} 
\[
    K''+\left(\frac{n+1}{t}-t\right)K'-3K=0;
\]
while the endpoint data give
\[
    K'(a)=-\left(a-\frac{2\mu}{a}\right)^2-\frac{2\mu}{a^2}<0,
    \qquad
    K'(b)=\left(b-\frac{2\mu}{b}\right)^2+\frac{2\mu}{b^2}>0.
\]
At every zero of $K'$ one has $K''=3K>0$, so $K'$ has exactly one zero $t_0$. Put $B(t):=\frac{n+1}{t}-t$. Let $L:=K''$. Differentiating the equation for $K$ gives
\[
    L'+BL+(B'-3)K'=0.
\]
At a zero of $L$,
\[
    L'=(3-B')K'=\left(4+\frac{n+1}{t^2}\right)K'.
\]
Since $L(t_0)=3K(t_0)>0$, a last zero of $L$ before $t_0$ would have nonnegative derivative, contradicting $K'<0$ there. Likewise, a first zero after $t_0$ would have nonpositive derivative, contradicting $K'>0$. Hence $L>0$ on $(a,b)$, so $K$ is strictly convex.

\emph{Step 2: the midpoint identity.} For $0<x<d$ set
\[
    P(x):=\frac12\int_{m-x}^{m+x}K(t)\,dt>0,
    \qquad
    \Delta:=m^2-x^2,
    \qquad
    \sigma(x):=e^{-x^2/2}\left(1-\frac{x^2}{m^2}\right)^{\mu}.
\]
Taking the odd part of \eqref{eq:annular-y-equation} at the reflected points $m\pm x$ shows that $\mathcal O(x):=\tfrac12\bigl(y(m+x)-y(m-x)\bigr)$ satisfies
\begin{equation}\label{eq:annular-odd-part}
    (\sigma\mathcal O')'+\sigma\mathcal O
    =-m\frac{R-x^2}{\Delta}\,\sigma P.
\end{equation}
Consider $\chi$, the solution of $(\sigma\chi')'+\sigma\chi=0$ with $\chi(0)=0$, $\chi'(0)=1$.

Until the first critical point of $\chi$,
\[
    (\chi'\sin x-\chi\cos x)'=-\frac{\sigma'}{\sigma}\,\chi'\sin x>0,
\]
so that critical point lies beyond $\pi/2$. Indeed, if \(x_*\leq\pi/2\) were the first critical point of \(\chi\), then \(\chi'>0\) and \(\chi>0\) on \((0,x_*)\), so \(\mathcal W:=\chi'\sin x-\chi\cos x\), with \(\mathcal W(0)=0\), would satisfy \(\mathcal W(x_*)>0\). However, \(\chi'(x_*)=0\) gives \(\mathcal W(x_*)=-\chi(x_*)\cos x_*\leq0\), a contradiction. 

Since $d<\sqrt7/2<\pi/2$ by Proposition~\ref{prop:annulus-width}, $\chi>0$ on $(0,d]$.
Multiplying the equation for $\mathcal O$ by $\chi$ and subtracting
$\mathcal O$ times the equation for $\chi$ gives
\[
    \bigl[\sigma(\chi\mathcal O'-\mathcal O\chi')\bigr]'
    =-m(R-x^2)\frac{\sigma\chi P}{\Delta}.
\]
The boundary term vanishes at $x=0$ and $x=d$, because
$\chi(0)=\mathcal O(0)=0$ and $\mathcal O(d)=\mathcal O'(d)=0$.
Integrating the preceding identity therefore gives
\[
    \int_0^d (R-x^2)\frac{\sigma\chi P}{\Delta}\,dx=0,
\]
and hence
\begin{equation}\label{eq:annular-midpoint}
    R=\frac{\int_0^d x^2\,W(x)\,dx}{\int_0^d W(x)\,dx},
    \qquad
    \text{ where } W:=\frac{\sigma\chi P}{\Delta}>0.
\end{equation}
In particular, since $0<x^2<d^2$ on $(0,d)$, this identity gives $0<R<d^2$.

\emph{Step 3: comparison.} The ratio $P(x)/x$ is increasing, because the symmetric average $P(x)/x$ of the strictly convex positive function $K$ is increasing in the radius. Also $\chi/x$ is increasing: indeed, if $z:=x\chi'-\chi$, then
\[
    (\sigma z)'=\frac{2\mu x}{\Delta}\,\sigma\chi>0.
\]
Thus $z>0$, and hence $(\chi/x)'>0$.
Finally,
\[
    \frac{d}{dx}\log\frac{\sigma/\Delta}{e^{-x^2}}
    =\frac{x(R+2-x^2)}{\Delta}>0.
\]
Thus
\[
    \frac{W}{x^2e^{-x^2}}
    =\frac{P}{x}\,\frac{\chi}{x}\,
      \frac{\sigma/\Delta}{e^{-x^2}}
\]
is strictly increasing, and the Chebyshev's integral inequality  in \eqref{eq:annular-midpoint} gives $R>\gamma_d$.
\end{proof}

\section{Positivity of the moment}\label{sec:l3-details}
This appendix proves the moment inequality \eqref{eq:l3-moment}, in the rescaled variables \eqref{eq:annular-scaled-notation}. 
The discussion in Subsection \ref{ssec:l3} indicates the sign of the moment integral depends on the length of the interval and the outward bias of the annulus. In the rescaled variables, these parameters are denoted as 
\begin{equation}\label{eq:l3-mdr}
    m:=\frac{a+b}{2},\qquad d:=\frac{b-a}{2},
    \qquad R:=m^2-(n-1),
\end{equation}
see \eqref{def:asymmetry}. Thus $a=m-d$, $b=m+d$, and $m^2=n-1+R$. 
The change of variable $r=\sqrt2\,t$ turns the
moment integral in \eqref{eq:l3-moment} into a positive multiple of
\begin{equation}\label{eq:l3-moment-scaled}
    \mathcal D_n(R,d):=
    \int_{\sqrt{n-1+R}-d}^{\sqrt{n-1+R}+d}
    \bigl(t^2-(n+2)\bigr)\,t^{n+3}e^{-t^2/2}\,dt.
\end{equation}
Thus it suffices to prove $\mathcal D_n(R,d)>0$ in every dimension.

\begin{proposition}\label{prop:l3-moment-positive}
For every $n\geq2$, the parameters $R$ and $d$ defined in
\eqref{eq:l3-mdr} satisfy $\mathcal D_n(R,d)>0$.
\end{proposition}

The following elementary bound is used twice below: for every $z>0$,
\begin{equation}\label{eq:l3-T}
    T(z):=\int_0^1e^{z(1-u^2)}\,du
    \geq
    1+\frac23z+\frac4{15}z^2+\frac8{105}z^3+\frac{16}{945}z^4,
\end{equation}
since every term of the expansion of $e^{z(1-u^2)}$ is positive and $\int_0^1(1-u^2)^j\,du=\frac23,\frac8{15},\frac{16}{35},\frac{128}{315}$ for $j=1,\dots,4$.

\begin{lemma}\label{lem:l3-monotone}
Let $n\geq2$, $\frac35\leq R<3$, and $d\geq\frac87$ with
$\sqrt{n-1+R}>d$. Then
\[
    \partial_R\mathcal D_n(R,d)>0,
    \qquad
    \partial_d\mathcal D_n(R,d)>0.
\]
\end{lemma}

\begin{proof}
Set $m:=\sqrt{n-1+R}$, and write
$g(t)=t^{n+3}e^{-t^2/2}$ and $F(t)=\bigl(t^2-(n+2)\bigr)g(t)$ for the
integrand of \eqref{eq:l3-moment-scaled}, so that
$\mathcal D_n(R,d)=\int_{m-d}^{m+d}F(t)\,dt$. Then $F(m-d)<0$, since
$(m-d)^2<m^2=n-1+R<n+2$.

For $0<x\leq d$, direct pairing gives
\begin{equation}\label{eq:l3-paired-integrand}
\begin{aligned}
    F(m+x)+F(m-x)
    &=(R+x^2-3)\bigl(g(m+x)+g(m-x)\bigr)\\
    &\quad+2mx\bigl(g(m+x)-g(m-x)\bigr).
\end{aligned}
\end{equation}
We claim that the second term satisfies
\begin{equation}\label{eq:l3-weight-asymmetry}
    m\bigl(g(m+x)-g(m-x)\bigr)
    >x\bigl(g(m+x)+g(m-x)\bigr).
\end{equation}
Indeed,
\[
\begin{aligned}
    \frac{d}{dx}\log\frac{(m-x)g(m+x)}{(m+x)g(m-x)}
    &=(n+2)\left(\frac1{m+x}+\frac1{m-x}\right)-2m\\
    &=2m\left(\frac{n+2}{m^2-x^2}-1\right)>0,
\end{aligned}
\]
because $m^2=n-1+R<n+2$ by \eqref{eq:l3-mdr} and the assumption $R<3$.
The logarithm vanishes at $x=0$, so
$(m-x)g(m+x)>(m+x)g(m-x)$, which is
\eqref{eq:l3-weight-asymmetry}.
Combining \eqref{eq:l3-paired-integrand} and
\eqref{eq:l3-weight-asymmetry} at $x=d$, we obtain
\[
    F(m+d)+F(m-d)
    >\bigl(R+3d^2-3\bigr)\bigl(g(m+d)+g(m-d)\bigr)>0,
\]
because
\[
    R+3d^2-3\geq\frac35+3\left(\frac87\right)^2-3
    =\frac{372}{245}>0.
\]
Together with $F(m-d)<0$, this also gives $F(m+d)>0$. Consequently
\[
 \partial_R\mathcal D_n=\frac{F(m+d)-F(m-d)}{2m}>0,\qquad
 \partial_d\mathcal D_n=F(m-d)+F(m+d)>0.
\]
\end{proof}

\begin{lemma}\label{lem:l3-corner}
For every $n\geq4001$,
\[
    \mathcal D_n\left(\frac35,\frac87\right)>0.
\]
\end{lemma}

\begin{proof}
Set
\begin{equation}\label{eq:l3-corner-parameters}
    R_0:=\frac35,\qquad d_0:=\frac87,\qquad
    c:=4-R_0=\frac{17}{5},\qquad
    \eta:=\frac1{n-1+R_0}<\frac1{4000},\qquad s:=\sqrt\eta.
\end{equation}
Also write $g(t):=t^{n+3}e^{-t^2/2}$ for the Gaussian factor of the
integrand of \eqref{eq:l3-moment-scaled}, so that
$\eta^{-1/2}=\sqrt{n-1+R_0}$ is the midpoint of the interval of
integration in \eqref{eq:l3-moment-scaled}. Pairing the points
$\eta^{-1/2}\pm x$ there gives
\begin{equation}\label{eq:l3-paired-normal-form}
 J(\eta):=\frac{\mathcal D_n(R_0,d_0)}{2g(\eta^{-1/2})}
 =\int_0^{d_0}i(\eta,x)\,dx,
\end{equation}
where
\[
 i(\eta,x)=\frac12\sum_{\pm}
 \left(R_0-3+x^2\pm\frac{2x}{s}\right)
 (1\pm xs)^{1/s^2+c}e^{\mp x/s-x^2/2}.
\]

To identify the limit as $\eta\downarrow0$, define
\[
 h(w):=\frac{\log(1+w)-w+w^2/2}{w^2}
 =\frac w3-\frac{w^2}4+\cdots,\qquad
 G_x(s):=\exp\{x^2h(xs)+c\log(1+xs)\},
\]
where $h(0):=0$. Then $i(\eta,x)=e^{-x^2}\Phi_x(s)$, where
\begin{equation}\label{eq:l3-Phi}
 \Phi_x(s):=(R_0-3+x^2)\frac{G_x(s)+G_x(-s)}2
 +\frac{x}{s}\bigl(G_x(s)-G_x(-s)\bigr).
\end{equation}
The right-hand side extends smoothly and evenly across $s=0$. Since
$G_x(0)=1$ and $G_x'(0)=x(c+x^2/3)$,
\[
 i(0,x)=e^{-x^2}\left(-\frac{12}{5}+\frac{39}{5}x^2
 +\frac23x^4\right).
\]
Let $z:=d_0^2=64/49$ and $E:=e^{-z}$. Gaussian integration by parts gives
\[
 J(0)=2d_0E\left[T(z)-\frac{11}{5}-\frac z6\right].
\]
Substitution in \eqref{eq:l3-T} gives
$T(z)-\frac{11}{5}-\frac z6>\frac{49}{400}$, while the alternating
series for $e^{-z}$ gives $E>1/4$. Therefore
\begin{equation}\label{eq:l3-J0}
    J(0)>2\cdot\frac87\cdot\frac14\cdot\frac{49}{400}
    =\frac7{100}.
\end{equation}

It remains to control the error. By \eqref{eq:l3-corner-parameters} and
\eqref{eq:l3-paired-normal-form}, $s<1/63$ and $0\leq x\leq8/7$, so
we have $|xs|<1/50$. The power series for $h$ gives
\begin{equation}\label{eq:l3-h-bounds}
 |h|<\frac1{100},\qquad |h'|<\frac38,\qquad
 |h''|<\frac58,\qquad |h'''|<2
 \qquad (|w|<1/50).
\end{equation}
Put $L_x:=\log G_x$. Using \eqref{eq:l3-corner-parameters},
\eqref{eq:l3-h-bounds}, $x<6/5$, $c<7/2$, and
$(1+xs)^{-1}<50/49$, direct differentiation gives
\[
 |L_x|<\frac1{10},\qquad |L_x'|<5,\qquad
 |L_x''|<7,\qquad |L_x'''|<18.
\]
In particular, $0<G_x<6/5$, and the identities
\[
 G_x''=G_x\bigl(L_x''+(L_x')^2\bigr),\qquad
 G_x'''=G_x\bigl(L_x'''+3L_x'L_x''+(L_x')^3\bigr)
\]
yield
\begin{equation}\label{eq:l3-G-derivative-bounds}
 |G_x''|<39,\qquad |G_x'''|<298.
\end{equation}

Finally, rewrite the second term in \eqref{eq:l3-Phi} as
\begin{equation}\label{eq:l3-symmetric-difference}
 \frac{x}{s}\bigl(G_x(s)-G_x(-s)\bigr)
 =x\int_{-1}^1G_x'(s\tau)\,d\tau.
\end{equation}
Using \eqref{eq:l3-Phi} and \eqref{eq:l3-symmetric-difference}, the evenness
of $\Phi_x$, Taylor's theorem, \eqref{eq:l3-G-derivative-bounds}, and
$|R_0-3+x^2|\leq3-R_0$ imply
\begin{align}
 |J(\eta)-J(0)|
 &\leq d_0\left((3-R_0)\frac{39}{2}
 +d_0\frac{298}{3}\right)\eta\notag\\
 &=\frac{134672}{735}\eta<184\eta<\frac{23}{500}.
 \label{eq:l3-error-bound}
\end{align}
Combining \eqref{eq:l3-J0} and \eqref{eq:l3-error-bound}, we obtain
\[
    J(\eta)>\frac7{100}-\frac{23}{500}
    =\frac3{125}>0,
\]
and hence $\mathcal D_n(R_0,d_0)>0$ by
\eqref{eq:l3-paired-normal-form}.
\end{proof}

\begin{lemma}\label{lem:l3-tail}
$\mathcal D_n(R,d)>0$ for every $n\geq4001$.
\end{lemma}

\begin{proof}
By Proposition~\ref{prop:annulus-localization}, $b>\sqrt{n-1}+\kappa$ and
$a<(n-1)/(\sqrt{n-1}+\kappa)$. Using $d=(b-a)/2$ from
\eqref{eq:l3-mdr}, we obtain
\[
    d=\frac{b-a}2>\frac{(\sqrt{n-1}+\kappa)^2-(n-1)}{2(\sqrt{n-1}+\kappa)}
    =\kappa-\frac{\kappa^2}{2(\sqrt{n-1}+\kappa)}.
\]
By Lemma~\ref{lem:kappa-bounds},
$\kappa>\frac2{\sqrt3}>\frac{15}{13}$ and
$\kappa^2<\frac{23}{17}$. Since $\sqrt{n-1}>63$, it follows that
\begin{equation}\label{eq:l3-tail-d-lower}
    d>\frac{15}{13}-\frac{23}{34\cdot63}
    =\frac87+\frac1{3978}>\frac87.
\end{equation}
Next, \eqref{eq:l3-tail-d-lower} and Lemma~\ref{lem:annular-midpoint} give
$R>\gamma_d>\gamma_{8/7}$. Integration by parts shows that, for
$z=64/49$,
\[
    \gamma_{8/7}>\frac35
    \quad\Longleftrightarrow\quad
    T(z)>1+\frac{10}{9}z.
\]
The latter follows from the terms through order three in \eqref{eq:l3-T},
since $6z^2+21z-35=6397/2401>0$. Thus
\begin{equation}\label{eq:l3-tail-R-lower}
    R>\frac35.
\end{equation}

Finally, $R<d^2<\frac74$ by Lemma~\ref{lem:annular-midpoint} and
Proposition~\ref{prop:annulus-width}. Moreover,
$\sqrt{n-1+R}=m>d$, since $a=m-d>0$ by \eqref{eq:l3-mdr}.
Together with \eqref{eq:l3-tail-d-lower} and \eqref{eq:l3-tail-R-lower},
these bounds verify all the hypotheses of Lemma~\ref{lem:l3-monotone}, and
Lemma~\ref{lem:l3-corner} then gives
\[
    \mathcal D_n(R,d)
    >\mathcal D_n\left(\frac35,\frac87\right)>0. \qedhere
\]
\end{proof}

\begin{lemma}\label{lem:l3-finite}
$\mathcal D_n(R,d)>0$ for every $2\leq n\leq4000$.
\end{lemma}

\begin{proof}
The verifier and its captured output accompany the paper as the ancillary
files \path{verify_lambda31.py} and
\path{verify_lambda31_output.txt}. It uses python-flint~0.9.0
(FLINT~3.6.0, which incorporates the Arb library \cite{arb}) at $160$-bit
working precision. Every quantity entering the computation is rational, and every computed ball contains
the exact value of the expression it represents. For each
$2\leq n\leq4000$, the verifier selects rational numbers
$A_n<B_n$ and certifies the following three facts:
\begin{enumerate}[label=\textup{(\roman*)}]
    \item $A_n^2<n+2<B_n^2$;
    \item the forward solution from $(A_n,0,1)$ returns to zero with slope
    greater than $-1+\frac{23}{10000}$, while the backward solution from
    $(B_n,0,-1)$ returns to zero with slope greater than
    $1+\frac{3}{2000}$;
    \item
    \[
        \frac{1}{B_n^{n+3}e^{-B_n^2/2}}
        \int_{A_n}^{B_n}\bigl(t^2-(n+2)\bigr)
        t^{n+3}e^{-t^2/2}\,dt>\frac3{400}.
    \]
\end{enumerate}
The first-return property in (ii), together with
Corollary~\ref{cor:shooting-bracket}, implies $A_n<a$ and $B_n<b$.
We describe how these inequalities are certified.

The two shooting problems for \eqref{eq:annular-y-equation} are integrated
by Taylor expansions of degree $14$ on steps of length $1/100$. If
$t+\tau-(n-1)/(t+\tau)=\sum_jp_j\tau^j$ at the expansion point $t$, the Taylor coefficients
are generated by the exact recurrence
\begin{equation}\label{eq:l3-taylor-recurrence}
\begin{aligned}
    y_{j+1}=\frac{v_j}{j+1},
    \qquad
    v_{j+1}=\frac{\sum_{\ell=0}^jp_\ell v_{j-\ell}-y_j}{j+1}.
\end{aligned}
\end{equation}
On a step of length $\Delta t$, with $L=1+\sup|t-(n-1)/t|$, the a priori bound
\begin{equation}\label{eq:l3-taylor-apriori}
    \max\{|y|,|v|\}
    \leq
    \frac{\max\{|y(0)|,|v(0)|\}}{1-|\Delta t|L}
\end{equation}
holds at every point of the step. Rerunning \eqref{eq:l3-taylor-recurrence}
from the enclosure \eqref{eq:l3-taylor-apriori}, with each $p_\ell$ enclosed
over the whole step, encloses the Taylor coefficient of order $15$ at every
intermediate point, and hence the remainder in Lagrange form. The verification
therefore relies on no numerical routine beyond ball arithmetic itself.

The solution is certified positive at every grid point before its sign
change and remains in the region $100t^2\geq n-1$. In this region the
normal form $u''+Qu=0$, $u=we^{-\int h}$, with
$h(t)=\frac12(t-(n-1)/t)$, has
\[
    Q=1-h^2+h'\leq\frac32+\frac{n-1}{2t^2}\leq\frac{103}{2},
\]
so any two zeros of $w$ are more than $\pi\sqrt{2/103}>\frac13$ apart.
Since the grid step is $1/100$, the certified sign change is therefore the
first return to zero. This proves (ii).

Finally, consider the system
\begin{equation}\label{eq:l3-moment-system}
\begin{aligned}
    w'=\left(\frac{n+3}{t}-t\right)w,
    \qquad
    M'=(t^2-n-2)\,w,
    \qquad
    w(B_n)=1,\quad M(B_n)=0,
\end{aligned}
\end{equation}
Applying the same Taylor scheme backwards from $B_n$ to
\eqref{eq:l3-moment-system} gives (iii), because
\[
    w(t)=\frac{t^{n+3}e^{-t^2/2}}{B_n^{n+3}e^{-B_n^2/2}},
    \qquad
    -M(A_n)=\frac{1}{B_n^{n+3}e^{-B_n^2/2}}
    \int_{A_n}^{B_n}\bigl(t^2-(n+2)\bigr)
    t^{n+3}e^{-t^2/2}\,dt.
\]

It remains to compare the endpoints. The integrand of
\eqref{eq:l3-moment-scaled} is negative below $\sqrt{n+2}$ and positive
above it, while $a^2<ab<n-1$ by Lemma~\ref{lem:annular-phase} and
$b>B_n>\sqrt{n+2}$. Increasing the endpoints from $(A_n,B_n)$ to $(a,b)$ therefore increases the moment:
\[
    \mathcal D_n(R,d)\geq
    \int_{A_n}^{B_n}\bigl(t^2-(n+2)\bigr)
    t^{n+3}e^{-t^2/2}\,dt>0. \qedhere
\]
\end{proof}

\begin{proof}[Proof of Proposition~\ref{prop:l3-moment-positive}]
Lemmas~\ref{lem:l3-tail} and~\ref{lem:l3-finite} give
$\mathcal D_n(R,d)>0$ in every dimension.
\end{proof}

\end{document}